\documentclass[a4paper,reqno]{amsart}

\usepackage[a4paper,width={16cm},left=2.5cm,bottom=3cm, top=3cm, marginparwidth=50pt]{geometry}
\usepackage{enumerate}

\usepackage[english]{babel}

\usepackage[dvipsnames]{xcolor}
\usepackage{amsmath}
\usepackage{amssymb}
\usepackage{amsfonts}
\usepackage{amsthm,graphicx}

\usepackage{hyperref}

\usepackage{mathtools}
\numberwithin{equation}{section}
\newtheorem{theorem}{Theorem}[section]
\newtheorem{lemma}[theorem]{Lemma}
\newtheorem{corollary}[theorem]{Corollary}

\newtheorem{proposition}[theorem]{Proposition}
\theoremstyle{definition}

\newtheorem{example}[theorem]{Example}
\newtheorem{remark}[theorem]{Remark}

\newcommand{\mc}{\mathcal}
\newcommand{\mb}{\mathbb}

\newcommand{\la}{\lambda}
\newcommand{\norm}[1]{\left\lVert#1\right\rVert}
\newcommand{\pd}[2]{\frac{\partial#1}{\partial#2}}

\newcommand{\R}{\mb{R}}
\newcommand{\C}{\mb{C}}
\newcommand{\N}{\mb{N}}

\newcommand{\e}{\varepsilon}

\newcommand{\ps}[3]{\left( #2, #3 \right)_{#1}}

\usepackage{subfig}
\usepackage{tikz}
\usepackage{xcolor}
\usepackage{pgfplots}
\pgfplotsset{compat=newest}
\usepackage{mathrsfs}

\begin{document}

\title[Aharonov-Bohm potentials with Neumann Boundary]
{Magnetic Neumann problems with Aharonov-Bohm potentials: boundary asymptotics of eigenvalues and splitting phenomena}

\author[V. Felli, P. Roychowdhury,  and G. Siclari]{Veronica Felli, Prasun Roychowdhury, and Giovanni Siclari}

\address{Veronica Felli 
  \newline \indent Dipartimento di Matematica e Applicazioni
  \newline \indent Universit\`a degli Studi di Milano–Bicocca
\newline\indent Via Cozzi 55, 20125 Milano, Italy}
\email{veronica.felli@unimib.it}

\address{Prasun Roychowdhury
  \newline \indent Graduate School of Advanced Science and Engineering
  \newline \indent Waseda University, 3-4-1 Okubo, Shinjuku-ku
  \newline\indent Tokyo 169-8555, Japan}
\email{w.iac25210@kurenai.waseda.jp}

\address{Giovanni Siclari 
  \newline \indent Centro di Ricerca Matematica Ennio De Giorgi
  \newline \indent Scuola Normale Superiore di Pisa
  \newline\indent Piazza dei Cavalieri 3, 56126 Pisa, Italy}
\email{giovanni.siclari@sns.it}

\date{February 16, 2026}

\begin{abstract}
  We study a planar magnetic Schr\"odinger operator with an
  Aharonov-Bohm vector potential, under Neumann boundary
  conditions. Through a gauge transformation, the corresponding
  eigenvalue problem can be formulated in terms of the Laplacian on a
  fractured domain, where the fracture lies along the segment
  connecting the pole to its projection on the boundary.  As the pole
  approaches the boundary, we prove that the eigenvalues converge to
  those of the Neumann Laplacian and the variation exhibits a
  logarithmic vanishing rate. In the case of multiple
  eigenvalues, when the pole approaches a fixed point of the boundary,
  we observe a splitting phenomenon, with the largest branch
  separating from the others.
\end{abstract}

\maketitle


\medskip 


\section{Introduction}\label{sec_introduction}

This paper deals with the spectral properties of the magnetic
Schr\"odinger operator $(i \nabla + A_a)^2$ in a planar domain
$\Omega$, subject to Neumann boundary conditions.  Here $A_a$ denotes
the Aharonov-Bohm vector potential with pole at $a=(a_1,a_2)\in \R^2$
and circulation $1/2$, defined as
\begin{equation}\label{eq:ABfield}
  A_a(x_1,x_2):=\frac12 \left( - \frac{x_2 - a_2}{(x_1 - a_1)^2 + (x_2
      - a_2)^2}, \frac{x_1 - a_1}{(x_1 - a_1)^2 + (x_2 -
      a_2)^2}\right),\quad (x_1,x_2)\in \R^2\setminus\{a\}.
\end{equation}
Let
\begin{equation}\label{eq:ass-Omega}
  \Omega\subset\R^2\quad\text{be a bounded, open, connected, and Lipschitz  set}
\end{equation}
and let 
\begin{equation}\label{hp_p}
 p\in L^\infty(\Omega)=L^\infty(\Omega,\R) \quad  \text{with}\quad
 p(x)\geq c>0 \text{ for a.e. }
 x \in \Omega,
\end{equation}
where $c\in(0,+\infty)$ is a positive constant.
For every $a\in \Omega$, we are interested in the eigenvalue problem
\begin{equation}\label{eq:eige_equation_a}\tag{$E_{\Omega,p}^a$}
  \begin{cases}
   (i\nabla + A_{a})^2 u = \lambda p u,  &\text{in }\Omega,\\
   (i\nabla + A_{a}) u\cdot \nu = 0, &\text{on }\partial \Omega,
 \end{cases}
\end{equation}
where $\nu$ denotes the outer unit normal vector to $\partial\Omega$.

The vector potential \eqref{eq:ABfield} generates the Aharonov-Bohm
magnetic field in $\mathbb{R}^2$ with pole at $a$ and circulation
$1/2$. This can be interpreted as the limit field produced by an
infinitely long, thin solenoid intersecting the $(x_1,x_2)$–plane
orthogonally at the point $a$, as the radius of the solenoid tends to
zero while the magnetic flux remains constant and equal to $1/2$ (see,
for instance, \cite{AT98} and \cite{AB59}). The case of a half-integer
circulation considered in \eqref{eq:ABfield} has received the greatest
attention in the literature, due to its particular mathematical
relevance and its connections with the problem of spectral minimal
partitions; see \cite{BNHHO2009} and \cite{NT2010}.

For Schrödinger operators of the form $(i\nabla + A_a)^2$, the
continuity of the eigenvalues with respect to the position of the pole
was established in \cite{BNNNT2014} for the case of a single pole, and
in \cite{L2015} for multiple poles, even when collisions
occur. Building on these results, several authors have investigated the
problem of determining the precise asymptotic behavior of the
eigenvalue variation induced by small perturbations of the pole
configuration. In the setting of a single moving pole with
half-integer circulation, \cite{BNNNT2014} first identified a
relationship between the convergence rate of the eigenvalues and the
number of nodal lines of the associated eigenfunction. Sharper
asymptotic expansions of simple eigenvalues were derived in
\cite{AFaharonov}, for the pole moving along the tangent to a nodal
line of the limit eigenfunction, and in \cite{AFsiam}, for the pole
moving along an arbitrary direction. The case of a pole approaching
the boundary was addressed in \cite{AFNN2017} and \cite{NMT-ab-bd},
while genericity issues were discussed in \cite{A2019}, \cite{AN2018},
and more recently in \cite{AFgen}.  The case of two colliding poles
was studied in \cite{AFHL,AFL2017,AFL2020}, whereas the coalescence
of many poles was treated in \cite{FNOS_multi} and \cite{FNS}.

In all the aforementioned papers, Dirichlet boundary conditions are
considered.  Such a choice of boundary conditions is particularly
relevant in the case of a pole approaching the boundary considered in
\cite{AFNN2017} and \cite{NMT-ab-bd}, as it is responsible for the
power-type vanishing order of the eigenvalue variation with respect to
the distance from the boundary. In the present study, we show instead
that, in the case of Neumann boundary conditions, when the pole
approaches the boundary, the typical behavior is logarithmic, in
contrast with what occurs in the Dirichlet setting.

While the spectral properties of the Aharonov-Bohm operator under
Dirichlet boundary conditions have been widely studied in the
literature, the case of Neumann conditions is less explored. In this
regard, we cite the contributions of \cite{CPS2022}, which studies
isoperimetric inequalities for the eigenvalues of
\eqref{eq:eige_equation_a},  and 
\cite{FLPS24}, dealing with Pólya-type inequalities on the disk, covering both
  Dirichlet and magnetic Neumann boundary conditions.
We also mention \cite{FH2019} for the study of ground-state energies
of the Neumann magnetic Laplacian, with both constant and non constant
magnetic fields, and \cite{CLPS2023,CS2018} for lower and upper bounds
for eigenvalues of magnetic Laplacians with magnetic Neumann boundary
conditions.

For every $a\in \Omega$, we can write a weak formulation of
\eqref{eq:eige_equation_a} in the functional space
$H^{1,a}(\Omega,\C)$, defined as the completion of
\begin{equation*}
  \{\phi \in H^1(\Omega,\C) \cap C^\infty(\Omega,\C): \phi\equiv 0
  \text{ in  a neighborhood of } a\}
\end{equation*}
with respect to the norm 
\begin{equation*}
\norm{w}_{H^{1,a}(\Omega,\C)}:=\bigg(\norm{w}_{L^2(\Omega,\C)}^2+
\norm{\nabla w}_{L^2(\Omega,\C^2)}^2+
\Big\|\tfrac{w}{|\cdot-a|}\Big\|_{L^2(\Omega,\C)}^2\bigg)^{\!\!1/2}.
\end{equation*}
We notice that
$H^{1,a}(\Omega,\C)=\left\{u\in H^1(\Omega,\C):\frac{u}{|\cdot-a|}\in
  L^2(\Omega,\C)\right\}$.  For every $a\in\ \Omega$, $\lambda\in\R$
is an eigenvalue of \eqref{eq:eige_equation_a} if there exists
$u\in H^{1,a}(\Omega,\C)\setminus\{0\}$ (called an eigenfunction) such
that
\begin{equation}\label{eq:weaksense}
  \int_\Omega (i\nabla u+A_{a} u)\cdot \overline{(i\nabla v+A_{a}
    v)}\,dx=\lambda\int_\Omega p u\overline{ v}\,dx \quad
  \text{for all }v\in H^{1,a}(\Omega,\C).
\end{equation}
By classical spectral theory, \eqref{eq:eige_equation_a} admits a sequence of real
diverging eigenvalues
\begin{equation*}
\lambda_1^{a}(\Omega,p) \leq \lambda_2^a (\Omega,p)\leq\ldots\leq \lambda_j^a (\Omega,p)\leq\ldots
\end{equation*}
(repeated according to their finite multiplicity).

We are interested in the dependence of the spectrum on the pole $a$ as
it moves through the domain towards the boundary.  As a first result,
in the present paper we prove that, as $a$ approaches $\partial \Omega$, the eigenvalues of
\eqref{eq:eige_equation_a} converge to those of the  Neumann Laplacian
on $\Omega$ weighted by $p$, i.e.  to the eigenvalues of
\begin{equation}\label{eq:eige_lapla}\tag{$E_{\Omega,p}$}
  \begin{cases}
   -\Delta  u = \lambda p u,  &\text{in }\Omega,\\
   \frac{\partial u}{\partial\nu}= 0, &\text{on }\partial \Omega,
 \end{cases}
\end{equation}
which are known to form an increasing real sequence (with repetitions
in accordance with multiplicities)
\begin{equation*}
    \lambda_1 (\Omega,p) \leq \lambda_2 (\Omega,p)\leq\ldots\leq \lambda_j (\Omega,p)\leq\ldots
\end{equation*}
diverging to $+\infty$.

\begin{proposition}\label{prop_spectral_stability}
Under assumptions \eqref{eq:ass-Omega}--\eqref{hp_p}, for every $a\in\Omega$,
    let $\{\lambda_j^a(\Omega,p)\}_{j\geq1}$ and $\{\lambda_j(\Omega,p)\}_{j\geq1}$ be the
    eigenvalues of \eqref{eq:eige_equation_a} and
    \eqref{eq:eige_lapla}, respectively. Then, for every
    $k \in \mathbb{N} \setminus \{0\}$,
\begin{equation}\label{limit_eigen_conv}
  \lambda_k^a(\Omega,p)\to \lambda_k(\Omega,p)     \quad\text{as }
  \mathop{\rm dist}(a,\partial\Omega) \to0. 
\end{equation}
\end{proposition}

After establishing spectral stability, we turn to the problem of
quantifying it, i.e., of estimating the vanishing rate of the eigenvalue variation.
A first rough estimate of the eigenvalue variation, which holds for
any eigenbranch and regardless of the direction along which the pole
approaches the boundary, is provided by the following result.
\begin{theorem}\label{t:rough_estimate}
  Under assumptions \eqref{eq:ass-Omega}--\eqref{hp_p}, for every
  $a\in\Omega$, let $\{\lambda_j^a(\Omega,p)\}_{j\geq1}$ and
  $\{\lambda_j(\Omega,p)\}_{j\geq1}$ be the eigenvalues of
  \eqref{eq:eige_equation_a} and \eqref{eq:eige_lapla},
  respectively. Then, for every $k \in \mathbb{N} \setminus \{0\}$,
\begin{equation*}
  |\lambda_k^a(\Omega,p)- \lambda_k(\Omega,p)|=
  O\bigg(\frac1{\sqrt{|\log (\mathop{\rm
        dist}(a,\partial\Omega))|}}\bigg)
  \quad\text{as }
  \mathop{\rm dist}(a,\partial\Omega) \to0. 
\end{equation*}
\end{theorem}

The proof of the above theorem is based on an asymptotic expansion of
the eigenbranches of \eqref{eq:eige_equation_a}, 
see Theorem \ref{theo_asymp_abstract}, obtained by applying
the \emph{Lemma on small eigenvalues} by Y. Colin de Verdi\`ere
\cite{ColindeV1986} to an equivalent problem derived via a gauge
transformation. This consists of an eigenvalue problem for the
Laplacian with jumping conditions across a crack lying along the segment
joining the moving pole with its projection on the boundary. We refer
to \cite{FNOS_multi} for an analogous formulation for the
Aharonov-Bohm operator with many poles of half-integer circulation. In
the spirit of \cite{FNOS_multi}, the coefficients governing the
expansion are related to a sort of weighted torsional rigidity of the
segment along which the pole is moving, obtained as the minimum of a
functional depending on the position of the pole and on the
eigenfunctions of the limit problem on the crack, see \eqref{eq:E_a}.

In the case of eigenbranches emanating from a simple eigenvalue, the
procedure described above gives more precise information. 
Let
$n \in \N\setminus\{0\}$ be such that
\begin{equation}\label{hp_la_0n_simple}
\lambda_n(\Omega,p) \text{ is simple}
\end{equation}
as an eigenvalue of \eqref{eq:eige_lapla}.  
Let $u_n$ be
an eigenfunction associated to $\lambda_n(\Omega,p)$ normalized so
that 
$\int_\Omega pu_n^2\,dx=1$. We consider the quantity
\begin{equation*}
  \mathcal E^{a,u_n}_{\Omega,p}:=
  \min\left\{
\begin{array}{ll}
   &\hskip-10pt {\displaystyle{\frac12 \int_{\Omega\setminus S_a^\Omega}\left(|\nabla
      v|^2+pv^2\right)\,dx-
    \int_{\Omega\setminus S_a^{\Omega}}\nabla u_n\cdot\nabla
    v\,dx+\lambda_n(\Omega,p)\int_{\Omega}pu_nv\,dx}}:\\[15pt]
& \hskip-10pt   v\in H^1(\Omega\setminus S_a^\Omega), \
\gamma_{a,\Omega}^+(v-u_n)+\gamma_{a,\Omega}^-(v-u_n)=0 \text{ on
                                                         }S_a^\Omega
\end{array}
\right\},
\end{equation*}
see \eqref{eq:E_a}, where $S_a^\Omega$ is the segment joining the pole
$a$ with its projection on $\partial\Omega$ as defined in
\eqref{eq:def-segment}, and $\gamma_{a,\Omega}^+$ and
$\gamma_{a,\Omega}^-$ denote the trace operators on the two sides of
$S_a^\Omega$ , see \eqref{eq:trace}. Let $V^{a,u_n}_{\Omega,p}$ be the
unique minimizer attaining $\mathcal E^{a,u_n}_{\Omega,p}$, see
\eqref{eq:def_Va} and \eqref{eq:eq_Va}.
 
 \begin{theorem}\label{th:simple}
   Let assumptions \eqref{eq:ass-Omega}, \eqref{hp_p}, and
   \eqref{hp_la_0n_simple} be satisfied. Then
\begin{equation}\label{eq:expa-simple}
  \lambda_n^a(\Omega,p) -\lambda_n(\Omega,p)=2
\mathcal E^{a,u_n}_{\Omega,p}+O\left(\|V^{a,u_n}_{\Omega,p}\|_{L^2(\Omega)}^2\right)
  \quad \text{as  }\mathop{\rm
    dist}(a,\partial\Omega)\to 0.
\end{equation}
\end{theorem}
We refer to Corollary \ref{cor:simple} for a more complete statement
of the above result, also dealing with the behaviour of
eigenfunctions.  We mention that, in the simple case, expansion
\eqref{eq:expa-simple} can be obtained using an alternative and
equivalent approach due to \cite{BS_eigen}, where a broad class of
abstract perturbative problems are treated.

Our final result addresses the case in which the pole approaches a
fixed point on the boundary, under stronger regularity assumptions on
the domain $\Omega$. It provides an asymptotic expansion of the
eigenbranches emanating from a possibly multiple eigenvalue, and
yields a splitting property, consisting in the separation of the
largest branch from the others, with a consequent decrease in the
multiplicity of all eigenvalues.

\begin{theorem}\label{t:toward-a-fixed-point}
  Let $\Omega\subset \R^2$ be a planar Jordan domain with
  $C^{1,\alpha}$ boundary, for some $\alpha\in(0,1)$. Let $p$ satisfy
  \eqref{hp_p}.  Let $n,m \in \N\setminus\{0\}$ be such that
  $\lambda_n(\Omega,p)$ has multiplicity $m$ as an eigenvalue of
  \eqref{eq:eige_lapla}.  Let $\{u_{n}, u_{n+1},\dots, u_{n+m-1}\}$ be
  a basis of the associated $m$-dimensional eigenspace such that
  $\int_\Omega p\,u_{n+i-1}^2dx=1$ for all $1\leq i\leq m$ and
  $\int_\Omega p\,u_{n+i-1}u_{n+j-1}dx=0$ if $i\neq j$.
Let  $a_0\in \partial\Omega$. Then 
\begin{equation*}
\lambda_{n+m-1}^{a}(\Omega,p) =\lambda_n (\Omega,p)+
\frac{\pi\big(\sum_{i=1}^mu_{n+i-1}^2(a_0)\big)}{|\log \mathop{\rm
    dist}(a,\partial\Omega)|}+o\left(\frac{1}{|\log \mathop{\rm
    dist}(a,\partial\Omega)|}\right)\quad\text{as }a\to a_0.
\end{equation*}
Furthermore, if $m>1$, 
\begin{equation*}
\lambda_{n+i-1}^{a}(\Omega,p) =\lambda_n (\Omega,p)+o\left(\frac{1}{|\log \mathop{\rm
    dist}(a,\partial\Omega)|}\right)\quad\text{as }a\to a_0,
\end{equation*}
for every $1\leq i\leq m-1$.
\end{theorem}

In particular, Theorem~\ref{t:toward-a-fixed-point} shows that, if
$a_0 \in \partial \Omega$ is such that at least one eigenfunction of
problem \eqref{eq:eige_lapla} associated with the eigenvalue
$\lambda_n (\Omega,p)$ does not vanish at $a_0$ (and this holds at all
boundary points, with at most finitely many exceptions, see Remark
\ref{rem:cf}), then, for $a$ sufficiently close to $a_0$, the
eigenvalue $\lambda_{n+m-1}^{a}(\Omega,p)$ is simple and the
eigenvalues $\{\lambda_{n+i-1}^{a}(\Omega,p)\}_{i<m}$ have
multiplicity strictly smaller than $m$. Thus,
Theorem~\ref{t:toward-a-fixed-point} describes a splitting phenomenon
for a multiple eigenvalue, in the spirit of the analogous result
proved in \cite{flucher} for the Dirichlet Laplacian in domains with
small holes; see also \cite{FLO_2} for the Neumann Laplacian in
perforated sets.

The proof of Theorem \ref{t:toward-a-fixed-point} relies on a detailed
analysis of the asymptotic behavior of the quantities
$\mathcal E^{a,u_{n+i-1}}_{\Omega,p}$. It is at this stage that the
main technical difficulties of the present paper arise, as the methods
developed for the Dirichlet case do not seem to be applicable in the
Neumann setting.  The main techniques adopted in the literature to
derive asymptotics of Aharonov--Bohm eigenvalues under Dirichlet
boundary conditions are blow-up arguments for eigenfunctions or for
torsion-like functions (see, e.g.,
\cite{AFaharonov,AFsiam,AFNN2017,FNOS_multi}), or, alternatively, the
explicit computation of torsion-like functions or capacitary potentials, carried
out through a change to elliptic coordinates (see, e.g.,
\cite{AFHL,FNOS_multi}).  As far as blow-up analysis is concerned, it
is effective when Hardy-type functional inequalities are available, as
they allow one to identify a concrete functional space in which the
limit of the blow-up sequence lives. This occurs, for instance, in the
case - treated in \cite{AFaharonov,AFsiam,FNOS_multi} - of a single
pole or of an odd number of poles with circulation $1/2$ colliding at
an interior point of the domain, or in the case of a pole approaching
a boundary where Dirichlet conditions are imposed, as in
\cite{AFNN2017}. However, this approach is not effective for an even
number of poles colliding in the interior of the domain, as well as for
a pole approaching a boundary with Neumann conditions, due to the lack
of a concrete functional characterization of the blow-up profile.  In
the case of two poles colliding in the interior under Dirichlet
boundary conditions, an alternative approach based on the use of
elliptic coordinates was developed in \cite{AFHL}; see also
\cite{FNOS_multi}. However, for a Neumann problem, the presence of a
zero-order term in the operator - required to recover the coercivity of
the associated quadratic form \eqref{eq:qf}- causes a singular
potential to appear when rewriting the problem in elliptic
coordinates, leading to considerably greater difficulties compared to
the Dirichlet case. In this work, we address this new and challenging
difficulty by introducing an intermediate auxiliary problem, see
\eqref{prob_Wj}, for which the potential arising from the change to
elliptic coordinates are more easily manageable, and then showing that
this problem provides a good approximation of the original one for our
purposes, see Proposition \ref{prop_WZ}.

The paper is organized as follows. In section \ref{sec_preliminaries}
we present some preliminary results, including an equivalent
formulation of problem \eqref{eq:eige_equation_a} in terms of the
Laplacian on a fractured domain obtained via a gauge transformation,
as well as a Mosco convergence result for Sobolev spaces in domains
with a crack along the segment connecting the pole to its projection
on the boundary, see Proposition \ref{prop_mosco}. Section
\ref{sec:spectral-stability} is devoted to the proof of the spectral
stability stated in Proposition \ref{prop_spectral_stability}.  In
section \ref{sec_eigen_expansion} we prove an asymptotic expansion of
the eigenvalue variation as the pole approaches the boundary, see
Theorem \ref{theo_asymp_abstract}, from which we then deduce Theorem
\ref{t:rough_estimate} and, in the simple case, Theorem
\ref{th:simple}. In section \ref{sec_asymptoic_precise} we prove
Theorem \ref{t:toward-a-fixed-point}.  For the sake of completeness,
in the final appendix \ref{sec_appendix}, we collect some known results used in the paper:
the \emph{Lemma on small eigenvalues} by Y. Colin de Verdi\`ere and
the well-known \emph{diamagnetic inequality}.

\section{Preliminaries}\label{sec_preliminaries}

\subsection{An equivalent formulation via gauge transformation}
For every $a\in\Omega$, we denote as
\begin{equation*}
  d_a^\Omega=\mathop{\rm dist}(a,\partial\Omega)  
\end{equation*}
the distance of $a$ from the boundary $\partial\Omega$. We observe
that, for every $a\in\Omega$, there exists some point
$P_a^\Omega\in \partial\Omega$ (possibly not unique) such that
\begin{equation*}
    d_a^\Omega=|a-P_a^\Omega|.
\end{equation*}
Let 
\begin{equation}\label{eq:def-segment}
    S_a^\Omega=\{ta+(1-t)P_a:t\in[0,1]\}
\end{equation}
be the segment with endpoints $a$ and $P_a^\Omega$ and let 
$\omega_a^\Omega\in [0,2\pi)$ be its slope, so that 
\begin{equation*}
    P_a^\Omega=a+d_a^\Omega(\cos\omega_a^\Omega,\sin\omega_a^\Omega),
\end{equation*}
see Figure \ref{f:notation}. It can be easily verified that 
\begin{equation*}
  S_a^\Omega\setminus\{P_a^\Omega\}\subset\Omega.
\end{equation*}
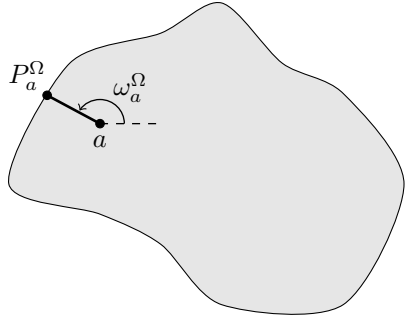
\begin{figure}[ht]
\centering
\begin{tikzpicture}[scale=0.8]
%\draw[gray,very thin] (-4,-3) grid (4,4);
 \filldraw[fill=gray!20]  plot [smooth cycle]
 coordinates {(3,0) (2,1.5) (1,2) (0,3) (-1,2.5) (-2.5,2) (-3.5,0)
   (-2,-0.5) (-1,-1) (0,-2) (2,-2)};
 \draw [fill=black] (-2,1) circle (2pt);
 \draw [fill=black] (-2.87,1.47) circle (2pt);
   \draw [line width=1pt](-2,1)--(-2.87,1.47);
\draw [line width=0.5pt,dashed](-2,1)--(-1,1);
 \draw[color=black] (-3.2,1.8) node {$P_a^\Omega$};
 \draw[color=black] (-2,0.7) node {$a$};
 \draw[->,line width=0.5pt] (-1.6,1) arc (0:147:0.4);
 \draw[color=black] (-1.5,1.6) node {$\omega_a^\Omega$};
 \end{tikzpicture}
 \caption{Definition of $P_a^\Omega$ and $\omega_a^\Omega$.}
 \label{f:notation}
\end{figure}
In the spirit of \cite{FNOS_multi}, we perform a gauge transformation,
which makes \eqref{eq:eige_equation_a} equivalent to an eigenvalue
problem for the Neumann Laplacian in $\Omega$ with a straight crack
along the segment $S_a^\Omega$.  For any $a\in \Omega$, we let
\begin{equation*}
\nu_a^\Omega:=(-\sin\omega_a^\Omega,\cos\omega_a^\Omega)
\end{equation*}
and  consider the half-planes
\begin{equation*}
    \Pi_{a,\Omega}^+=\{x\in \R^2:x\cdot \nu_a^\Omega>0\},
    \quad \Pi_{a,\Omega}^-=\{x\in \R^2:x\cdot \nu_a^\Omega<0\}.
\end{equation*}
By embedding theorems for fractional Sobolev spaces in dimension $1$
and classical trace results, for every $a\in\Omega$ and
$q \in [2, +\infty)$ there exist continuous trace operators
\begin{equation}\label{eq:trace}
\gamma_{a,\Omega}^+:H^1(\Pi_{a,\Omega}^+)\to L^q(\Sigma_a^\Omega), 
\quad
\gamma_{a,\Omega}^-:H^1(\Pi_{a,\Omega}^-)\to L^q(\Sigma_a^\Omega), 
\end{equation}
where
$\Sigma_a^\Omega=\partial \Pi_{a,\Omega}^+=\partial \Pi_{a,\Omega}^-$.
For a function $u\in H^1(\Omega\setminus \Sigma_a^\Omega)$, we will
simply write $\gamma_{a,\Omega}^+(u)$ and $\gamma_{a,\Omega}^-(u)$ to
indicate $\gamma_{a,\Omega}^+(u\big|_{\Pi_{a,\Omega}^+})$ and
$\gamma_{a,\Omega}^-(u\big|_{\Pi_{a,\Omega}^-})$, respectively.

Since the Aharonov-Bohm vector field \eqref{eq:ABfield} is
irrotational, it is found to be the gradient of some scalar potential
function in simply connected domains, e.g., in the complement of the
straight half-line
$\Gamma_a^\Omega=\{t(\cos\omega_a^\Omega,\sin\omega_a^\Omega):t\geq0\}$
emanating from the pole $a$ with slope $\omega_a^\Omega$. More
precisely, we consider the function
\begin{equation*}
\Theta_a^\Omega: \R^2\setminus\{a\}\to \R,\quad 
\Theta_a^\Omega(a+(r\cos t,r\sin t))=
    \begin{cases}
       \frac 12 t ,&\text{if }t\in[0,\omega_a^\Omega),\\[3pt]
       \frac12 t-\pi,&\text{if }t\in[\omega_a^\Omega, 2\pi),
    \end{cases}
  \end{equation*}
and observe that $\Theta_a^\Omega\in C^\infty(\R^2\setminus\Gamma_a^\Omega)$ and  
$\nabla \Theta_a^\Omega$ can be smoothly extended to be in
$C^\infty(\R^2\setminus\{a\})$ with
\begin{equation*}
  \nabla\Theta_a^\Omega=A_a\quad \text{in }\R^2\setminus\{a\}.
\end{equation*}
In view of \cite[Lemma 3.3]{HHOO99}, every
eigenspace of \eqref{eq:eige_equation_a} admits a basis
made of eigenfunctions $\varphi$ satisfying 
\begin{equation}\label{eq:propertyP}\tag{$K$}
 e^{-i\frac t2}\varphi(a+(r\cos t, r\sin t))\text{ is a real-valued function};
\end{equation}
see also \cite{BNH2017}.
Furthermore,
each eigenfunction
$u\in H^{1,a}(\Omega,\C)\setminus\{0\}$ of \eqref{eq:eige_equation_a}
 satisfying \eqref{eq:propertyP} corresponds, via the gauge transformation
\begin{equation}\label{eq:gauge}
  u(x) \mapsto v(x):= e^{-i\Theta_a^\Omega(x)}u(x) , \quad x\in
  \Omega\setminus\Gamma_a^\Omega,
\end{equation}
to a real valued function $v\in\mathcal H_a(\Omega)\setminus\{0\}$ solving
the following problem 
\begin{equation}\label{eq:eige_a}\tag{$P^a_{\Omega,p}$}
    \begin{cases}
        -\Delta v = \lambda p v, &\text{in }\Omega\setminus S_a^\Omega,\\
        \frac{\partial v}{\partial \nu}=0, &\text{on }\partial\Omega,\\
        \gamma_{a,\Omega}^+(v) + \gamma_{a,\Omega}^+(v)=0, &\text{on }S_a^\Omega,\\
        \gamma_{a,\Omega}^+(\nabla v\cdot \nu_a^\Omega) + \gamma_{a,\Omega}^+(\nabla
        v\cdot\nu_a^\Omega)=0,
        &\text{on }S_a^\Omega,
    \end{cases}
  \end{equation}
  where the functional space $\mathcal H_a(\Omega)$ is defined as
\begin{equation}\label{eq:defHa}
  \mathcal H_a(\Omega)=\{v\in H^1(\Omega\setminus S_a^\Omega):
  \gamma_{a,\Omega}^+(v)+\gamma_{a,\Omega}^-(v)=0
\text{ on }S_a^\Omega\}.
\end{equation}
In view of the Hardy type inequality proved in \cite{LW} (see also
\cite[Lemma 3.1, Remark 3.2]{FFT}), the embedding
$H^{1,a}(\Omega,\C) \hookrightarrow H^{1}(\Omega,\C)$ is continuous;
since $\Omega$ is a Lipschitz domain, it follows that the embedding
$H^{1,a}(\Omega,\C) \hookrightarrow L^2(\Omega,\C)$ is compact, so
that the gauge correspondence \eqref{eq:gauge} implies that, for every
$a\in \Omega$,
\begin{equation}\label{eq:compact}
  \text{the embedding }\mathcal H_a(\Omega) \hookrightarrow \hookrightarrow
  L^2(\Omega)\text{ is compact}.
\end{equation}
Problem \eqref{eq:eige_a} admits the following weak formulation:
$v\in H^1(\Omega\setminus S_a^\Omega)$ weakly solves \eqref{eq:eige_a} if
\begin{equation*}
  v\in\mathcal H_a(\Omega)\quad\text{and}\quad
  \int_{\Omega\setminus S_a^\Omega}\nabla v\cdot
  \nabla w\,dx=\lambda\int_\Omega p vw\,dx\quad\text{for all }w\in 
  \mathcal H_a(\Omega).
\end{equation*}
Hence problems \eqref{eq:eige_equation_a} and \eqref{eq:eige_a} are
spectrally equivalent: indeed, they have the same eigenvalues and 
the eigenfunctions correspond to each other through the gauge 
transformation \eqref{eq:gauge}, 
once those of \eqref{eq:eige_equation_a} are chosen to
satisfy \eqref{eq:propertyP}.

\subsection{Mosco type convergence of the complements of segments}

Let
\begin{equation}\label{eq:a_n}
  \{a_k\}_{k\in\N}\subset\Omega \text{ be such that
    $\lim_{k\to\infty}\mathop{\rm
  dist}(a_k,\partial\Omega)=0$},
\end{equation}
 i.e. $\{a_k\}_{n\in\N}$ is a sequence of points in
 $\Omega$ approaching $\partial\Omega$. 
Letting
\begin{equation}\label{eq:K}
  K=\bigcup_{k\in\N}S_{a_k}^\Omega,
\end{equation}
we observe that the $2$-dimensional Lebesgue measure of $K$ is zero;
furthermore, it can be easily verified that $\Omega\setminus K$ is an
open set.

\begin{proposition}\label{prop_mosco}
  Let $\{a_k\}_{k\in\N}$ and $K$ be as in \eqref{eq:a_n} and
  \eqref{eq:K}, respectively. If
  $\{v_k\}_k\subset H^1(\Omega \setminus K)$ and
  $v\in H^1(\Omega \setminus K)$ are such that
  $v_k\in H^1(\Omega \setminus S_{a_k}^\Omega)$ for every $k\in \N$ and
  $v_k\rightharpoonup v$ weakly in $H^1(\Omega \setminus K)$ as
  $k\to\infty$, then
\begin{equation}\label{eq:mosco}
  v\in H^1(\Omega).
\end{equation}
In addition, by assuming $v_k\in \mathcal H_{a_k}(\Omega)$ for every
$k\in \N$ along with the previous hypotheses, we have
\begin{equation}\label{eq:comp}
  v_k\to v \text{ strongly in }L^2(\Omega)\text{  as
  $k\to\infty$}.
\end{equation}
\end{proposition}
\begin{proof}
Since $v\in H^1(\Omega \setminus K)$, there exists
    $\mathbf{f}=(f_1,f_2)\in L^2(\Omega,\R^2)$ such that $\nabla v=\mathbf{f}$
    in $\mathcal D'(\Omega\setminus K)$, i.e.
    \begin{equation*}
      \int_{\Omega}v\frac{\partial \varphi}{\partial
        x_i}\,dx=-\int_\Omega f_i\varphi\,dx\quad\text{for all
      }\varphi\in C^\infty_{\rm c}(\Omega\setminus K),\quad i=1,2.
    \end{equation*}
    Furthermore, for every $k\in\N$, since
    $v_k\in H^1(\Omega \setminus S_{a_k}^\Omega)$, there exists
    $\mathbf{f}_k=(f_1^k,f_2^k)\in L^2(\Omega,\R^2)$ such that
    $\nabla v_k=\mathbf{f}_k$ in
    $\mathcal D'(\Omega\setminus S_{a_k}^\Omega)$, i.e.
    \begin{equation*}
      \int_{\Omega}v_k\frac{\partial \varphi}{\partial
        x_i}\,dx=-\int_\Omega f_i^k\varphi\,dx\quad\text{for all
      }\varphi\in C^\infty_{\rm c}(\Omega\setminus S_{a_k}^\Omega),\quad i=1,2.
    \end{equation*}
Let $\varphi\in C^\infty_{\rm c}(\Omega)$. Since $\lim_{k\to\infty}\mathop{\rm
  dist}(a_k,\partial\Omega)=0$, there exists $\bar k$ such that
$\varphi\in C^\infty_{\rm c}(\Omega\setminus S_{a_k}^\Omega)$ for all $k\geq
\bar k$. Then
    \begin{equation*}
      \int_{\Omega}v_k\frac{\partial \varphi}{\partial
        x_i}\,dx=-\int_\Omega f_i^k\varphi\,dx.
    \end{equation*}
    for all $k\geq \bar k$. By the weak convergence
    $v_k\rightharpoonup v$ in $H^1(\Omega \setminus K)$ it follows
    that $v_k\rightharpoonup v$ weakly in $L^2(\Omega)$ and
    $\mathbf{f}_k\rightharpoonup \mathbf{f}$ weakly in
    $L^2(\Omega,\R^2)$, so that, passing to the limit as $k\to\infty$
    in the above identity, we obtain
 \begin{equation*}
      \int_{\Omega}v\frac{\partial \varphi}{\partial
        x_i}\,dx=-\int_\Omega f_i\varphi\,dx
      \quad\text{for all
      }\varphi\in C^\infty_{\rm c}(\Omega),
    \end{equation*}
    thus implying that $\nabla v=\mathbf{f}$ in $\mathcal
    D'(\Omega)$ and proving \eqref{eq:mosco}.

    To prove \eqref{eq:comp}, we make the further assumption that
    $v_k\in \mathcal H_{a_k}(\Omega)$ for every $k\in \N$.  We first
    observe that there exists a subsequence $a_{k_j}$ such that
    $a_{k_j}\to\bar a$ as $j\to\infty$, for some point
    $\bar a\in\partial \Omega$. For every sufficiently small $\e>0$,
    $\Omega\setminus D(\bar a,\e)$ contains only a finite number of
    segments $S_{a_{k_j}}^\Omega$, hence the embedding
    $H^1((\Omega\setminus \bar K)\setminus D(\bar
    a,\e))\hookrightarrow \hookrightarrow L^2(\Omega\setminus D(\bar
    a,\e))$ is compact, where
    $\bar K=\bigcup_{j\in\N}S_{a_{k_j}}^\Omega\subseteq K$. Hence
    $v_{k_j}\to v$ strongly in $L^2(\Omega\setminus D(\bar a,\e))$ for
    every $\e>0$ small. By a diagonal argument, it easily follows that,
    up to extracting a further subsequence,
    \begin{equation}\label{eq:ae-conv}
      v_{k_j}\to v\text{ a.e. in $\Omega$}. 
    \end{equation}
For every $j$, let 
\begin{equation*}
    g_j:= e^{i\Theta^\Omega_{a_{k_j}}(x)}v_{k_j}(x).
\end{equation*}
Since $v_k\in \mathcal H_{a_k}(\Omega)$, we have $g_j\in
H^{1,a_k}(\Omega,\C)$. Moreover,
\begin{equation*}
  (i\nabla +A_{a_{k_j}})g_j=i e^{i\Theta^\Omega_{a_{k_j}}}\nabla v_{k_j}.
\end{equation*}
Hence, in view of the 
diamagnetic inequality stated in Lemma \ref{l:diam}, $|g_j|\in
H^1(\Omega)$ for all $j$ and
\begin{align*}
  \| v_{k_j}\|_{H^1(\Omega\setminus K)}^2&=
  \int_{\Omega\setminus S^\Omega_{a_{k_j}}}\left(|\nabla  v_{k_j}|^2+
                                           v_{k_j}^2\right)\,dx\\
  &=
  \int_{\Omega}\left(|(i\nabla
    +A_{a_{k_j}})g_j|^2+|g_j|^2\right)\,dx\geq \int_{\Omega}\left(\big|\nabla|g_j|\big|^2+|g_j|^2\right)\,dx.
\end{align*}
It follows that the sequence $\{|g_j|\}_j$ is bounded in
$H^1(\Omega)$. Therefore, by the compactness of the embedding
$H^1(\Omega)\hookrightarrow \hookrightarrow L^2(\Omega)$, there exist
a subsequence of $\{|g_j|\}_j$ , still denoted as $\{|g_j|\}_j$, and
some $G\in L^2(\Omega)$ such that $|g_j|\to G$ strongly in
$L^2(\Omega)$ and a.e.  in $\Omega$. Since $|g_j|=|v_{k_j}|$ a.e.,
\eqref{eq:ae-conv} implies that $G=|v|$ a.e. in $\Omega$. Hence
$|g_j|\to |v|$ strongly in $L^2(\Omega)$, so that
\begin{equation*}
  \int_\Omega|g_j|^2\,dx=\int_\Omega |v_{k_j}|^2\,dx\to
  \int_\Omega |v|^2\,dx\quad \text{as }j\to\infty.
\end{equation*}
The convergence of $L^2$-norms established above, together with the
weak convergence $v_{k_j}\rightharpoonup v$ in $L^2(\Omega)$, implies
that\begin{equation*} v_{k_j}\to v\text{ strongly in $L^2(\Omega)$}.
    \end{equation*}
    The conclusion \eqref{eq:comp} then follows from the Urysohn
    subsequence principle.
  \end{proof}

  \begin{remark}
    We observe that, in general, $\Omega\setminus K$ is not regular
    enough to ensure the compactness of the embedding
    $H^1(\Omega\setminus K) \hookrightarrow L^2(\Omega\setminus
    K)$. Therefore, weak convergence in $H^1(\Omega\setminus K)$ does
    not always directly imply strong convergence in
    $L^2(\Omega\setminus K)$. In Proposition \ref{prop_mosco}, it was
    possible to deduce strong $L^2$-convergence for $\{v_k\}$
    exploiting the additional condition that each $v_k$ jumps across
    only one segment.
\end{remark}

Arguing as in \cite[Lemma 3.4]{FNOS_multi}, we can obtain the following
result.
\begin{proposition}\label{prop_density_0}
  For any $b \in \partial \Omega$, the set
  $H_{0,b}^1(\Omega):=\{v\in H^1(\Omega):v\equiv0\text{ in a
    neighborhood of } b\}$ is dense in $H^1(\Omega)$.
\end{proposition}

\section{Spectral stability}\label{sec:spectral-stability}
This section is devoted to the proof of the stability result contained
in Proposition \ref{prop_spectral_stability}.
\begin{proof}[Proof of Proposition \ref{prop_spectral_stability}]
Let $\Omega$ and $p$ be fixed as in
\eqref{eq:ass-Omega}--\eqref{hp_p}.
For every $a\in\Omega$, let 
\begin{equation*}
R_{a}:L^2(\Omega) \to  \mathcal H_a(\Omega)\subset L^2(\Omega)  
\end{equation*}
be the linear operator defined as follows: for every
$f \in L^2(\Omega)$, $R_{a} f$ is the unique weak solution to the
problem
\begin{equation*}
\begin{cases}
-\Delta v +pv = pf, &\text{in }\Omega\setminus S_{a}^\Omega,\\
\frac{\partial v}{\partial \nu}=0, &\text{on }\partial\Omega,\\
\gamma_{a,\Omega}^+(v) + \gamma_{a,\Omega}^+(v)=0, &\text{on }S_{a}^\Omega,\\
\gamma_{a,\Omega}^+(\nabla v\cdot \nu_{a}^\Omega) + \gamma_{a,\Omega}^+(\nabla
v\cdot\nu_{a}^\Omega)=0,
&\text{on }S_{a}^\Omega,
\end{cases}
\end{equation*}
in the sense that $v =R_{a} f\in  \mathcal H_{a}(\Omega)$ and 
\begin{equation}\label{eq:eq-Raf}
  \int_{\Omega \setminus S_{a}^\Omega} \big(\nabla v \cdot \nabla w +pvw\big)
  \, dx = \int_{\Omega}  pf w \, dx \quad
  \text{for all } w \in  \mathcal H_{a}(\Omega).
\end{equation}
In the same way, we consider the linear operator 
\begin{equation*}
R:L^2(\Omega) \to  H^1(\Omega)\subset L^2(\Omega)  
\end{equation*}
defined as follows: for every
$f \in L^2(\Omega)$, $R f$ is the unique weak solution to the
problem
\begin{equation*}
\begin{cases}
-\Delta v +pv = pf, &\text{in }\Omega,\\
\frac{\partial v}{\partial \nu}=0, &\text{on }\partial\Omega.
\end{cases}
\end{equation*}
The operator $R_{a}$ is linear and bounded as an operator from
$L^2(\Omega)$ into $ \mathcal H_{a}(\Omega)$.  Moreover, as an operator from
$L^2(\Omega)$ into $L^2(\Omega)$, $R_{a}$ is compact and 
self-adjoint
if $L^2(\Omega)$ is endowed with the scalar product 
\begin{equation}\label{eq:scal-p}
\ps{L^2(\Omega,p)}{u}{v}:=\int_{\Omega}p uv \, dx,
\end{equation}
which is equivalent to the standard one by \eqref{hp_p}. The norm
corresponding to the scalar product \eqref{eq:scal-p} is
\begin{equation*}
  \|u\|_{L^2(\Omega,p)}
  :=\sqrt{\int_{\Omega}p u^2 \, dx}.
\end{equation*}
Finally,
testing \eqref{eq:eq-Raf} with $w=R_af$, we easily obtain that, by
\eqref{hp_p},
  \begin{equation}\label{eq:stima-Ra}
    \|R_a f\|_{H^1(\Omega\setminus S_a^\Omega)}\leq
 C_p   \|f\|_{L^2(\Omega)}\quad\text{for every }f\in L^2(\Omega),
\end{equation}
for some $C_p>0$ depending only on the weight $p$.

This implies in particular, that
$\norm{R_{a}}_{\mc{L}(L^2(\Omega))}\le C_p$, where
$\norm{R_{a}}_{\mc{L}(L^2(\Omega))}$ is the operator norm.

Let $f \in L^2(\Omega)$.  Let $\{a_k\}_{k\in\N}\subset\Omega$ be
  such that $\lim_{k\to\infty}\mathop{\rm dist}(a_k,\partial\Omega)=0$
  and $K=\bigcup_{k\in\N}S_{a_k}^\Omega$.  By \eqref{eq:stima-Ra},
$\{R_{a_k}f\}_{k\geq 1}$ is bounded in $H^1(\Omega\setminus
K)$. Hence, in view of Proposition \ref{prop_mosco}, there exist
$\bar a\in\partial\Omega$, $w \in H^1(\Omega)$, and a subsequence,
which we still label with $a_k$, such that
$\lim_{k\to\infty}a_k=\bar a$, $R_{a_k} f \rightharpoonup w$ weakly in
$H^1(\Omega\setminus K)$, and $R_{a_k} f \to w$ strongly in
$L^2(\Omega)$ as $k \to \infty$.  For any
$g \in H_{0,\bar a}^1(\Omega)$, we have $g\in \mathcal H_{a_k}(\Omega)$ for
sufficiently large $k$, so that, in view of \eqref{hp_p},
\begin{equation*}
  \int_{\Omega} pfg \, dx=
  \lim_{k\to \infty}\int_{\Omega\setminus S_{a_k}^\Omega} \big(\nabla (R_{a_k} f)
  \cdot \nabla g + p(R_{a_k}f) g\big) \, dx
  =\int_{\Omega} (\nabla w \cdot \nabla g +pw g ) \, dx.
\end{equation*}
We conclude that $w= R f$ by Proposition \ref{prop_density_0}.  Since
the limit of $\{R_{a_k}f\}$ in $L^2(\Omega)$ depends neither on the
sequence $\{a_k\}$ nor on the extracted subsequence, by the Urysohn
subsequence principle we conclude that, for every $f\in L^2(\Omega)$,
\begin{equation}\label{proof_spectral_stability_1}
  R_{a}f\to R f\quad\text{in }L^2(\Omega)\text{ as }
  \mathop{\rm dist}(a,\partial\Omega)\to 0.
\end{equation}
By compactness of the operator $R_{a}- R$, for every $a\in\Omega$
there exists $f_{a} \in L^2(\Omega)$ such that
$\norm{f_{a}}_{L^2(\Omega)}=1$ and
\begin{equation}\label{eq:normRa-R}
  \norm{R_{a}- R}_{\mc{L}(L^2(\Omega))}=
  \norm{R_{a}f_{a}- R f_{a}}_{ L^2(\Omega)}.
\end{equation}
Let $\{a_k\}_{k\in\N}\subset\Omega$ be such that
  $\lim_{k\to\infty}\mathop{\rm dist}(a_k,\partial\Omega)=0$ and
  $K=\bigcup_{k\in\N}S_{a_k}^\Omega$. Then there exists
$\bar a\in\partial\Omega$ and $f\in L^2(\Omega)$ such that, up to a
subsequence, $a_k\to\bar a$ and $f_{a_k} \rightharpoonup f$ weakly in
$L^2(\Omega)$ as $k \to \infty$. By compactness of $R$
\begin{equation}\label{eq:compR}
  R f_{a_k} \to R f \quad\text{as $k \to \infty$ in $L^2(\Omega)$}.
\end{equation}
In view of \eqref{eq:stima-Ra}, up to passing to a further
subsequence, there exists $w \in H^1(\Omega\setminus K)$ such that
$R_{a_k}f_{a_k} \rightharpoonup w$ weakly as $k \to \infty$ in
$H^1(\Omega \setminus K)$. By Proposition \ref{prop_mosco} we have
$w\in H^1(\Omega)$ and $R_{a_k}f_{a_k}\to w$ strongly in
$L^2(\Omega)$. Moreover, in view of the self-adjointness of $R_{a_k}$
and $R$ with respect to the scalar product \eqref{eq:scal-p}, and
\eqref{proof_spectral_stability_1}, for every $g \in L^2(\Omega)$ we
have
\begin{equation*}
\int_{\Omega} pwg \, dx=\lim_{k \to \infty}\int_{\Omega}  p(R_{a_k}
f_{a_k}) g  \, dx
=\lim_{k \to \infty}\int_{\Omega}  pf_{a_k} ( R_{a_k}g)  \, dx
=\int_{\Omega}  pf (R g)  \, dx=\int_{\Omega}p( Rf)  g  \, dx.
\end{equation*}
We conclude that $w=R f$, so that
\begin{equation}\label{eq:Rafa}
 R_{a_k}f_{a_k}\to Rf\quad\text{strongly in }L^2(\Omega) \quad\text{as }
  k\to\infty.
\end{equation}
From \eqref{eq:normRa-R}, \eqref{eq:compR}, and \eqref{eq:Rafa} it
follows that
\begin{equation*}
\|R_{a_k}- R\|_{\mc{L}(L^2(\Omega))}=\|R_{a_k}f_{a_k}- R
  f_{a_k}\|_{ L^2(\Omega)}
\leq \|R_{a_k}f_{a_k}-Rf\|_{L^2(\Omega)}+\| Rf-
  Rf_{a_k}\|_{ L^2(\Omega)}\to 0
\end{equation*}
as $k\to\infty$.
 Since
the limit of $\|R_{a_k}- R\|_{\mc{L}(L^2(\Omega))}$ depends neither on the
sequence $\{a_k\}$ nor on the extracted subsequence, by the Urysohn subsequence
principle we conclude that
\begin{equation}\label{oper_conv}
  \|R_{a}- R\|_{\mc{L}(L^2(\Omega))}\to 0\quad\text{as }
  \mathop{\rm dist}(a,\partial\Omega)\to 0.
\end{equation}
We observe that
\begin{equation*}
  \sigma( R_a)= \bigg\{\frac1{1+\la_k^a(\Omega,p)}\bigg\}_{k \in
  \mathbb{N}\setminus\{0\}} \cup \{0\},\quad
\sigma( R)= \bigg\{\frac1{1+\la_k(\Omega,p)}\bigg\}_{k \in
  \mathbb{N}\setminus\{0\}} \cup \{0\},
\end{equation*}
where $\sigma(R_a)$ and $\sigma(R)$ denote the spectrum of $R_a$ and
$R$, respectively.  Hence, \eqref{limit_eigen_conv} follows from
\cite[Chapter XI 9.5, Lemma 5, Page~1091]{DSJ_book} and \eqref{oper_conv}.

\end{proof}
\section{Expansion of eigenvalues}\label{sec_eigen_expansion}

Throughout this section, we let $p$ be as in \eqref{hp_p} and assume
that the domain $\Omega$ satisfies assumption~\eqref{eq:ass-Omega}.
For every $a\in\Omega$, we define the bilinear symmetric form
$q^a_{\Omega,p}: H^1(\Omega\setminus S_a^\Omega) \times
H^1(\Omega\setminus S_a^\Omega) \to\R$ as
\begin{equation}\label{eq:qf}
   q^a_{\Omega,p} (v,w):=\int_{\Omega\setminus S_a^\Omega}\big(\nabla v\cdot\nabla w+pv w)\,dx.
 \end{equation}
Furthermore,
 for every $a\in\Omega$ and
 $u\in \mathcal F(\Omega)=\{u\in H^1(\Omega):\Delta u\in
 L^2(\Omega)\}$, we consider the linear and bounded functional
 \begin{align}
   \notag &   L^{a,u}_{\Omega}: H^1(\Omega\setminus S_a^\Omega) \to \R,\\
\label{eq:cara_La}
          & L^{a,u}_{\Omega}(v):=
\int_{\Omega\setminus S_a^\Omega}\nabla u\cdot\nabla v\,dx+
   \int_{\Omega}(\Delta u)   v\,dx
\quad\text{for every }v\in
  H^1(\Omega\setminus S_a^\Omega).
 \end{align}
 We observe that, if $u\in \mathcal F(\Omega)\cap
 C^1(\overline{\Omega})$ and $\partial_\nu u=0$ on $\partial\Omega$, an integration by parts yields
 \begin{equation}\label{eq:cara_La2}
   L^{a,u}_{\Omega}(v)=-\int_{S_a^\Omega}(\nabla u\cdot \nu_a^\Omega)
   (\gamma_{a,\Omega}^+(v)-\gamma_{a,\Omega}^-(v))\,dS\quad\text{for every }v\in
   H^1(\Omega\setminus S_a^\Omega).
\end{equation}
For every $a\in\Omega$ and
 $u\in \mathcal F(\Omega)$,  we consider the quadratic functional
\begin{equation*}
  J^{a,u}_{\Omega,p}: H^1(\Omega\setminus S_a^\Omega)\to \R, \quad
  J^{a,u}_{\Omega,p}(v):=\frac12 q^{a}_{\Omega,p}(v,v)-L^{a,u}_{\Omega}(v).
\end{equation*}
By a standard minimization argument, for every $a\in\Omega$ and
$u\in \mathcal F(\Omega)$, there exists a unique function
$V^{a,u}_{\Omega,p}\in H^1(\Omega\setminus S_a^\Omega)$ such that
\begin{equation}\label{eq:def_Va}
  V^{a,u}_{\Omega,p}-u\in \mathcal H_a(\Omega)\quad\text{and}\quad
  J^{a,u}_{\Omega,p}(V^{a,u}_{\Omega,p})=\mathcal E^{a,u}_{\Omega,p},
\end{equation}
where 
\begin{equation}\label{eq:E_a}
  \mathcal E^{a,u}_{\Omega,p}:=
  \min\left\{J^{a,u}_{\Omega,p}(v):v\in H^1(\Omega\setminus S_a^\Omega), \
    v-u\in \mathcal H_a(\Omega)\right\}.
\end{equation}
Furthermore, $V^{a,u}_{\Omega,p}$ is the unique solution to the variational problem
\begin{equation}\label{eq:eq_Va}
  \begin{cases}
    V^{a,u}_{\Omega,p}-u\in \mathcal H_a(\Omega),\\
    q^{a}_{\Omega,p}(V^{a,u}_{\Omega,p},v)=L^{a,u}_{\Omega}(v),
    \quad\text{for all }v\in \mathcal H_a(\Omega).
  \end{cases}
\end{equation}
\begin{remark}\label{rem:comparison-weight}
  We observe that, if $p_1,p_2\in L^\infty(\Omega)$ satisfy
  \eqref{hp_p} and $p_1\leq p_2$ a.e. in $\Omega$, then
  \begin{equation*}
  \mathcal E^{a,u}_{\Omega,p_1}\leq  \mathcal
  E^{a,u}_{\Omega,p_2}.
\end{equation*}
\end{remark}

To estimate the quantity $\mathcal E^{a,u}_{\Omega,p}$, it is useful
to introduce the following cut--off functions.
  For every $r\in(0,1)$, let $\eta_{r} \in W^{1,\infty}(\R^2)$ be
  defined as
\begin{equation}\label{eq:def-cutoff}
  \eta_{r}(x)=
  \begin{cases}
    1, &\text{if } |x|<r,\\
    \frac{2\log|x|-\log r}{\log r}, &\text{if }
    r\leq |x|\leq \sqrt r, \\
    0, &\text{if } |x|>\sqrt r.
\end{cases}
\end{equation}
We observe that $0\le\eta_{r}(x)\le 1$ for all
$x\in\R^2$ and
\begin{equation}\label{eq:int-cut}
  \int_{\R^2}|\nabla
  \eta_r|^2\,dx=\frac{4\pi}{|\log r|}\quad\text{for all }r\in(0,1).
\end{equation}

\begin{example}
  Let us consider, for example, the case in which the function $u$ in
  the definition
\eqref{eq:E_a}   is identically equal to a constant $\kappa\in \R\setminus\{0\}$. In this
case $L^{a,\kappa}_{\Omega}\equiv0$ and
\begin{align*}
  \mathcal E^{a,\kappa}_{\Omega,p}&=\frac12 q^{a}_{\Omega,p}(
                                    V^{a,\kappa}_{\Omega,p}, V^{a,\kappa}_{\Omega,p})
                                    \leq \frac12 q^{a}_{\Omega,p}(\kappa\, \eta_{d_a^\Omega}(x-P_a), \kappa\,
                                    \eta_{d_a^\Omega}(x-P_a))\\
                                  &  \leq\frac{\kappa^2}2  \int_{\R^2}\big(|\nabla
                                    \eta_{d_a^\Omega}|^2+\|p\|_{L^\infty(\Omega)}\eta_{d_a^\Omega}^2\big)\,dx,
\end{align*}
where $\eta_{d_a^\Omega}$ is defined in \eqref{eq:def-cutoff}.
Hence, in view of \eqref{eq:int-cut} and \eqref{hp_p}
\begin{equation}\label{eq:const-case}
  \mathcal E^{a,\kappa}_{\Omega,p}
  =O\left(\frac1{|\log d_a^\Omega|}\right)\quad\text{and}
  \quad \| V^{a,\kappa}_{\Omega,p}\|_{H^1(\Omega\setminus S_a^\Omega)}=O\left(\frac1{\sqrt{|\log d_a^\Omega|}}\right)
  \quad\text{as }
  d_a^\Omega=\mathop{\rm dist}(a,\partial\Omega) \to0.
\end{equation}
\end{example}

\subsection{\texorpdfstring{$ \mathcal E^{a,u}_{\Omega,p}$}{E} for an eigenfunction
    \texorpdfstring{$u$}{u}}
In the remainder of this section, we focus on the case where $u$ is an
eigenfunction of problem \eqref{eq:eige_lapla} associated to a
possibly multiple eigenvalue $\lambda_n(\Omega,p)$.   
Let $n,m \in \N\setminus\{0\}$ be such that
\begin{equation}\label{lambda-n_multiple-m}
\text{$\lambda_n(\Omega,p)$ has
  multiplicity $m$ as an eigenvalue of \eqref{eq:eige_lapla}},
\end{equation}
i.e. 
\begin{equation}\label{lambda-n_multiple-m-bis}
  \lambda_{n-1}(\Omega,p)<\lambda_n(\Omega,p)=\lambda_{n+1}(\Omega,p)
  =\cdots=\lambda_{n+m-1}(\Omega,p)<\lambda_{n+m}(\Omega,p).
\end{equation}
Let $E(\lambda_n(\Omega,p))$ be the associated $m$-dimensional eigenspace, i.e.
\begin{equation*}
  E(\lambda_n(\Omega,p))=\big\{v\in H^1(\Omega): v\text{ weakly solves
    }\eqref{eq:eige_lapla} \text{ with }\lambda=\lambda_n(\Omega,p)\big\}.
\end{equation*}
Let
\begin{equation}\label{eq:basis}
\{u_{n}, u_{n+1},\dots, u_{n+m-1}\} \quad\text{be a basis of
  $E(\lambda_n(\Omega,p))$ orthonormal in $L^2(\Omega,p)$},
\end{equation}
i.e., orthonormal with respect to the scalar product \eqref{eq:scal-p}.
For every $j\in\{1,\dots,m\}$, since $u_{n+j-1}$ is an eigenfunction
of the weighted Neumann Laplacian in a Lipschitz domain, it is well
known that
\begin{equation}\label{eq:reg_u_n}
  u_{n+j-1} \in C^{1,\alpha}(\overline{\Omega})\quad\text{for all }\alpha\in (0,1], 
\end{equation}
see e.g. \cite{winkert2010}.

If $u\in E(\lambda_n(\Omega,p))$, 
$L^{a,u}_{\Omega}$ can be characterized as in  \eqref{eq:cara_La2}
by \eqref{eq:reg_u_n}; moreover, by \eqref{eq:eige_lapla} we have 
 \begin{equation}
   \label{eq:def_La}
   L^{a,u}_{\Omega}(v)= q^a_{\Omega,p}(u,v)-(\lambda_n(\Omega,p)+1) \int_{\Omega} pu
   v\,dx
   =\int_{\Omega\setminus S_a^\Omega}\nabla u\cdot\nabla v\,dx
   -\lambda_n(\Omega,p) \int_{\Omega}p u   v\,dx.
 \end{equation}
For every $a\in\Omega$, we consider the bilinear form 
\begin{align}
\notag  &r_{\Omega,p,n}^a: E(\lambda_n(\Omega,p)) \times E(\lambda_n(\Omega,p))\to\R,\\
 \label{eq:r-a} &  r_{\Omega,p,n}^a (u,w) :=(\lambda_n(\Omega,p)+1)\int_\Omega p V_{\Omega,p}^{a,u}(w-V_{\Omega,p}^{a,w})\,dx .
\end{align}

\begin{remark}\label{r_r-a-symm}
  Let $r_{\Omega,p,n}^a$ be the bilinear form defined in \eqref{eq:r-a}. For every
  $u,w\in  E(\lambda_n(\Omega,p))$ we have 
  \begin{align*}
    r_{\Omega,p,n}^a (u,w)&-r_{\Omega,p,n}^a (w,u)=(\lambda_n(\Omega,p)+1)\int_\Omega p
                       w V_{\Omega,p}^{a,u}\,dx-(\lambda_n(\Omega,p)+1)\int_\Omega p
                       u V_{\Omega,p}^{a,w} \,dx\\
    &=q_{\Omega,p}^a(w,
      V_{\Omega,p}^{a,u})-L^{a,w}_\Omega(V_{\Omega,p}^{a,u})-\left(
      q_{\Omega,p}^a(u, V_{\Omega,p}^{a,w})-L^{a,u}_\Omega(V_{\Omega,p}^{a,w})\right)\\
            &=q_{\Omega,p}^a(V_{\Omega,p}^{a,u},w-V_{\Omega,p}^{a,w})+q_{\Omega,p}^a(V_{\Omega,p}^{a,u},
              V_{\Omega,p}^{a,w})-L_\Omega^{a,w}(V_{\Omega,p}^{a,u})\\
    &\qquad-
              \left(q_{\Omega,p}^a(V_{\Omega,p}^{a,w},u-V_{\Omega,p}^{a,u})+q_{\Omega,p}^a(V_{\Omega,p}^{a,w},
              V_{\Omega,p}^{a,u})-
              L^{a,u}_\Omega( V_{\Omega,p}^{a,w})\right)\\
    &=L^{a,u}_\Omega (w)-L^{a,u}_\Omega (V_{\Omega,p}^{a,w})
      -L^{a,w}_\Omega(V_{\Omega,p}^{a,u})-L^{a,w}_\Omega
      (u)+L^{a,w}_\Omega (V_{\Omega,p}^{a,u}) +L^{a,u}_\Omega (V_{\Omega,p}^{a,w})=0
  \end{align*}
  by
  \eqref{eq:def_La}, \eqref{eq:eq_Va}, and the fact that
  $L^{a,u}_\Omega (w)=L^{a,w}_\Omega (u)=0$ because both $u$ and $w$  solve \eqref{eq:eige_lapla} with $\lambda=\lambda_n(\Omega,p)$.
Hence
\begin{equation*}
  r_{\Omega,p,n}^a: E(\lambda_n(\Omega,p)) \times E(\lambda_n(\Omega,p))\to\R \quad \text{is a symmetric
  bilinear form}.
\end{equation*}
\end{remark}

\medskip\noindent For notational convenience, we set, for every
$a\in\Omega$ and $u\in E(\lambda_n(\Omega,p))$,
\begin{align}
\notag&  \lambda_n=\lambda_n(\Omega,p),\quad  d_a=d_a^\Omega, \quad P_a=P_a^\Omega, \quad S_a=S_a^\Omega,\\
\label{eq:simply-notation} &q_a=q^a_{\Omega,p}, \quad L^u_a= L^{a,u}_{\Omega},\quad J^u_a=
   J^{a,u}_{\Omega,p},\quad 
   \mathcal E^u_a=\mathcal E^{a,u}_{\Omega,p},\quad
                             V^u_a=V^{a,u}_{\Omega,p},\quad r_a=r_{\Omega,p,n}^a, 
\end{align}
omitting the dependence on $\Omega$ and $p$
 in the notation, while keeping it implicitly understood.

\begin{proposition}\label{prop_La}
  For every $a\in\Omega$ and $u\in E(\lambda_n)$, the map $L^u_a$
  defined in \eqref{eq:cara_La} and \eqref{eq:simply-notation} belongs to the dual space
  $(H^1(\Omega\setminus S_a))^*$. Furthermore, there exists $\delta_1>0$ and, for every
  $u\in E(\lambda_n)$, there exists a constant $C_1(u,\Omega,p)>0$ (which depends  on $u, \Omega,p$
  but is independent of $a$) such that, for all $a\in\Omega$ with
  $d_a\in(0,\delta_1)$,
\begin{equation}\label{eq_norm_La}
\|L^u_a\|_{(H^1(\Omega\setminus
  S_a))^*}\leq \frac{C_1(u, \Omega,p)}{\sqrt{|\log d_a|}}.
\end{equation}
\end{proposition}
\begin{proof}
For every $v\in H^1(\Omega\setminus S_a)$, the function 
\begin{equation*}
  w(x):=v(x)\big(1-\eta_{d_a}(x-P_a)\big), \quad x\in\Omega,
\end{equation*}
belongs to $H^1(\Omega)$, where $\eta_{d_a}$ is defined in
\eqref{eq:def-cutoff}.  Hence, by \eqref{eq:def_La} and the fact that
$u$ solves \eqref{eq:eige_lapla} with $\lambda=\lambda_n$, we have
\begin{align*}
  L^u_a(v)&=\int_{\Omega}\nabla u\cdot\nabla w\,dx
            -\lambda_n\int_{\Omega}p u   w\,dx\\
  &\quad +
            \int_{\Omega\setminus S_a}\nabla u\cdot\nabla (v \eta_{d_a}(\cdot-P_a))\,dx
    -\lambda_n \int_{\Omega}p u   v \eta_{d_a}(\cdot-P_a)\,dx\\
  &=\int_{\Omega\setminus S_a}\eta_{d_a}(\cdot-P_a)\nabla u\cdot\nabla v\,dx+
    \int_{\Omega\setminus S_a}v\nabla u\cdot\nabla
    \eta_{d_a}(\cdot-P_a)\,dx\\
  &\quad
    -\lambda_n \int_{\Omega}p u   v \eta_{d_a}(\cdot-P_a)\,dx.
\end{align*}
Hence, taking into account \eqref{eq:int-cut}, \eqref{hp_p} and the fact
that $u,|\nabla u|\in L^\infty(\Omega)$ in
view of \eqref{eq:reg_u_n},
\begin{align*}
  &\big|L^u_a(v)\big|\\
  &\leq \|v\|_{H^1(\Omega\setminus
                         S_a)}\left(\|\nabla u\|_{L^\infty(\Omega)}
                         \sqrt{\pi d_a}+\|\nabla
                         u\|_{L^\infty(\Omega)}\sqrt\frac{4 \pi}{|\log
                                      d_a|}+
                                      \lambda_n
                         \|u\|_{L^\infty(\Omega)}\|p\|_{L^\infty(\Omega)}\sqrt{\pi
                         d_a}\right)\\
  &\leq \|v\|_{H^1(\Omega\setminus
                         S_a)}\left(\|\nabla u\|_{L^\infty(\Omega)}
                         +\|\nabla
                         u\|_{L^\infty(\Omega)}+\lambda_n
    \|u\|_{L^\infty(\Omega)}\|p\|_{L^\infty(\Omega)}\right)
    \sqrt\frac{4 \pi}{|\log d_a|}
\end{align*}
for every $v\in H^1(\Omega\setminus S_a)$, provided $d_a$ is
sufficiently small (independently of $u$ and $v$). This completes the proof.
\end{proof}
\begin{proposition}\label{prop_Ee_limit}
  For every $a\in\Omega$ and $u\in E(\lambda_n)$, let $\mathcal E^u_a$
  and $V^u_a$ be defined in \eqref{eq:E_a},
  \eqref{eq:def_Va}, and \eqref{eq:simply-notation}.
There exists $\delta_2>0$ and, for every
  $u\in E(\lambda_n)$, there exists a constant $C_2(u,\Omega,p)>0$ (which depends  on $u$,
  $\Omega$ and $p$,
  but is independent of $a$) such that, for all $a\in\Omega$ with
  $d_a\in(0,\delta_2)$,
\begin{equation}\label{ineq_Ea_upper-lower}
|\mathcal E^u_a|\leq 
\frac{C_2(u,\Omega,p)}{|\log d_a|}
\end{equation}
and 
\begin{equation}\label{eq:Vato0}
\|V^u_a\|_{H^1(\Omega\setminus S_a)}\leq 
\frac{C_2(u,\Omega,p)}{\sqrt{|\log d_a|}}.
\end{equation}
In particular, for every $u\in E(\lambda_n)$,
\begin{equation}\label{eq:Ea-to-zero}
  \mathcal E^u_a\to0\quad\text{and}\quad
\|V^u_a\|_{H^1(\Omega\setminus S_a)}\to 0
  \quad\text{as }
\mathop{\rm dist}(a,\partial\Omega) \to0.
\end{equation}
\end{proposition}
\begin{proof}
  For every $r\in(0,1)$, let $\eta_{r} \in W^{1,\infty}(\R^2)$ be
  defined in \eqref{eq:def-cutoff}.
For any $u\in E(\lambda_n)$ and $a\in\Omega$, the function 
\begin{equation*}
  v(x):=u(x)\eta_{d_a}(x-P_a), \quad x\in\Omega,
\end{equation*}
satisfies
\begin{equation*}
v\in H^1(\Omega)\quad\text{and}\quad 
    v-u\in \mathcal H_a(\Omega).
  \end{equation*}
  Hence, by \eqref{eq:E_a}, \eqref{eq:int-cut}, and the fact that
  $u,|\nabla u|\in L^\infty(\Omega)$ in view of \eqref{eq:reg_u_n},
  \begin{align*}
    \mathcal E^u_a&\leq J^{a,u}_{\Omega,p}(v)\\
                  &=\frac12\int_{\Omega\setminus S_a}|\nabla
                  u(x)|^2 \eta_{d_a}^2(x-P_a)\,dx+
                  \frac12\int_{\Omega\setminus S_a}u^2(x)|\nabla
                  \eta_{d_a}(x-P_a)|^2\,dx\\
    &\quad +
                  \int_{\Omega\setminus S_a}u(x)
                  \eta_{d_a}(x-P_a)\nabla u(x)\cdot 
      \nabla \eta_{d_a}(x-P_a)\,dx
      +
                 \frac12  \int_{\Omega\setminus
      S_a}p(x)u^2(x)\eta_{d_a}^2
      (x-P_a)\,dx\\
    &\quad- \int_{\Omega\setminus S_a}\nabla u(x)
\cdot ( \eta_{d_a}(x-P_a)\nabla u(x)+u(x)
      \nabla \eta_{d_a}(x-P_a))\,dx \\
    &\quad+\lambda_n
       \int_{\Omega\setminus
      S_a}p(x)u^2(x)\eta_{d_a}
      (x-P_a)\,dx\\
&\leq \left(\|\nabla u\|^2_{L^\infty(\Omega)}+\|
      u\|^2_{L^\infty(\Omega)}\right)
      \left(\Big(\norm{p}_{L^\infty(\Omega)}(\lambda_n+\tfrac12)+2\Big)\pi d_a+\frac{4\pi}{|\log d_a|}
+\frac12 \sqrt{\pi d_a}\sqrt \frac{4\pi}{|\log
                      d_a|}\right)\\
    &\leq \frac{C(u,\Omega,p)}{|\log d_a|} 
  \end{align*}
  provided $d_a$ is sufficiently small (independently of $u$), for
  some constant $C(u,\Omega,p)>0$ depending only on $u,\Omega,p$.

On the other hand,  
 by \eqref{hp_p} and \eqref{eq:def_Va}
\begin{align*}
  \min\{1,c\}\|V^u_a\|_{H^1(\Omega\setminus S_a)}^2&\leq q^a_{\Omega,p}(V^u_a,V^u_a)=
  2\mathcal E^u_a+2 L^u_a(V^u_a)\\
  &\leq 2 \mathcal E^u_a +
  2 \|L^u_a\|_{(H^1(\Omega\setminus S_a))^*}\|V^u_a\|_{H^1(\Omega\setminus S_a)}\\
  &\leq
   2 \mathcal E^u_a+\tfrac{2}{\min\{1,c\}} \|L^u_a\|_{(H^1(\Omega\setminus
    S_a))^*}^2+\tfrac{\min\{1,c\}}2  \|V^u_a\|_{H^1(\Omega\setminus S_a)}^2,
\end{align*}
so that 
\begin{equation}\label{eq:est-Ea+La}
   \mathcal E^u_a+\tfrac1{\min\{1,c\}}\|L^u_a\|_{(H^1(\Omega\setminus
    S_a))^*}^2\geq \tfrac{\min\{1,c\}}4 \|V^u_a\|_{H^1(\Omega\setminus S_a)}^2\geq0.
\end{equation}
The above estimate implies that, if $d_a\in(0,\delta_1)$, then 
\begin{equation*}
  \mathcal E^u_a \geq -\tfrac1{\min\{1,c\}}\|L^u_a\|_{(H^1(\Omega\setminus
    S_a))^*}^2\geq -\frac{C_1^2(u,\Omega,p)}{\min\{1,c\}|\log d_a|}
\end{equation*}
in view of \eqref{eq_norm_La}. Estimate \eqref{ineq_Ea_upper-lower} is
thereby proved.

To conclude, we observe that \eqref{eq:Vato0} follows from
\eqref{eq:est-Ea+La}, taking into account \eqref{eq_norm_La} and
\eqref{ineq_Ea_upper-lower}. Finally \eqref{eq:Ea-to-zero} is a direct
consequence of \eqref{ineq_Ea_upper-lower} and \eqref{eq:Vato0}.
\end{proof}

The following proposition clarifies the relationship between the
quantity $\mathcal E^u_a$ and the bilinear form $r_a$.

\begin{proposition}\label{p:relation_E_r}
  For every $a\in\Omega$ and $u\in E(\lambda_n)$, let $r_a$,
  $\mathcal E^u_a$, and $V^u_a$ be defined in \eqref{eq:r-a},
  \eqref{eq:E_a}, \eqref{eq:def_Va}, and \eqref{eq:simply-notation}.
  Then, for every $u\in E(\lambda_n)$,
  \begin{equation*}
  r_a(u,u)=2\mathcal E^u_a+O\left(\|V^u_a\|_{L^2(\Omega)}^2\right)
  \quad\text{as }
\mathop{\rm dist}(a,\partial\Omega) \to0.
\end{equation*}  
\end{proposition}
\begin{proof}
  In view of \eqref{eq:r-a}, for every $u\in E(\lambda_n)$
\begin{equation*}
  r_a(u,u)=(\lambda_n+1)\int_\Omega p
  V^{u}_a(u-V^{u}_a)\,dx=
  (\lambda_n+1)\int_\Omega  p\, u
  V^{u}_a\,dx+O\left(\|V_a^{u}\|_{L^2(\Omega)}^2\right)
\end{equation*}
as $\mathop{\rm dist}(a,\partial\Omega)\to 0$. 
Moreover, by \eqref{eq:def_La},
\eqref{eq:eq_Va}, and \eqref{eq:def_Va},
\begin{align*}
  (\lambda_n+1)\int_\Omega p u V_a^{u}\,dx&=
                                              q_a(u,V_a^{u})-L_a^{u}(V_a^{u})\\
  &=q_a(V_a^{u},u-V_a^{u})+q_a(V_a^{u},V_a^{u}) -L_a^{u}(V_a^{u})\\
  &=L_a^{u}(u-V_a^{u})+q_a(V_a^{u},V_a^{u}) -L_a^{u}(V_a^{u})\\
  &=L_a^{u}(u)-2L_a^{u}(V_a^{u})+2\big(\mathcal E_a^{u}+L_a^{u}(V_a^{u})\big)=2\mathcal
    E_a^{u}+L_a^{u}(u)=
    2\mathcal E_a^{u},
\end{align*}
where, in the last identity, we used the fact that $L_a^{u}(u)=0$,
because $u$ is an eigenfunction of \eqref{eq:eige_lapla} associated to
$\lambda_n$. The proof is thereby complete.
\end{proof}

  For every $a\in\Omega$, let
  \begin{equation}\label{eq:Ta}
    \mathcal T_a:E(\lambda_n)\to \mathcal H_a,\quad
    \mathcal T_a(u)=u-V^u_a,
  \end{equation}
  with $V^u_a$ as in \eqref{eq:def_Va} and \eqref{eq:simply-notation}. 
We observe that $\mathcal T_a$ is a linear operator. Moreover, the
following proposition establishes that  $\mathcal T_a$ is close to the
identity, by estimating the operator norm of the difference.
\begin{proposition}\label{p:op-norm}
  For every $a\in\Omega$, let
  \begin{equation}\label{eq:def-tau-a}
    \tau_a:=\|\mathop{\rm Id}-\mathcal T_a\|_{\mathcal
      L(E(\lambda_n),L^2(\Omega))},
  \end{equation}
  where $\mathop{\rm Id}$ is the identity map on $L^2(\Omega)$,
  $E(\lambda_n)$ is meant as a subspace of $L^2(\Omega)$, and both
  $E(\lambda_n)$ and $L^2(\Omega)$ are endowed with the scalar product
  \eqref{eq:scal-p}, so that
  \begin{equation*}
   \tau_a=\sup\{\|u-\mathcal T_a(u)\|_{L^2(\Omega,p)}:u\in
   E(\lambda_n) \text{ and }\|u\|_{L^2(\Omega,p)}=1\}.
  \end{equation*}
  Then
  \begin{equation}\label{eq:est-tau-1}
    \tau_a\leq \|p\|_{L^\infty(\Omega)}^{1/2}\sum_{j=1}^m\left\|V^{u_{n+j-1}}_a\right\|_{L^2(\Omega)},
  \end{equation}
  where $\{u_{n+j-1}\}_{j=1}^m$ is a $L^2(\Omega,p)$-orthonormal basis
  of $E(\lambda_n)$ as in \eqref{eq:basis}. Furthermore,
  \begin{equation}\label{eq:est-tau-2}
    \tau_a=O\bigg(\frac1{\sqrt{|\log d_a|}}\bigg)
  \end{equation}
as $d_a=\mathop{\rm dist}(a,\partial\Omega) \to0$.
\end{proposition}
\begin{proof}
  We have
  \begin{align*}
    \|\mathop{\rm Id}-\mathcal T_a\|_{\mathcal
    L(E(\lambda_n),L^2(\Omega))}&=
                                  \sup_{\substack{(\alpha_1,\dots,\alpha_m)\in
                                  \R^m\\\sum_{j=1}^m
    \alpha_j^2=1}}\left\|(\mathop{\rm Id}-\mathcal T_a)\Big(\sum\nolimits_{j=1}^m
    \alpha_j u_{n+j-1}\Big)\right\|_{L^2(\Omega,p)}\\
    &=\sup_{\substack{(\alpha_1,\dots,\alpha_m)\in
                                  \R^m\\\sum_{j=1}^m
    \alpha_j^2=1}} \left\|\sum\nolimits_{j=1}^m    \alpha_j
    V_a^{u_{n+j-1}}\right\|_{L^2(\Omega,p)}\le \|p\|_{L^\infty(\Omega)}^{1/2}\sum_{j=1}^m\left\|V^{u_{n+j-1}}_a\right\|_{L^2(\Omega)},
  \end{align*}
thus proving \eqref{eq:est-tau-1}. Estimate \eqref{eq:est-tau-2}
directly follows from \eqref{eq:est-tau-1} and \eqref{eq:Vato0}.
\end{proof}

For every $j\in\N\setminus\{0\}$ and $a\in\Omega$ with small
$\mathop{\rm dist}(a,\partial\Omega)$, we denote by
$E(\lambda_j^a)\subset \mathcal H_a(\Omega)$ the eigenspace of problem
\eqref{eq:eige_a} corresponding to the $j$-th eigenvalue
$\lambda_j^a=\lambda_j^a(\Omega,p)$. Then we consider the direct sum of the eigenspaces
corresponding to the eigenvalues $\{\lambda_{n+j-1}^a\}_{j=1}^{m}$,
i.e.
\begin{equation*}
  E_a=\bigoplus_{j=1}^{m}E(\lambda_{n+j-1}^a),
\end{equation*}
and the orthogonal projection 
\begin{equation}\label{eq:projection}
  \Pi_a: \  L^2(\Omega)\to E_a\subset \mathcal H_a(\Omega)
\end{equation}
with respect to the scalar product \eqref{eq:scal-p}.

The following theorem provides an asymptotic expansion of the
eigenvalue variation $\lambda_{n+j-1}^a -\lambda_n$ as
$\mathop{\rm dist}(a,\partial\Omega)\to 0$. We provide a proof based
on the \emph{Lemma on small eigenvalues} by Y. Colin de Verdi\`ere
\cite{ColindeV1986}, whose statement is recalled in the appendix for
completeness.

  \begin{theorem}\label{theo_asymp_abstract}
    Let assumptions
    \eqref{lambda-n_multiple-m}--\eqref{lambda-n_multiple-m-bis} be
    satisfied. For every $a\in \Omega$ and $j\in\{1,\dots,m\}$,
    the $j$-th eigenbranch emanating from $\lambda_n$ can be expanded as
\begin{equation}\label{eq:expa1} 
  \lambda_{n+j-1}^a=\lambda_n+\mu_j^a+O(\tau_a^2)\quad 
  \text{as }\mathop{\rm dist}(a,\partial\Omega)\to 0,
\end{equation}
and
\begin{equation}\label{eq:expa2} 
  \lambda_{n+j-1}^a=\lambda_n+O(\tau_a)\quad 
  \text{as }\mathop{\rm dist}(a,\partial\Omega)\to 0,
\end{equation}where $\tau_a$ is defined in \eqref{eq:def-tau-a} and
$\{\mu_j^a\}_{j=1}^m=\{\mu_j^a(\Omega,p)\}_{j=1}^m$ are the eigenvalues (taken in non-decreasing
order) of the bilinear form $r_a=r_{\Omega,p,n}^a$ introduced in \eqref{eq:r-a}.
Furthermore, 
\begin{equation}\label{eq_energy_eigenfunctions-a}
  \|\mathcal T_a-\Pi_a\circ \mathcal T_a\|_{\mathcal
    L(E(\lambda_n),H^1(\Omega\setminus S_a))}
  =\sup_{\substack{u\in
      E(\lambda_n)\\
      \|u\|_{L^2(\Omega,p)}=1}}
  \|(\mathcal T_a-\Pi_a\circ \mathcal T_a)u\|_{H^1(\Omega\setminus
    S_a)} =O(\tau_a)
  \end{equation}
  as
$\mathop{\rm dist}(a,\partial\Omega)\to 0$.
\end{theorem}
\begin{proof}
  For every $a\in\Omega$, let $\mathcal T_a$ be defined in \eqref{eq:Ta}.
  We apply Lemma \ref{lemma:CdV} with
  \begin{align*}
    & \mathcal D=\mathcal H_a(\Omega),\quad H=L^2(\Omega)
      \text{ endowed with the scalar product \eqref{eq:scal-p}},\\
    & q(u,v)=\tilde q_a(u,v)=q_a(u,v)-(\lambda_n+1)\int_{\Omega}p uv\,dx,\\
    & \nu_j=\lambda_j^a-\lambda_n,\quad 
      F=\mathcal T_a(E(\lambda_n)).
  \end{align*}
  We recall from \eqref{eq:compact} that the embedding
  $\mathcal H_a(\Omega) \hookrightarrow \hookrightarrow L^2(\Omega)$ is
  compact.  Moreover, by Proposition \ref{p:op-norm},
  $\|\mathop{\rm Id}-\mathcal T_a\|_{\mathcal
    L(E(\lambda_n),L^2(\Omega))}\to 0$ as
  $\mathop{\rm dist}(a,\partial\Omega)\to 0$. Hence, if
  $\{u_{n+j-1}\}_{j=1}^m$ is a $L^2(\Omega,p)$-orthonormal basis of
  $E(\lambda_n)$ as in \eqref{eq:basis},
  $\{\mathcal T_a(u_{n+j-1})\}_{j=1}^m$ are linearly independent in
  $L^2(\Omega)$ provided $\mathop{\rm dist}(a,\partial\Omega)$ is
  sufficiently small, so that $F=\mathcal T_a(E(\lambda_n))$ is a
  $m$-dimensional subspace of $\mathcal H_a(\Omega)$.

Since $\lambda_n$ has multiplicity
$m$ and, by Proposition \ref{prop_spectral_stability}, $\lambda_k^a\to\lambda_k$ as $d_a\to0$
for every $k\in\N\setminus\{0\}$, assumption (H1) in Lemma
\ref{lemma:CdV} is satisfied, provided
$\mathop{\rm dist}(a,\partial\Omega)$ is small enough, with the choice 
\begin{equation*}
	\gamma:={\frac12}\min\{\lambda_n-\lambda_{n-1},\lambda_{n+m}-\lambda_n\}
\end{equation*}
when $n\geq2$, whereas $\gamma:={ \frac12}
	\left(\lambda_{2}-\lambda_{1}\right)$ when $n=1$ (recall that
        $\lambda_1=0$, with multiplicity $1$).

Let us now estimate the value $\delta$ in assumption (H2). For every
$v\in\mathcal H_a(\Omega)$ and $u\in E(\lambda_n)$, by \eqref{eq:eq_Va} and \eqref{eq:def_La}
  \begin{align*}
|\tilde q_a(\mathcal T_a(u),v)|&=
   \left| q_a(u-V^u_a,v)-(\lambda_n+1)\int_{\Omega} p (u-V^u_a)v\,dx\right|
                      =(\lambda_n+1)\left|\int_\Omega p V_a^uv\,dx\right|\\
    &\leq (\lambda_n+1) \|(\mathop{\rm Id}-\mathcal
      T_a)(u)\|_{L^2(\Omega,p)}\|v\|_{L^2(\Omega,p)}\\
&      \leq (\lambda_n+1)\tau_a\|u\|_{L^2(\Omega,p)}\|v\|_{L^2(\Omega,p)}
      \leq (\lambda_n+1)\frac{\tau_a}{1-\tau_a}\|\mathcal T_a (u)\|_{L^2(\Omega,p)}\|v\|_{L^2(\Omega,p)},
  \end{align*}
  where $\tau_a$ is defined in \eqref{eq:def-tau-a}; in the last
  inequality above, we have used the fact that
  \begin{align*}
    \|\mathcal T_a (u)\|_{L^2(\Omega,p)}&\geq
    \|u\|_{L^2(\Omega,p)}-\|(\mathop{\rm Id}-\mathcal T_a) (u)\|_{L^2(\Omega,p)}\\
&  \geq
    \|u\|_{L^2(\Omega,p)}-\tau_a\|u\|_{L^2(\Omega,p)}=
  (1-\tau_a)\|u\|_{L^2(\Omega,p)}.
  \end{align*}
  Hence
  \begin{align}\notag
    \delta:&=\sup\{|\tilde q_a(\mathcal T_a(u),v)|\colon v\in
             \mathcal H_a(\Omega),u\in
     E(\lambda_n),\|v\|_{L^2(\Omega,p)}
     =\|\mathcal
             T_a(u)\|_{L^2(\Omega,p)}=1\}\\
    \label{eq:scelt-delta}&\leq
     (\lambda_n+1) \frac{\tau_a}{1-\tau_a}
  \end{align}
        and assumption (H2) of Lemma \ref{lemma:CdV} is satisfied
        provided $\mathop{\rm dist}(a,\partial\Omega)$ is small
        enough, in view of  \eqref{eq:est-tau-2}.

Therefore, we can apply Lemma \ref{lemma:CdV} and infer the
following expansion for every $j\in\{1,\dots,m\}$:
\begin{equation}\label{eq:asympt_proof}
  \lambda_{N+j-1}^a=\lambda_n+\xi_j^a +O(\tau_a^2)\quad\text{as }
  \mathop{\rm dist}(a,\partial\Omega)\to0,
\end{equation}
where $\{ \xi_j^a \}_{i=1}^m$  are the eigenvalues (in ascending order)
of the restriction of $\tilde q_a$ to $\mathcal T_a(E(\lambda_n))$.

In order to characterize the values $\{\xi_j^a\}_{i=1,\dots,m}$
appearing in \eqref{eq:asympt_proof}, we consider the bilinear
form $r_a$ defined in \eqref{eq:r-a} and observe that, by
\eqref{eq:eq_Va}, \eqref{eq:def_La}, and \eqref{eq:r-a}, 
 \begin{align*}
      \tilde q_a&(\mathcal T_a(u), \mathcal T_a(w))=
       q_a(u-V^u_a,w-V^w_a)-(\lambda_n+1)\int_\Omega p (u-V^u_a)(w-V^w_a)\,dx\\
    &=q_a(u,w-V^w_a)- q_a(V^u_a,w-V^w_a)-  (\lambda_n+1)\int_\Omega p
      u(w-V^w_a)\,dx +
      (\lambda_n+1)\int_\Omega  p V^u_a(w-V^w_a)\,dx\\
   &=q_a(u,w-V^w_a)-L^u_a(w-V^w_a)-  (\lambda_n+1)\int_\Omega p
   u(w-V^w_a)\,dx+
     (\lambda_n+1)\int_\Omega p V^u_a(w-V^w_a)\,dx\\
   &=(\lambda_n+1)\int_\Omega  p V^u_a(w-V^w_a)\,dx=r_a(u,w)
    \end{align*}
    for every $u,w\in E(\lambda_n)$.  Hence, the bilinear form
    $\tilde q_a$ restricted to $\mathcal T_a(E(\lambda_n))$ coincides
    with the bilinear form
    \begin{equation*}
      (\varphi,\psi)\in  \mathcal T_a(E(\lambda_n)) \times \mathcal T_a(E(\lambda_n)) \ \mapsto\  r_a\big( \mathcal T_a^{-1}(\varphi),
      \mathcal T_a^{-1} (\psi)\big).
  \end{equation*}
  It follows that  the eigenvalues $\{\xi_j^a\}_{j=1,\dots,m}$ of
  $\tilde q_a$
  restricted to $\mathcal T_a(E(\lambda_n))$ are precisely the eigenvalues of the
  $m\times m$ symmetric real matrix
  \begin{equation}\label{eq:xi_mu_1}
         \mathsf{B}_a=\mathsf{C}_a^{-1}\mathsf{R}_a,
    \end{equation}
    where
    \begin{equation*}
  \mathsf{R}_a=\big(r_a(u_{n+j-1}, u_{n+k-1})\big)_{jk}
\end{equation*}
and 
  \begin{equation*}
    \mathsf{C}_a=\big(c_{jk}^a \big)_{jk},\quad
    c_{jk}^a=\int_\Omega p\,\mathcal T_a(u_{n+j-1}) \mathcal T_a(u_{n+k-1})\,dx,
\end{equation*}
being   $\{u_{n+j-1}\}_{j=1}^m$ as in \eqref{eq:basis}. From
\eqref{eq:basis} and \eqref{eq:def-tau-a} we have
\begin{equation*}
  \mathsf{C}_a=\mathbb{I}+O(\tau_a)\quad\text{as }
  \mathop{\rm dist}(a,\partial\Omega)\to0,
\end{equation*}
with $\mathbb{I}$ being the identity $m\times m$ matrix, in the sense
that all the entries of the matrix $\mathsf{C}_a-\mathbb{I}$ are
$O(\tau_a)$ as $\mathop{\rm dist}(a,\partial\Omega)\to0$. Hence
$\mathsf{C}_a^{-1}=\mathbb{I}+O(\tau_a)$ as 
  $\mathop{\rm dist}(a,\partial\Omega)\to0$.
Furthermore, by \eqref{eq:r-a} and \eqref{eq:def-tau-a}, 
\begin{equation}\label{eq:expRa}
  \mathsf{R}_a=O(\tau_a)\quad\text{as }
  \mathop{\rm dist}(a,\partial\Omega)\to0,
\end{equation}
i.e.  all the entries of the matrix $\mathsf{R}_a$ are
$O(\tau_a)$ as $\mathop{\rm dist}(a,\partial\Omega)\to0$.
It then follows  from
\eqref{eq:xi_mu_1}  that
\begin{equation}\label{eq:exp-matr}
  \mathsf{B}_a=(\mathbb{I}+O(\tau_a))\mathsf{R}_a=\mathsf{R}_a+O(\tau_a^2) \quad\text{as }
  \mathop{\rm dist}(a,\partial\Omega)\to0.
\end{equation}
Since the eigenvalues$\{\mu_j^a\}_{j=1}^m$ of the bilinear form $r_a$
coincide with the eigenvalues of the matrix $\mathsf{R}_a$, the min-max
characterization of eigenvalues and  \eqref{eq:exp-matr} yield the
expansion 
\begin{equation}\label{eq:expa-xi-mu}
  \xi_j^a=\mu_j^a +O(\tau_a^2) \quad\text{as }
  \mathop{\rm dist}(a,\partial\Omega)\to0,
\end{equation}
for all $j\in \{1,\dots,m\}$. Combining \eqref{eq:asympt_proof} with
\eqref{eq:expa-xi-mu}, we finally obtain \eqref{eq:expa1}, which also
implies \eqref{eq:expa2} due to the fact that
\begin{equation*}
  \mu_j^a=O(\tau_a)\quad\text{as }
  \mathop{\rm dist}(a,\partial\Omega)\to0,
\end{equation*}
by \eqref{eq:expRa} and the min-max
characterization of eigenvalues.

To prove \eqref{eq_energy_eigenfunctions-a}, we first observe that
estimate \eqref{eq:cdv-eigenfunctions} in
Lemma \ref{lemma:CdV} and \eqref{eq:scelt-delta} directly imply 
\begin{equation}\label{eq:energy_eige-L2}
   \|\mathcal T_a-\Pi_a\circ \mathcal T_a\|_{\mathcal
      L(E(\lambda_n),L^2(\Omega))}=
    O(\tau_a)\quad\text{as }\mathop{\rm dist}(a,\partial\Omega)\to 0,
  \end{equation}
  where in the operator norm $\|\mathcal T_a-\Pi_a\circ \mathcal
    T_a\|_{\mathcal L(E(\lambda_n),L^2(\Omega))}$ we mean that both
    $E(\lambda_n)$ and  $L^2(\Omega)$ are endowed
  with the scalar product \eqref{eq:scal-p}.
  Moreover, for every $u\in E(\lambda_n)$ and $\varphi\in \mathcal
  H_a(\Omega)$, \eqref{eq:eq_Va} and \eqref{eq:def_La} yield 
  \begin{equation}\label{eq:eq-Tau}
    q_a(\mathcal T_a(u),\varphi) =q_a(u-V^u_a,\varphi)=q_a(u,
    \varphi)-L^u_a(\varphi)=(\lambda_n+1)\int_\Omega p u\varphi\,dx.
  \end{equation}
For every $u\in E(\lambda_n)$, we denote $f_a^u:= \mathcal T_a
u-\Pi_a(\mathcal T_a u)\in\mathcal H_a(\Omega)$.
From \eqref{eq:eq-Tau} it follows that, for every $u\in E(\lambda_n)$,
  \begin{align}
   \notag q_a(f_a^u, f_a^u)&=
    (\lambda_n+1)\int_\Omega p u f_a^u \,dx-
    q_a(\Pi_a(\mathcal T_a u), f_a^u)\\
   \notag &=(\lambda_n+1)\int_\Omega p  V^u_a f^u_a\,dx+
      (\lambda_n+1)\int_\Omega p|f^u_a|^2\,dx\\
   \label{eq:st-qff} &\quad +
       (\lambda_n+1)\int_\Omega p  \Pi_a(\mathcal T_a(u)) f^u_a\,dx-
    q_a(\Pi_a(\mathcal T_a( u)), f_a^u).
  \end{align}
  We have
  \begin{multline}\label{eq:st-qff1}
    \left| \int_\Omega p V^u_a f^u_a\,dx+
      \int_\Omega p |f^u_a|^2\,dx\right|=
    \left| \int_\Omega p  (u-\mathcal T_a(u)) f^u_a\,dx+
      \int_\Omega  p |f^u_a|^2\,dx\right|\\
    \leq 
    \tau_a
    \|\mathcal T_a-\Pi_a\circ \mathcal T_a\|_{\mathcal
      L(E(\lambda_n),L^2(\Omega))}\|u\|_{L^2(\Omega,p)}^2
      +    \|\mathcal T_a-\Pi_a\circ \mathcal T_a\|_{\mathcal
      L(E(\lambda_n),L^2(\Omega))}^2\|u\|_{L^2(\Omega,p)}^2.
  \end{multline}
  On the other hand, let us consider a basis $\{u_{n+j-1}^a\}_{j=1}^m$
  of $E_a$, orthonormal with respect to the scalar product
  \eqref{eq:scal-p}, such that $u_{n+j-1}^a$ is an eigenfunction of
  \eqref{eq:eige_a} associated to $\lambda_{n+j-1}^a$. For any fixed
  $u\in E(\lambda_n)$, we define
  \begin{equation*}
    \alpha^a_j=(\Pi_a( \mathcal T_a(u)), u_{n+j-1}^a)_{L^2(\Omega,p)}\quad
    \text{for all $j\in\{1,\dots,m\}$,}
  \end{equation*}
  so that
  \begin{equation*}
  \Pi_a( \mathcal T_a(u))=\sum_{j=1}^m \alpha^a_j
  u_{n+j-1}^a\quad\text{and}\quad
  \|\Pi_a( \mathcal T_a(u))\|_{L^2(\Omega,p)}^2=\sum_{j=1}^m |\alpha^a_j|^2.
\end{equation*}
Hence, by \eqref{eq:expa2},
\begin{align}
\notag  \bigg|(\lambda_n+1)\int_\Omega & p\Pi_a(\mathcal T_a(u)) f^u_a\,dx-
    q_a(\Pi_a(\mathcal T_a( u)),
  f_a^u)\bigg|=\bigg|\sum_{j=1}^m\alpha_j^a(\lambda_n-\lambda_{n+j-1}^a)\int_\Omega
                p f^u_a  u_{n+j-1}^a\,dx\bigg|\\
 \notag &\leq 
    \Big(\max_{j\in\{1,\dots,m\}}|\lambda_n-\lambda_{n+j-1}^a|\Big)\|\Pi_a(\mathcal
    T_a(u)) \|_{L^2(\Omega,p)}\sqrt{m}
    \|f^u_a\|_{L^2(\Omega,p)}\\
  \label{eq:st-qff2} &\leq C
                       \tau_a(1+\tau_a)
                       \|\mathcal T_a-\Pi_a\circ \mathcal T_a\|_{\mathcal
      L(E(\lambda_n),L^2(\Omega))}\|u\|_{L^2(\Omega,p)}^2
\end{align}
for some constant $C>0$ independent of $a$ and $u$, provided $\mathop{\rm
  dist}(a,\partial\Omega)$ is sufficiently small.
From \eqref{eq:st-qff}, \eqref{eq:st-qff1}, \eqref{eq:st-qff2}, \eqref{eq:energy_eige-L2}
it follows that
\begin{align*}
     \|\mathcal T_a-\Pi_a\circ \mathcal T_a&\|_{\mathcal
      L(E(\lambda_n),H^1(\Omega\setminus S_a))}=\sup_{\substack{u\in
                                                 E(\lambda_n)\\
  \|u\|_{L^2(\Omega,p)}=1}}
  \|(\mathcal T_a-\Pi_a\circ \mathcal T_a)u\|_{H^1(\Omega\setminus
  S_a)}\\
  &\leq \frac1{\sqrt{\min\{1,\inf_\Omega p\}}}
  \sup_{\substack{u\in E(\lambda_n)\\\|u\|_{L^2(\Omega,p)}}=1}\sqrt{
  q_a(f_a^u, f_a^u)}=O(\tau_a)
  \quad\text{as }\mathop{\rm dist}(a,\partial\Omega)\to 0,
\end{align*}
thus proving \eqref{eq_energy_eigenfunctions-a}.
\end{proof}

As a first notable consequence of Theorem \ref{theo_asymp_abstract},
we are now in a position to prove Theorem \ref{t:rough_estimate}.
\begin{proof}[Proof of Theorem \ref{t:rough_estimate}]
  It immediately follows by combining \eqref{eq:expa2} with \eqref{eq:est-tau-2}.
\end{proof}

\subsection{The case of simple eigenvalues}
In the case of simple eigenvalues, Theorem \ref{theo_asymp_abstract} can be formulated as follows.
Let $n \in \N\setminus\{0\}$ be such that \eqref{hp_la_0n_simple} is
satisfied, i.e.
$\lambda_n=\lambda_n(\Omega,p)$ is simple
as an eigenvalue of \eqref{eq:eige_lapla}.  The stability result
established in Proposition \ref{prop_spectral_stability} implies that
$\lambda_n^a=\lambda_n^a(\Omega,p)$ is  simple as an eigenvalue of
\eqref{eq:eige_equation_a} as well, provided
$\mathop{\rm dist}(a,\partial\Omega)$ is sufficiently small.

Let $u_n$ be
an eigenfunction associated to $\lambda_n$ such that
\begin{equation*}
 \|u_n\|_{L^2(\Omega,p)}=1. 
\end{equation*}
For every $a\in\Omega$ with small $\mathop{\rm
  dist}(a,\partial\Omega)$, let $v_n^a\in\mathcal H_a(\Omega)\setminus\{0\}$
be the unique eigenfunction of \eqref{eq:eige_a} associated to
$\lambda_n^a$ satisfying
\begin{equation*}
  \int_\Omega p|v_n^a|^2\,dx=1\quad\text{and}\quad
  \int_\Omega pv_n^a u_n\,dx\geq 0.
\end{equation*}
In this case, the orthogonal projection $\Pi_a$ introduced in
\eqref{eq:projection} acts as
\begin{equation*}
  \Pi_a: \  L^2(\Omega)\to \mathcal H_a(\Omega),\quad 
  w \mapsto \left(\int_\Omega p w v_n^a\,dx\right)v_n^a.
\end{equation*}
The following result is a more complete version of Theorem \ref{th:simple}. 

 \begin{corollary}\label{cor:simple}
Let assumption \eqref{hp_la_0n_simple} be satisfied. Then
\begin{equation}
  \label{eq:expa-Esimple}
  \lambda_n^a -\lambda_n=2\mathcal
  E_a^{u_n}+O\left(\|V_a^{u_n}\|_{L^2(\Omega)}^2
  \right)
  \quad \text{as  }\mathop{\rm
    dist}(a,\partial\Omega)\to 0,
\end{equation}
where $V_a^{u_n}$ and $\mathcal E^{u_n}_a$ are introduced in
\eqref{eq:def_Va} and \eqref{eq:E_a} respectively (see also
\eqref{eq:simply-notation}).  Furthermore, as
$\mathop{\rm dist}(a,\partial\Omega)\to 0$,
\begin{align}\label{eq_energy_eigenfunctions-a-simple}
  & \|u_n-V_a^{u_n}-\Pi_a(u_n-V_a^{u_n})\|_{H^1(\Omega\setminus S_a)}=
    O\left(\|V_a^{u_n}\|_{L^2(\Omega)}\right),\\
  &\label{eq_energy_eigenfunctions3-a-simple}\norm{\Pi_a(u_n-V_a^{u_n})}^2_{L^2(\Omega)}
    =1+O\left(\|V_a^{u_n}\|_{L^2(\Omega)}\right).
\end{align}
\end{corollary}
\begin{proof}
  In the case $m=1$, Theorem \ref{theo_asymp_abstract} yields
  \begin{equation}\label{eq:si1}
  \lambda_{n}^a=\lambda_n+\mu_1^a+O(\tau_a^2)\quad 
  \text{as }\mathop{\rm dist}(a,\partial\Omega)\to 0,
\end{equation}
with
\begin{equation}\label{eq:si2}
  \mu_1^a=r_a(u_n,u_n).
\end{equation}
Moreover, \eqref{eq:est-tau-1} implies that
$\tau_a=O\left(\|V_a^{u_n}\|_{L^2(\Omega)}\right)$ as
$\mathop{\rm dist}(a,\partial\Omega)\to 0$, so that
\eqref{eq:expa-Esimple} follows from \eqref{eq:si1}, \eqref{eq:si2},
and Proposition \ref{p:relation_E_r}.

From \eqref{eq_energy_eigenfunctions-a} and \eqref{eq:est-tau-1} it
follows that
\begin{equation*}
  \|u_n-V_a^{u_n}-\Pi_a(u_n-V_a^{u_n})\|_{H^1(\Omega\setminus S_a)}=
  \|(\mathcal T_a-\Pi_a\circ\mathcal T_a)u_n \|_{H^1(\Omega\setminus
    S_a)}=O(\tau_a)
  =O\left(\|V_a^{u_n}\|_{L^2(\Omega)}\right)
\end{equation*}
as $\mathop{\rm dist}(a,\partial\Omega)\to 0$, thus proving \eqref{eq_energy_eigenfunctions-a-simple}.

Finally, \eqref{eq_energy_eigenfunctions-a-simple}
implies that 
\begin{equation*}
  \|u_n
  -\Pi_a(u_n-V_a^{u_n})\|_{L^2(\Omega)}=
  O\left(\|V_a^{u_n}\|_{L^2(\Omega)}\right)
\end{equation*}
as $\mathop{\rm dist}(a,\partial\Omega)\to 0$. From this 
and the identity
\begin{equation*}
  \|\Pi_a(u_n-V_a^{u_n})\|_{L^2(\Omega)}^2=1+
  \|u_n-\Pi_a(u_n-V_a^{u_n})\|_{L^2(\Omega)}^2
  -2\int_{\Omega}  u_n (u_n-\Pi_a(u_n-V_a^{u_n}) \, dx
\end{equation*}
we directly obtain \eqref{eq_energy_eigenfunctions3-a-simple}.
\end{proof}

\section{Pole approaching a fixed boundary point}\label{sec_asymptoic_precise}
In the case of the pole approaching a fixed point of the boundary, we
provide a more precise asymptotic expansion of the eigenvalue
variation than the ones obtained in Theorem \ref{theo_asymp_abstract},
Theorem \ref{t:rough_estimate}, Corollary \ref{cor:simple}.

Throughout this section, we require the following additional
regularity assumption on the domain $\Omega$:
\begin{equation}\label{eq:c1a}
  \text{$\Omega\subset \R^2$ is a planar Jordan
    domain with $C^{1,\alpha}$ boundary},
\end{equation}
for some $\alpha\in(0,1)$. Under assumption \eqref{eq:c1a} we make a
conformal change of variables to locally straighten the boundary of
$\Omega$. Then we perform a gauge transformation as in Section
\ref{sec_preliminaries}. Finally, a rewriting of the problem in
elliptic coordinates allows us to describe the asymptotic behaviour of
the quantities $\{\mu_j^a\}$ appearing in \eqref{eq:expa1}.

  More precisely, we consider a sequence of poles
  $\{a_\ell\}_{\ell \in \mathbb{N}\setminus \{0\}}\subset\Omega$ such that
\begin{equation}\label{eq:sequence}
  a_\ell \to a_0 \quad\text{as $\ell\to\infty$,}\quad  
  \text{ for some } a_0 \in \partial \Omega.
\end{equation}

\subsection{An equivalent problem}\label{subsec_equiv_prob}
The following well-known result about the conformal equivalence of planar
domains follow by combining the classical Riemann mapping theorem
with the Kellogg--Warschawski theorem, see e.g. \cite[Theorem
3.6]{pommerenke}.

\begin{theorem}\label{theor_conf}
  If $\Omega_1$ and $\Omega_2$ are planar Jordan domains with
  $C^{1,\alpha}$ boundaries, for some $\alpha\in(0,1)$, then there
  exists a map $\Phi:\overline{\Omega}_1 \to \overline{\Omega}_2$ such
  that \begin{enumerate}[(i)]
  \item $\Phi\in C^{1,\alpha}(\overline{\Omega}_1)$ and $\Phi$ is
    holomorphic in $\Omega_1$,
\item $\Phi: \overline{\Omega}_1 \to \overline{\Omega}_2$ is a homeomorphism,
\item $\Phi^{-1}\in C^{1,\alpha}(\overline{\Omega}_2)$ and $\Phi^{-1}$
  is holomorphic in $\Omega_2$.
\end{enumerate}
\end{theorem}

By Theorem \ref{theor_conf}, arguing as in
\cite[Lemma~3.1]{NMT-ab-bd}, we can prove the following result.
\begin{proposition}\label{prop_conf}
  Let $\Omega_1$ and $\Omega_2$ be planar Jordan domains with
  $C^{1,\alpha}$ boundaries, for some $\alpha\in(0,1)$.  Let $\Phi$ be
  the biholomorphic map from $\Omega_1$ to $\Omega_2$ given by Theorem
  \ref{theor_conf}.  Then, if $u$ is a solution of
  \eqref{eq:weaksense} in $\Omega=\Omega_1$ with $p$ as in
  \eqref{hp_p}, there exists
  $\chi \in C^{1,\alpha}(\overline{\Omega}_2\setminus\{\Phi(a)\})$
  such that
  \begin{equation}
  w:=e^{-i\chi} u\circ \Phi^{-1} \in
  H^{1,\Phi(a)}(\Omega_2,\C)\label{eq:corr-eig}
\end{equation}
  and
\begin{equation*}
  \int_{\Omega_2 }\left(i\nabla +A_{\Phi(a)} \right)w\cdot
  \overline{\left(i\nabla+ A_{\Phi(a)}\right)\varphi}\, dx
  =\la \int_{\Omega_2} \tilde p\,  w \overline{\varphi} dx \quad
  \text{for all }
  \varphi \in H^{1,\Phi(a)}(\Omega_2,\C),
\end{equation*}
where, letting $\Phi:=(\Phi_1,\Phi_2)$,  $\tilde p \in L^\infty(\Omega_2)$ is given by  
\begin{equation}\label{def_p}
  \tilde
  p(x):=\frac{p(\Phi^{-1}(x))}{\left|\pd{\Phi_1}{x_1}(\Phi^{-1}(x))\right|^2
    +\left|\pd{\Phi_1}{x_2}(\Phi^{-1}(x))\right|^2}
\end{equation}
and satisfies $\tilde p(x)\geq \tilde c>0$ a.e. in $\Omega_2$ for
  some positive constant $\tilde c$.
\end{proposition}
We apply Theorem \ref{theor_conf} and Proposition \ref{prop_conf} by
taking $\Omega_1$ to be the domain $\Omega$ as introduced in
\eqref{eq:c1a}, and $\Omega_2$ to be any domain $\Omega_s$ that
satisfies, in addition to \eqref{eq:c1a}, the following conditions:
\begin{equation}\label{eq:Omega-s}
  \{0\}\times [-1,1] = \partial \Omega_s\cap \{x=(x_1,x_2)\in \R^2:
  x_1=0\} \quad \text{and}
  \quad \Omega_s \subset \{x=(x_1,x_2) \in \R^2: x_1>0 \}.
\end{equation}
Let $\Phi$  be the biholomorphic map from $\Omega$ to $\Omega_s$
provided by
  Theorem \ref{theor_conf}. We observe that, up to suitable rotations, it is possible to choose
  $\Phi$ so that 
\begin{equation}\label{eq:phi-of-a0}
  (0,0)=\Phi(a_0),
\end{equation}
where $a_0\in\partial\Omega$ is a fixed boundary point as in  \eqref{eq:sequence}.

\begin{proposition}\label{p:Phi1asd}
  Let $\Omega\subset\R^2$ be a bounded open set satisfying
  \eqref{eq:c1a} and $\Omega_s$ be as in \eqref{eq:Omega-s}.  Let
  $\{a_\ell\}_{\ell \in \mathbb{N}\setminus \{0\}}\subset\Omega$ and
  $a_0\in\partial\Omega$ satisfy \eqref{eq:sequence}. Let
  $\Phi=(\Phi_1,\Phi_2)$ be the biholomorphic map from $\Omega_1=\Omega$
  to $\Omega_2=\Omega_s$ provided by Theorem \ref{theor_conf}, chosen
  in such a way that \eqref{eq:phi-of-a0} is satisfied. Then
  \begin{equation*}
    \log \Phi_1(a_\ell)=\log \mathop{\rm dist}(a_\ell,\partial\Omega)+
    o\big(\log \mathop{\rm dist}(a_\ell,\partial\Omega)\big)\quad\text{as }\ell\to\infty.
  \end{equation*}
\end{proposition}
\begin{proof}
  According to notation \eqref{eq:simply-notation}, we denote
  $P_{a_\ell}=P_{a_\ell}^\Omega$
  and $d_{a_\ell}=
  \mathop{\rm dist}(a_\ell,\partial\Omega)$. For every $\ell$ sufficiently large,
  let us fix some $Q_{a_\ell}\in\partial\Omega$ such that
  $|P_{a_\ell}-Q_{a_\ell}|=d_{a_\ell}$. Since
  \begin{equation*}
    \lim_{\ell\to\infty}\frac{P_{a_\ell}-a_\ell}{d_{a_\ell}}=\nu(a_0),
  \end{equation*}
  we have
  \begin{equation}\label{eq:tangort}
    \lim_{\ell\to\infty}\frac{Q_{a_\ell}-P_{a_\ell}}{|Q_{a_\ell}-P_{a_\ell}|}\cdot
    \frac{P_{a_\ell}-a_\ell}{d_{a_\ell}}=0.
  \end{equation}
  Moreover, since $\Phi$ is a conformal map, $J_{\Phi}(x)=g(x)R_x$,
  where $g(x)=\det(J_{\Phi}(x))\in C^0(\overline{\Omega},\R)$,
  $g(x)\neq0$ for all $x\in \overline{\Omega}$, and $R_x$ is a
  $2\times2$ orthogonal matrix.  Hence
\begin{align}
\label{eq:expa1n}  &\Phi(Q_{a_\ell})-\Phi(P_{a_\ell})=J_\Phi(P_{a_\ell})(Q_{a_\ell}-P_{a_\ell})+o(|Q_{a_\ell}-P_{a_\ell}|)=
    g(P_{a_\ell})R_{P_{a_\ell}}(Q_{a_\ell}-P_{a_\ell})+o(d_{a_\ell}),\\
 \label{eq:expa2n} &\Phi(a_\ell)-\Phi(P_{a_\ell})=J_\Phi(P_{a_\ell})(a_\ell-P_{a_\ell})+o(|a_\ell-P_{a_\ell}|)=
    g(P_{a_\ell})R_{P_{a_\ell}}(a_\ell-P_{a_\ell})+o(d_{a_\ell}),
\end{align}
as $\ell\to\infty$.  
This implies that there exist constants $0<D_1<D_2$ such that
\begin{equation}\label{eq:uplo}
  D_1d_{a_\ell}\leq |\Phi(Q_{a_\ell})-\Phi(P_{a_\ell})| \leq
  D_2d_{a_\ell}\quad\text{and}\quad
  D_1d_{a_\ell}\leq |\Phi(a_\ell)-\Phi(P_{a_\ell})| \leq
  D_2d_{a_\ell}.
\end{equation}
Furthermore, from \eqref{eq:tangort}, \eqref{eq:expa1n},
\eqref{eq:expa2n}, and \eqref{eq:uplo} it follows that
  \begin{align}
\notag    |\Phi_2(a_\ell)-\Phi_2(P_{a_\ell})|&=\left|
(\Phi(a_\ell)-\Phi(P_{a_\ell}))\cdot \frac{\Phi(Q_{a_\ell})-\Phi(P_{a_\ell})}{|\Phi(Q_{a_\ell})-\Phi(P_{a_\ell})|}
                                   \right|\\
  \label{eq:stimaPhi2}  &=
    \frac{
(g(P_{a_\ell}))^2 (Q_{a_\ell}-P_{a_\ell})\cdot (a_\ell-P_{a_\ell})+o(d_{a_\ell}^2)
    }{|\Phi(Q_{a_\ell})-\Phi(P_{a_\ell})|}=o(d_{a_\ell})
  \end{align}
  as $\ell\to\infty$. In view of \eqref{eq:stimaPhi2} and \eqref{eq:uplo}
  we finally obtain
  \begin{equation*}
    D_2^2d_{a_\ell}^2+o(d_{a_\ell}^2)\geq |\Phi_1(a_\ell)|^2=
    |\Phi(a_\ell)-\Phi(P_{a_\ell})|^2- |\Phi_2(a_\ell)-\Phi_2(P_{a_\ell})|^2
    \geq     D_1^2d_{a_\ell}^2+o(d_{a_\ell}^2),
  \end{equation*}
  as $\ell\to\infty$. This directly implies
  \begin{equation*}
\log d_{a_\ell}+O(1)\geq    \log |\Phi_1(a_\ell)|\geq \log d_{a_\ell}+O(1)
  \end{equation*}
  as $\ell\to\infty$, thus completing the proof.  
\end{proof}

In view of Proposition \ref{prop_conf}, under assumption
  \eqref{eq:c1a}, for every $a\in\Omega$  the eigenvalue problem
  \eqref{eq:eige_equation_a}  is equivalent to the eigenvalue problem
  $\big(E_{\Omega_s,\tilde p}^{\Phi(a)}\big)$, in the sense that
  \begin{equation}\label{eq:equiv-sp-a}
    \lambda_k^a(\Omega,p)= \lambda_k^{\Phi(a)} (\Omega_s,\tilde
    p)\quad\text{for all }k\in\N\setminus\{0\},
  \end{equation}
  with eigenfunctions that are matched through the transformation
  \eqref{eq:corr-eig} and weight $\tilde p$ given by \eqref{def_p}.
  Similarly, by the conformality of the map $\Phi$,
  \eqref{eq:eige_lapla} is equivalent to $(E_{\Omega_s,\tilde p})$, so
  that
\begin{equation}\label{eq:equiv-sp}
  \lambda_k(\Omega,p)= \lambda_k (\Omega_s,\tilde
  p)\quad\text{for all }k\in\N\setminus\{0\},
  \end{equation}
  with eigenfunctions matched through the transformation $\Phi$.

  Let $n,m \in \N\setminus\{0\}$ be such that
  \eqref{lambda-n_multiple-m-bis} holds, i.e. $\lambda_n(\Omega,p)$
  has multiplicity $m$ as an eigenvalue of \eqref{eq:eige_lapla}.
Let $E(\lambda_n(\Omega,p))$ be the associated $m$-dimensional
eigenspace and let $\{u_{n}, u_{n+1},\dots, u_{n+m-1}\}$ be as in
\eqref{eq:basis}, i.e. $\{u_{n}, u_{n+1},\dots, u_{n+m-1}\}$ is a
basis of $E(\lambda_n(\Omega,p))$ orthonormal with respect to the
scalar product \eqref{eq:scal-p}.  Hence
$\lambda_n(\Omega,p)=\lambda_n(\Omega_s,\tilde p)$ has multiplicity
$m$ also as an eigenvalue of $(E_{\Omega_s,\tilde p})$ and the
functions 
\begin{equation}\label{eq:w-u}
  w_{n+i-1}=u_{n+i-1}\circ\Phi^{-1}\in H^1(\Omega_s),\quad i=1,\dots,m,
\end{equation}
form a basis of the associated eigenspace
$E(\lambda_n(\Omega_s,\tilde p))$ orthonormal in
$L^2(\Omega_s,\tilde p)$, i.e.
\begin{equation*}
  \int_{\Omega_s}\tilde p\,  w_{n+i-1} w_{n+j-1}\,dx=
  \begin{cases}
    1,&\text{if }i=j,\\
  0,&\text{if }i\neq j.
  \end{cases}
\end{equation*}
Let
$\{a_\ell\}_{\ell\in\N\setminus\{0\}}$ and $a_0\in\partial\Omega$ be as in
\eqref{eq:sequence}.  From \eqref{eq:phi-of-a0}, we have
\begin{equation*}
  \Phi(a_\ell)\to (0,0)\quad\text{as }\ell\to\infty.
\end{equation*}

\begin{remark}\label{rem:cf}
  Since the boundary of $\Omega_s$ is flat near $(0,0)$ and each
  function $w_{n+i-1}$ in \eqref{eq:w-u} satisfies a Neumann condition
  on $\partial \Omega_s$, it is possible to perform an even reflection
  of $w_{n+i-1}$ with respect to the variable $x_1$, thus obtaining a
  solution of an equation of type $(E_{\Omega_s,\tilde p})$ in a whole
  neighborhood of the point $(0,0)$. By applying well--known results
  on the asymptotic behavior near the zeros of solutions to elliptic
  PDEs (see, for example, \cite{cf}), we deduce that, if an
  eigenfunction $w_{n+i-1}$ vanishes at $(0,0)$, then there cannot be
  other zeros on $\partial \Omega_s$ in a neighborhood of
  $(0,0)$. Therefore, the zeros of the eigenfunction
  $u_{n+i-1}$ on $\partial\Omega$ are
  isolated, and hence they are finite in number.
\end{remark}

From \eqref{eq:equiv-sp-a}, \eqref{eq:equiv-sp}, Theorem
\ref{theo_asymp_abstract}, and \eqref{eq:est-tau-1}, it follows that,
for every $1\leq j\leq m$,
\begin{align}
 \notag \lambda_{n+j-1}^{a_\ell}(\Omega,p) -\lambda_n (\Omega,p)&=
\lambda_{n+j-1}^{\Phi(a_\ell)}(\Omega_s,\tilde p) -\lambda_n
                                                                  (\Omega_s,\tilde p)\\
  \label{eq:first-est-aj}&=
                           \mu_j^{\Phi(a_\ell)}(\Omega_s,\tilde p)+O\bigg(
                           \sum_{j=1}^m\left\|V^{w_{n+j-1}}_{\Phi(a_\ell)}\right\|_{L^2(\Omega_s)}^2\bigg)
                           \quad \text{as  }\ell\to\infty,
\end{align}
with $\{\mu_j^{\Phi(a_\ell)}(\Omega_s,\tilde p)\}_{j=1,\dots,m}$ being
the eigenvalues of the $m\times m$ matrix
\begin{equation*}
  \left(r_{\Omega_s,\tilde p,n}^{\Phi(a_\ell)} ( w_{n+i-1}, w_{n+j-1}) \right)_{ij}.
\end{equation*}
With the aim of identifying the asymptotics of
$r_{\Omega_s,\tilde p,n}^{\Phi(a_\ell)} ( w_{n+i-1}, w_{n+j-1})$ as $\ell\to\infty$, we
first study the behavior of
$\mathcal E^{\Phi(a_\ell),w}_{\Omega_s,\tilde p}$ 
for any $w\in E(\lambda_n(\Omega_s,\tilde p))$.
To this end, we combine an explicit computation in elliptic coordinates 
with a comparison argument with respect to domain inclusions. More
precisely, we are going to prove the following result.
\begin{proposition}\label{p:asyE}
  If $w\in E(\lambda_n(\Omega_s,\tilde p))$ and
  $\int_{\Omega_s}\tilde pw^2\,dx=1$, then
  \begin{equation}\label{eq:asyE1}
    \mathcal E^{\Phi(a_\ell),w}_{\Omega_s,\tilde p}=
    \frac{\pi(w(0))^2}{2|\log \Phi_1(a_\ell)|}+
    o\left(\frac{1}{|\log \Phi_1(a_\ell)|}\right)
\end{equation}
and
\begin{equation}\label{eq:asyE2}
  \big\|V^{\Phi(a_\ell),w}_{\Omega_s,\tilde
    p}\big\|_{L^2(\Omega_s)}^2=
  o\left(\frac{1}{|\log \Phi_1(a_\ell)|}\right)
\end{equation}
as $\ell\to\infty$.
\end{proposition}
The following expansion for
$r_{\Omega_s,\tilde p,n}^{\Phi(a_\ell)} ( w_{n+i-1}, w_{n+j-1})$ will
be deduced as a consequence of Proposition~\ref{p:asyE}.
\begin{proposition}\label{p:asy-r}
  Letting $\{w_{n+i-1}\}_{i=1}^m$ be as in \eqref{eq:w-u}, we have
  \begin{equation*}
    r_{\Omega_s,\tilde p,n}^{\Phi(a_\ell)} ( w_{n+i-1}, w_{n+j-1})=
    \frac{\pi }{|\log \Phi_1(a_\ell)|}\, w_{n+i-1}(0) w_{n+j-1}(0)
    +o\left(\frac{1}{|\log \Phi_1(a_\ell)|}\right)
\end{equation*}
as $\ell\to\infty$, for every $i,j\in\{1,\dots,m\}$.
\end{proposition}

The rest of this section is devoted to the proof of Propositions
\ref{p:asyE} and \ref{p:asy-r}, and to the derivation of Theorem
\ref{t:toward-a-fixed-point} from them.

\subsection{\texorpdfstring{$L^2$}{L2}-norm versus energy of potentials}

As a first step, we note that if the boundary of the domain contains a
segment as in \eqref{eq:Omega-s}, and a sequence of poles $\{b_\ell\}$
is approaching a fixed point on the straight part of the boundary,
then the $L^2$--norm of the potentials attaining
$\mathcal E^{b_\ell,w}_{\Omega_s,\zeta}$ is negligible compared to their
energy.
\begin{proposition}\label{p:negli}
  Let $\Omega_s$ satisfy \eqref{eq:c1a} and \eqref{eq:Omega-s}. Let
  $w\in \mathcal F(\Omega_s)\cap C^1(\overline{\Omega_s})$ with
  $\partial_\nu w=0$ on $\partial\Omega_s$,
  $\zeta\in L^\infty(\Omega_s)$ satisfy
  \eqref{hp_p}, and $\{b_\ell\}\subset\Omega_s$ be such that $b_\ell\to 0$
  as $\ell\to\infty$. Then
  \begin{equation*}
    \big\|V^{b_\ell,w}_{\Omega_s, \zeta}\big\|_{L^2(\Omega_s)}=o
    \left(\big\|\nabla
      V^{b_\ell,w}_{\Omega_s, \zeta}\big\|_{L^2(\Omega_s\setminus
      S_{b_\ell}^{\Omega_s})}\right)\quad\text{as }\ell\to\infty.
  \end{equation*}
\end{proposition}
\begin{proof}
  Arguing by contradiction, assume that, for some subsequence
  $\{b_{\ell_k}\}_k$, with $b_{\ell_k}=(b_{\ell_k,1},b_{\ell_k,2})$, and some constant $C>0$,
     \begin{equation}\label{eq:ass-by-contradiction}
    \Big\|V^{b_{\ell_k},w}_{\Omega_s, \zeta}\Big\|_{L^2(\Omega_s)}>C
    \Big\|\nabla
      V^{b_{\ell_k},w}_{\Omega_s, \zeta}\Big\|_{L^2\big(\Omega_s\setminus
      S_{b_{\ell_k}}^{\Omega_s}\big)}\quad\text{for all }k\in\N.
  \end{equation}
  For every $k\in\N$, let
  \begin{equation*}
    \omega_k:=\{(x_1,x_2)\in\R^2:(x_1,x_2+b_{\ell_k,2})\in\Omega_s\}
  \end{equation*}
  and
  \begin{equation*}
    W_k:\omega_k\to\R,\quad 
    W_k(x_1,x_2)= V^{b_{\ell_k},w}_{\Omega_s, \zeta}(x_1,x_2+b_{\ell_k,2}).
  \end{equation*}
  Since $V^{b_{\ell_k},w}_{\Omega_s, \zeta}\in H^1(\Omega_s\setminus
  S^{\Omega_s}_{b_{\ell_k}})$ and, if $k$ is sufficiently large, $S^{\Omega_s}_{b_{\ell_k}}=\{(t,
  b_{\ell_k,2}):0\leq t\leq b_{\ell_k,1}\}$, we
  have
  \begin{equation*}
    W_k\in H^1(\omega_k\setminus s_k)
  \end{equation*}
  where $s_k=\{(t,0): 0\leq t\leq b_{\ell_k,1}\}$. Moreover, since
  $b_{\ell_k}\to 0$ as $k\to+\infty$, there exist $R_1,R_2,r_0>0$ such that
  \begin{equation*}
    \Omega_s,\omega_k\subset R=(0,R_1)\times
    (-R_2,R_2)\quad\text{and}\quad
s_k\subset s:=\{(t,0): 0\leq t\leq r_0\}\subset\Omega_s\cup\{0\}
\end{equation*}
for $k$ sufficiently large.  Since the domains $\omega_k$ are small
vertical translations of the $C^{1,\alpha}$ fixed domain $\Omega_s$, it is easily seen that there
exist uniformly bounded extension operators, i.e., for every $k$
sufficiently large there exists $\mathcal P_k:H^1(\omega_k\setminus
s)\to H^1(R\setminus s)$ such that 
\begin{equation}\label{eq:uniform-ext}
  \mathcal P_k(v)\big|_{\omega_k\setminus s} =v \quad \text{ and } \quad
  \|\mathcal P_k(v)\|_{H^1(R\setminus s)}\leq K\|v\|_{H^1(\omega_k\setminus
    s)}
\end{equation}
for some positive constant $K$ that does not depend on
$k$. Furthermore, denoting with  $|\cdot|$ the Lebesgue measure in
$\R^2$, we have
\begin{equation}\label{eq:meas-to-zero}
\lim_{k\to\infty} \Big( |\omega_k\setminus\Omega_s|+  |\Omega_s \setminus\omega_k|\Big)=0,
\end{equation}
due to the fact that $b_{\ell_k}\to 0$ as $k\to+\infty$.

For $k$ sufficiently large, let
\begin{equation*}
  V_k=\frac{\mathcal P_k(W_k)}{\|W_k\|_{L^2(\omega_k)}}.
\end{equation*}
From \eqref{eq:ass-by-contradiction} and \eqref{eq:uniform-ext} it
follows that $\{V_k\}_k$ is uniformly bounded in $H^1(R\setminus
s)$. Hence there exists $V\in H^1(R\setminus s)$ such that, along a
subsequence still denoted as $\{V_k\}_k$, $V_k\rightharpoonup V$
weakly in $H^1(R\setminus s)$. Since $V_k\in H^1(R\setminus s_k)$,
Proposition \ref{prop_mosco} implies that $V\in H^1(R)$. Moreover, by
compactness of the embedding
$H^1(R\setminus s)\hookrightarrow \hookrightarrow L^2(R)$,
\begin{align*}
  1&=\int_{\omega_k}V_k^2\,dx=\int_{\Omega_s}V_k^2\,dx+\int_{\omega_k\setminus\Omega_s}V_k^2\,dx
     -\int_{\Omega_s\setminus \omega_k}V_k^2\,dx\\
  &=\int_{\Omega_s}V^2\,dx+o(1)\quad\text{as }k\to\infty,
\end{align*}
where we have used the fact, for every $\tau\in(2,+\infty)$, 
\begin{align*}
  \left| \int_{\omega_k\setminus\Omega_s}V_k^2\,dx
    -\int_{\Omega_s\setminus \omega_k}V_k^2\,dx\right|&\leq \Big(
  |\omega_k\setminus\Omega_s|+ |\Omega_s
  \setminus\omega_k|\Big)^{\frac{\tau-2}\tau}\|V_k\|_{L^\tau(R)}^2\\
  & \leq \mathop{\rm const} \Big( |\omega_k\setminus\Omega_s|+
  |\Omega_s
  \setminus\omega_k|\Big)^{\frac{\tau-2}\tau}\|V_k\|_{H^1(R\setminus
    s)}^2=o(1) \quad\text{as }k\to\infty,
\end{align*}
by H\"older's inequality, Sobolev's embeddings, and 
\eqref{eq:meas-to-zero}. In particular, we obtain
\begin{equation}
  \label{eq:Vnontriv}
  \int_{\Omega_s}V^2\,dx=1.
\end{equation}
Let $\varphi \in H^1(\Omega_s)$ be such that $\varphi\equiv0$ in a
neighborhood of $0$.
In particular, in view of \eqref{eq:cara_La2}, we have
\begin{equation}\label{eq:Lzero}
  L^{b_{\ell_k},w}_{\Omega_s}(\varphi)=0\quad\text{provided $k$ is sufficiently
    large}.
\end{equation}
Let $\tilde\varphi\in H^1(R)$ be an extension of $\varphi$ to $R$.  We
also denote as $\tilde \zeta$ the trivial extension of
$\zeta$ in $\R^2\setminus\Omega_s$.  We observe that
$\tilde \zeta (\cdot, \cdot+b_{\ell_k,2})$ converges to $\tilde \zeta$ weakly* in
$L^\infty(\R^2)$, thus implying
\begin{equation}\label{eq:fromweakstar}
  \int_{\Omega_s}\tilde \zeta (x_1,x_2+b_{\ell_k,2})
  V_k\varphi\,dx\to
   \int_{\Omega_s}\tilde \zeta
  V\varphi\,dx\quad\text{as }k\to\infty.
\end{equation}
From the definition of $V_k$, \eqref{eq:eq_Va}, \eqref{eq:Lzero},
and \eqref{eq:fromweakstar},  it
follows that
\begin{align*}
 0&= \int_{\omega_k\setminus s_k}\Big(\nabla
  V_k\cdot\nabla\tilde\varphi+\zeta (x_1,x_2+b_{\ell_k,2})
    V_k\tilde\varphi\Big)\,dx\\
  &=\int_{\Omega_s\setminus s_k}\Big(\nabla
  V_k\cdot\nabla\varphi+\tilde \zeta (x_1,x_2+b_{\ell_k,2})
  V_k\varphi\Big)\,dx+
  \int_{\omega_k\setminus \Omega_s}\Big(\nabla
  V_k\cdot\nabla\tilde\varphi+\zeta (x_1,x_2+b_{\ell_k,2})
    V_k\tilde\varphi\Big)\,dx\\
&\quad-\int_{\Omega_s\setminus \omega_k}\Big(\nabla
  V_k\cdot\nabla\tilde\varphi+\tilde \zeta (x_1,x_2+b_{\ell_k,2})
                                  V_k\tilde\varphi\Big)\,dx\\
  &=\int_{\Omega_s}\big(\nabla
  V\cdot\nabla\varphi+\zeta
  V\varphi\big)\,dx+o(1)\quad\text{as }k\to\infty,
\end{align*}
because
\begin{align*}
&\left|   \int_{\omega_k\setminus \Omega_s}\Big(\nabla
  V_k\cdot\nabla\tilde\varphi+\zeta (x_1,x_2+b_{\ell_k,2})
    V_k\tilde\varphi\Big)\,dx-\int_{\Omega_s\setminus \omega_k}\Big(\nabla
  V_k\cdot\nabla\tilde\varphi+\tilde \zeta (x_1,x_2+b_{\ell_k,2})
                 V_k\tilde\varphi\Big)\,dx\right|\\
  &\leq \mathop{\rm const}\|V_k\|_{H^1(R\setminus s)}\big(
    \|\tilde\varphi\|_{H^1(\Omega_s\setminus \omega_k)}+
    \|\tilde\varphi\|_{H^1(\omega_k\setminus
    \Omega_s)}\Big)=o(1)\quad\text{as }k\to\infty
\end{align*}
in view of the boundedness of $\{V_k\}$ in $H^1(R\setminus s)$ and
\eqref{eq:meas-to-zero}.

Hence, we obtain
\begin{equation*}
  \int_{\Omega_s}\big(\nabla
  V\cdot\nabla\varphi+\zeta
  V\varphi\big)\,dx=0,
\end{equation*}
for every $\varphi \in H^1(\Omega_s)$ such that $\varphi\equiv0$ in a
neighborhood of $0$. Since the singleton $\{0\}$ has null capacity in
$\R^2$, the above identity holds for every
$\varphi \in H^1(\Omega_s)$, thus implying $V\equiv 0$ in $\Omega_s$
and contradicting \eqref{eq:Vnontriv}.
\end{proof}

\subsection{Reduction to the case with constant jump across the
  segment}
If the boundary of the domain includes a straight segment as in
\eqref{eq:Omega-s}, and the poles ${b_\ell}$ converge to a fixed point on
that segment, then $\mathcal E^{b_\ell,w}_{\Omega_s, \zeta}$ can be
approximated with $\mathcal E^{b_\ell,\kappa}_{\Omega_s, \zeta}$, where
$\kappa=w(0)$.

\begin{proposition}\label{p:confronto-costante}
  Let $\Omega_s$ satisfy \eqref{eq:c1a} and \eqref{eq:Omega-s}. Let
  $w\in H^1(\Omega_s)$ be an eigenfunction of $(E_{\Omega_s,\zeta})$ for
  some $\zeta\in L^\infty(\Omega_s)$ satisfying
  \eqref{hp_p}, and $\kappa:=w(0)$.
  Let $\{b_\ell\}\subset\Omega_s$ be such that $b_\ell\to 0$
  as $\ell\to\infty$. Then
  \begin{equation*}
    \mathcal E^{b_\ell,w}_{\Omega_s, \zeta}=
    \mathcal E^{b_\ell,\kappa}_{\Omega_s, \zeta}+
    o\left(\frac1{|\log\mathop{\rm
          dist}(b_\ell,\partial\Omega_s)|}\right)
    \quad\text{as }\ell\to\infty.
  \end{equation*}
\end{proposition}
\begin{proof}
  Recalling the notations introduced in \eqref{eq:def_Va} and
  \eqref{eq:E_a}, we set, for every $\ell\in\N$,
  \begin{equation*}
    Y_\ell:=V^{b_\ell,w}_{\Omega_s, \zeta},\quad 
    Z_\ell:=V^{b_\ell,\kappa}_{\Omega_s, \zeta},\quad\text{and}\quad
    Q_\ell:=Y_\ell-Z_\ell.
  \end{equation*}
  We also denote
  \begin{equation*}
    \tilde \eta_\ell:=\eta_{d_{b_\ell}}(x-P_{b_\ell})
  \end{equation*}
  where $\eta_{d_{b_\ell}}$ is defined in \eqref{eq:def-cutoff}.  In view
  of \eqref{eq:eq_Va} we have
  \begin{equation*}
    q^{b_\ell}_{\Omega_s,\zeta}(Q_\ell,Q_\ell-\tilde \eta_\ell(w-\kappa))=
     L^{b_\ell,w}_{\Omega_s}(Q_\ell-\tilde \eta_\ell(w-\kappa)),
   \end{equation*}
   hence, in view of \eqref{eq:cara_La2},
   \begin{align}
   \notag & \int_{\Omega_s\setminus S_{b_\ell}^{\Omega_s}}|\nabla Q_\ell|^2\,dx=-
     \int_{\Omega_s}\zeta Q_\ell^2\,dx+
     \int_{\Omega_s\setminus S_{b_\ell}^{\Omega_s}}\nabla Q_\ell\cdot
                                                  \nabla(\tilde \eta_\ell(w-\kappa))\,dx\\
     &\label{eq:estQ1}+
     \int_{\Omega_s}\zeta Q_\ell \tilde \eta_\ell(w-\kappa)\,dx
     -2\int_{ S_{b_\ell}^{\Omega_s}}
(\nabla w\cdot \nu_{b_\ell}^{\Omega_s})
       \gamma_{b_\ell,\Omega_s}^+(Q_\ell)\,dS
       +2\int_{ S_{b_\ell}^{\Omega_s}}
(\nabla w\cdot \nu_{b_\ell}^{\Omega_s})
   \tilde \eta_\ell(w-\kappa)\,dS.
   \end{align}
From Proposition \ref{p:negli}, \eqref{eq:Vato0}, and
\eqref{eq:const-case} it follows that
\begin{equation}\label{eq:Qj}
   \int_{\Omega_s}\zeta Q_\ell^2\,dx=o\left(\frac1{|\log\mathop{\rm
         dist}(b_\ell,\partial\Omega_s)|}\right)
   \quad\text{as }\ell\to\infty.
\end{equation}
From \eqref{eq:def-cutoff}, \eqref{eq:int-cut},
\eqref{eq:const-case}, \eqref{eq:Vato0}, 
and
the fact that
\begin{equation*}
  w(x)=\kappa+O(|x|)\quad\text{as }|x|\to0,
\end{equation*}
we obtain
\begin{align}
\notag\int_{\Omega_s\setminus S_{b_\ell}^{\Omega_s}}\nabla Q_\ell\cdot
\nabla(\tilde
\eta_\ell(w-\kappa))\,dx&=
\notag\int_{\Omega_s\setminus
  S_{b_\ell}^{\Omega_s}}(w-\kappa)\nabla Q_\ell\cdot\nabla\tilde
\eta_\ell\,dx+ \int_{\Omega_s\setminus
  S_{b_\ell}^{\Omega_s}}\tilde \eta_\ell\nabla Q_\ell
                       \cdot   \nabla w\,dx\\
\notag &=o\left(\frac1{|\log\mathop{\rm
    dist}(b_\ell,\partial\Omega_s)|}\right)+\sqrt{\mathop{\rm
         dist}(b_\ell,\partial\Omega_s)}
    O\left(\frac1{\sqrt{|\log \mathop{\rm
    dist}(b_\ell,\partial\Omega_s)|}}\right)\\
  &\label{eq:estQ2}=o\left(\frac1{|\log\mathop{\rm
         dist}(b_\ell,\partial\Omega_s)|}\right)
   \quad\text{as }\ell\to\infty.
\end{align}
In view of \eqref{eq:Qj} and \eqref{eq:def-cutoff}   we have
\begin{equation}\label{eq:estQ3}
  \int_{\Omega_s}\zeta Q_\ell \tilde \eta_\ell(w-\kappa)\,dx =o\left(\frac1{|\log\mathop{\rm
         dist}(b_\ell,\partial\Omega_s)|}\right)
   \quad\text{as }\ell\to\infty.
\end{equation}
Furthermore, in view of \eqref{eq:reg_u_n}, $w$ and $\nabla w\cdot
\nu_{b_\ell}^{\Omega_s}$
are uniformly bounded on
   $S_{b_\ell}^{\Omega_s}$, so that
   \begin{equation}\label{eq:estQ4}
\int_{ S_{b_\ell}^{\Omega_s}}
(\nabla w\cdot \nu_{b_\ell}^{\Omega_s})
   \tilde \eta_\ell(w-\kappa)\,dS=O(\mathop{\rm
         dist}(b_\ell,\partial\Omega_s))=o\left(\frac1{|\log\mathop{\rm
         dist}(b_\ell,\partial\Omega_s)|}\right)
   \quad\text{as }\ell\to\infty,
 \end{equation}
 and, by well-known trace theorems, \eqref{eq:Vato0}, and
\eqref{eq:const-case}
 \begin{align}
\notag   \int_{ S_{b_\ell}^{\Omega_s}}
(\nabla w\cdot \nu_{b_\ell}^{\Omega_s})
       \gamma_{b_\ell,\Omega_s}^+(Q_\ell)\,dS&=O\Big(\sqrt{\mathop{\rm
         dist}(b_\ell,\partial\Omega_s)}\|Q_\ell\|_{H^1(\Omega_s\setminus
       S_{b_\ell}^{\Omega_s})}\Big)\\
\label{eq:estQ5}&     =o\left(\frac1{|\log\mathop{\rm
         dist}(b_\ell,\partial\Omega_s)|}\right)
   \quad\text{as }\ell\to\infty.
 \end{align}
  Combining \eqref{eq:estQ1} with \eqref{eq:Qj}, \eqref{eq:estQ2},
  \eqref{eq:estQ3},  \eqref{eq:estQ4}, and \eqref{eq:estQ5}, we
  finally obtain
  \begin{equation*}
    \int_{\Omega_s\setminus S_{b_\ell}^{\Omega_s}}|\nabla Q_\ell|^2\,dx=o\left(\frac1{|\log\mathop{\rm
         dist}(b_\ell,\partial\Omega_s)|}\right)
   \quad\text{as }\ell\to\infty.
  \end{equation*}
For the above estimate, \eqref{eq:Vato0}, and
\eqref{eq:const-case} it follows that
\begin{align}
 \notag \|\nabla Z_\ell\|^2_{L^2(\Omega_s\setminus S_{b_\ell}^{\Omega_s})} -
\|\nabla Y_\ell\|^2_{L^2(\Omega_s\setminus
  S_{b_\ell}^{\Omega_s})}&=\int_{\Omega_s\setminus S_{b_\ell}^{\Omega_s}}
(\nabla Z_\ell -\nabla Y_\ell)  \cdot
                        (\nabla Z_\ell +\nabla Y_\ell) \, dx\\
 \notag &=-\int_{\Omega_s\setminus S_{b_\ell}^{\Omega_s}}
\nabla Q_\ell\cdot  
    (\nabla Z_\ell +\nabla Y_\ell) \, dx\\
 \label{eq:fi1} &=o\left(\frac1{|\log\mathop{\rm
         dist}(b_\ell,\partial\Omega_s)|}\right)
   \quad\text{as }\ell\to\infty.
\end{align}
Furthermore, in view of \eqref{eq:cara_La2}, well-known trace theorems,
and \eqref{eq:Vato0},
\begin{align}
\notag  L^{b_\ell,w}_{\Omega_s}(Y_\ell)
&=-\int_{S_{b_\ell}^{\Omega_s}}(\nabla w\cdot \nu_{b_\ell}^{\Omega_s})
   (\gamma_{b_\ell,\Omega_s}^+(Y_\ell)-\gamma_{b_\ell,\Omega_s}^-(Y_\ell))=O\Big(\sqrt{\mathop{\rm
         dist}(b_\ell,\partial\Omega_s)}\|Y_\ell\|_{H^1(\Omega_s\setminus
       S_{b_\ell}^{\Omega_s})}\Big)\\
 \label{eq:fi2} &=o\left(\frac1{|\log\mathop{\rm
         dist}(b_\ell,\partial\Omega_s)|}\right)
   \quad\text{as }\ell\to\infty.
\end{align}
From \eqref{eq:fi1}, \eqref{eq:fi2}, Proposition \ref{p:negli}, \eqref{eq:Vato0}, and
\eqref{eq:const-case} it follows that
\begin{align*}
   \mathcal E^{b_\ell,w}_{\Omega_s, \zeta}-
   \mathcal E^{b_\ell,\kappa}_{\Omega_s, \zeta}&=\frac12\left(
     \|\nabla Y_\ell\|^2_{L^2(\Omega_s\setminus S_{b_\ell}^{\Omega_s})} -
\|\nabla Z_\ell\|^2_{L^2(\Omega_s\setminus
                                              S_{b_\ell}^{\Omega_s})}\right)\\
  &\quad +\frac12 \left(\int_{\Omega_s}\zeta
Y_\ell^2\,dx-
    \int_{\Omega_s}\zeta Z_\ell^2\,dx\right)-L^{b_\ell,w}_{\Omega_s}(Y_\ell)\\
  &=o\left(\frac1{|\log\mathop{\rm
         dist}(b_\ell,\partial\Omega_s)|}\right)
   \quad\text{as }\ell\to\infty.
\end{align*}
The proof is thereby complete.  
\end{proof}

\subsection{The case of ellipses}

Passing to elliptic coordinates, a very explicit estimate for
$\mathcal E^{a,\kappa}_{\omega, \zeta}$ can be derived when $\omega$
is a half-ellipse whose axis lies along the segment $S_{a}^{\omega}$.

For every $L>0$ and $\e\in(0,L/2)$, let us consider the halved ellipse defined as
\begin{equation*}
  E_\e(L):=\left\{(x_1,x_2)\in \R^2:
    x_1>0\text{ and }\frac{x_1^2}{L^2+\e^2}+\frac{x^2_2}{L^2}<1\right\}
\end{equation*}
and the segment
\begin{equation*}
  s_\e=[0,\e]\times\{0\}.
\end{equation*}
Let us  introduce the elliptic coordinates $(\xi, \eta)$:
\begin{equation}\label{def_elliptic_coor}
\begin{cases}
x_1=\e\cosh \xi \cos \eta, \\
x_2=\e\sinh\xi \sin \eta, \\
\end{cases}
\quad \xi\ge 0, \, \eta \in [-\pi/2,\pi/2],
	\end{equation} 
see e.g. \cite[Section 2.2]{AFHL} or \cite[Proposition
7.2]{FNOS_multi}.
In these coordinates, the segment  $s_\e$  is described by the
conditions 
\begin{equation}\label{proof_prop_Kj_1}
 \xi=0,\quad \eta\in [-\pi/2,\pi/2],
\end{equation}
 while the half-ellipse $E_\e(L)$ is characterized by 
\begin{equation*}
\xi\in [0,\xi_\e),\quad \eta\in (-\pi/2,\pi/2),
\end{equation*}
where $\xi_\e$ is such that $\e\sinh(\xi_\e)=L$, that is 
\begin{equation}\label{def_xi_j}
  \xi_\e=\mathop{\rm{arcsinh}}\left(\frac{L}{\e}\right)=
  \log\left(\frac{L}{\e}+\sqrt{1+\frac{L^2}{\e^2}}\right).
\end{equation}
In particular, $\partial E_\e(L)$ is described by the conditions 	
\begin{equation*}
  \xi=\xi_\e\text{ and } \eta \in (-\pi/2,\pi/2),
  \quad \text{or} \quad \xi \in [0,\xi_\e]\text{ and } \eta=-\pi/2,
  \quad \text{ or } \quad \xi \in [0,\xi_\e]\text{ and } \eta=\pi/2.
\end{equation*}
The map 
\begin{align*}
  &F_\e:[0,\xi_\e) \times (-\pi/2,\pi/2) \to E_\e(L), \\
  &F_\e(\xi,\eta)=(x_1,x_2)
    =\Big(\e\cosh\xi \cos \eta, \e\sinh\xi \sin \eta\Big),
\end{align*}
defined by \eqref{def_elliptic_coor},  has a Jacobian matrix of the form 
\begin{equation*}
J_{F_\e}(\xi,\eta)=\e \sqrt{\cosh^2\xi-\cos^2\eta}\, \,O(\xi,\eta)
\end{equation*}
for some orthogonal matrix	$O(\xi,\eta)$,  while   
$\det J_{F_\e}(\xi,\eta)=\e^2(\cosh^2\xi-\cos^2\eta)$. 
In particular, $F_\e$ is a conformal map.

Let us define, in elliptic coordinates, the weight 
\begin{equation*}
\rho:[0,\xi_\e)\times (-\pi/2,\pi/2)\to\R,\quad 
\rho (\xi, \eta):=\frac{e^{2\xi}}{\cosh^2\xi-\cos^2\eta},
\end{equation*}
which  corresponds in Cartesian coordinates to the function
\begin{equation}\label{def_q}
  \rho_\e:E_\e(L)\to\R,\quad \rho_\e(x_1,x_2):=(\rho\circ F^{-1}_\e)(x_1,x_2).
\end{equation}
For any  $\alpha\in \R\setminus\{0\}$,   let us consider the minimization problem
\begin{equation*}
  \inf\left\{
    \int_{E_\e(L) \setminus s_\e}\big(|\nabla  w|^2 +\rho_\e w^2\big)\,
    dx:w \in H^1(E_\e(L) \setminus s_\e), \ 
    \gamma^+_\e(w-\alpha)+\gamma^-_\e(w-\alpha)=0 \text{ on } s_\e\right\},
\end{equation*}
where $\gamma^+_\e:=\gamma^+_{(\e,0), E_\e(L)}$ and
$\gamma^-_\e:=\gamma^-_{(\e,0), E_\e(L)}$ according to the notation
introduced in \eqref{eq:trace}.  We observe that the functional
minimized above is well defined in $H^1(E_\e(L) \setminus s_\e)$
because $\rho_\e\in L^r(E_\e(L))$ for some $r>1$. Moreover, for every
$\e\in(0,L/2)$ and $(\xi,\eta)\in [0,\xi_\e)\times(-\pi/2,\pi/2)$,
\begin{equation*}
  \rho_\e(\e\cosh(\xi) \cos (\eta),\e\sinh(\xi) \sin (\eta))=\rho(\xi,\eta)
  =\frac{e^{2\xi}}{\cosh^2\xi-\cos^2\eta} \ge
  \frac{e^{2\xi}}{\cosh^2\xi}
  \ge \frac{1}{\cosh^2\xi_\e}>0.
\end{equation*}
Hence, by a standard variational argument, there exists a unique
minimizer
\begin{equation}\label{eq:defWe}
  W_\e=W_{\e,\alpha,L} \in H^1(E_\e(L) \setminus s_\e),
\end{equation}
which satisfies
\begin{equation}\label{eq_Wj}
  \int_{E_\e(L) \setminus s_\e} \big(
  \nabla W_\e\cdot \nabla w +\rho_\e W_\e w \big)
  \, dx =0 \quad\text{for all }w \in \mathcal H_{(\e,0)}(E_\e(L)),
\end{equation}
see \eqref{eq:defHa}, 
i.e. $W_\e$  weakly solves  the problem
\begin{equation}\label{prob_Wj}
\begin{cases}
-\Delta W_\e+\rho_\e W_\e =0, &\text{in } E_\e(L)\setminus s_\e,\\[2pt]
\gamma^+_\e(W_\e-\alpha)+\gamma^-_\e(W_\e-\alpha)=0, & \text {on }  s_\e,\\[2pt]
\gamma^+_\e\big(\frac{\partial W_\e}{\partial x_2}\big)+
\gamma^-_\e\big(\frac{\partial W_\e}{\partial x_2}\big)
=0, & \text{on } s_\e,\\[2pt]
\frac{\partial W_\e}{\partial \nu}=0, & \text{on } \partial E_\e(L),
\end{cases}
\end{equation}
where $\nu$ is the outer normal vector to $\partial E_\e(L)$.

\begin{proposition}\label{prop_Wj}
We have 
\begin{equation}\label{eq_E_Wj}
\int_{E_\e(L)\setminus s_\e} |\nabla W_\e|^2 \, dx=
\frac{\pi\alpha^2}{|\log \e|}+o\left(\frac{1}{|\log\e|}\right) \quad
\text{as } \e \to 0^+,
\end{equation}
and 
\begin{equation}\label{eq_qjWj_o}
  \int_{E_\e(L)} \rho_\e W_\e^2 \, dx=o\bigg(\int_{E_\e(L)\setminus s_\e}
    |\nabla W_\e|^2 \, dx
  \bigg) \quad \text{as }  \e\to0^+.
\end{equation}
\end{proposition}
\begin{proof}
  Let us define $\widehat{W}_\e:=W_\e\circ F_\e$, with $W_\e$ being
  the solution of \eqref{prob_Wj}. The Neumann boundary condition
  $\pd{W_\e}{\nu}=0$ in \eqref{prob_Wj} is then equivalent to
\begin{equation}\label{proof_prop_Kj_2}
  \pd{\widehat{W}_\e}{\xi}(\xi_\e,\eta)=0 \text{ for every }  \eta \in
  (-\pi/2,\pi/2)
  \quad \text{ and } \quad \pd{\widehat{W}_\e}{\eta}(\xi,\pm\pi/2)=0
  \text{ for every }  \xi \in (0,\xi_\e).
\end{equation}
Indeed, the curve $\eta\mapsto\big(\sqrt{L^2+ \e^2}\cos (\eta), L\sin(\eta)\big)$
parametrizes the curved part of the boundary of the half-ellipse
$E_\e(L)$, so that a normal vector to such curved boundary is given by
$\big( L\cos (\eta), \sqrt{L^2+ \e^2}\sin(\eta)\big)$.
Hence, 
\begin{align*}
\pd{\widehat{W}_\e}{\xi}(\xi_\e,\eta)&=\pd{W_\e}{x_1}\e \sinh(\xi_\e)
\cos(\eta)
+\pd{W_\e}{x_2}\e \cosh(\xi_\e) \sin(\eta)\\
&=\pd{W_\e}{x_1}L \cos(\eta)+\pd{W_\e}{x_2} \sqrt{L^2+\e^2}\sin(\eta)=0.
\end{align*}
The second condition in \eqref{proof_prop_Kj_2} can be obtain in a
similar way.

Since $F_\e$ is conformal, by \eqref{prob_Wj} and \eqref{proof_prop_Kj_1},
$\widehat{W}_\e$ turns out to solve 
\begin{equation}\label{prob_hatWj}
\begin{cases}
  -\Delta \widehat{W}_\e+\e^2 e^{2\xi} \widehat{W}_\e=0,
  &\text{in } (0,\xi_\e)\times (-\pi/2,\pi/2),\\[2pt]
  \widehat{W}_\e (0,\eta)+\widehat{W}_\e(0,-\eta)= 2 \alpha , &
  \text{for every  }  \eta \in (-\pi/2,\pi/2),\\[2pt]
  \pd{\widehat{W}_\e}{\xi}(0,\eta)-\pd{\widehat{W}_\e}{\xi}(0,-\eta)=0,
  & \text{for every  }  \eta \in (-\pi/2,\pi/2),\\[2pt]
  \pd{\widehat W_\e}{\xi}(\xi_\e, \eta)=0, & \text{for every }
  \eta \in (-\pi/2,\pi/2),\\[2pt]
  \pd{\widehat W_\e}{\eta}(\xi, \pm\pi/2)=0, & \text{for every } \xi
  \in (0,\xi_\e).
\end{cases}
\end{equation}
We expand $\widehat{W}_\e$ in Fourier series with respect to $\eta$:
\begin{equation*}
\widehat{W}_\e(\xi,\eta)=\frac{a_{0,\e}(\xi)}{2}+\sum_{h=1}^\infty
\Big(a_{h,\e}(\xi)
\cos(h\eta)+b_{h,\e}(\xi)\sin(h \eta)\Big),
\end{equation*}
where 
\begin{align*}
&a_{h,\e}(\xi):=\frac{1}{\pi}\int_{-\pi/2}^{\pi/2}\widehat{W}_\e(\xi,\eta) \cos(h \eta) \, d\eta \quad \text{for all } h\in \mathbb{N}, \\
&b_{h,\e}(\xi):=\frac{1}{\pi}\int_{-\pi/2}^{\pi/2}\widehat{W}_\e(\xi,\eta) \sin(h \eta) \, d\eta \quad \text{for all } h \in \mathbb{N}\setminus\{0\}.
\end{align*}
The boundary conditions in \eqref{prob_hatWj} imply
\begin{align} \label{proof_prop_Kj_3}
&a_{h,\e}(0)=0 \text{ for every } h \in \mathbb{N}\setminus \{0\},  \quad a'_{h,\e}(\xi_\e)=0 \text{ for every } h \in \mathbb{N},\\
\label{eq:coebor2}&b_{h,\e}'(0)=0 \text{ for every } h \in \mathbb{N}\setminus \{0\}, \quad  b'_{h,\e}(\xi_\e)=0 \text{ for every } h \in \mathbb{N}\setminus \{0\},\\
\label{eq:coebor3} &a_{0,\e}(0)=2 \alpha,
\end{align}
while the equation $-\Delta \widehat{W}_\e+\e^2e^{2\xi} \widehat{W}_\e=0$ yields
\begin{equation}\label{eq:eq-coef}
\begin{cases}
  a_{0,\e}''-\e^2 e^{2\xi} a_{0,\e}=0, &\text{in }(0,\xi_\e),\\
  a_{h,\e}''-(h^2+\e^2 e^{2\xi}) a_{h,\e}=0, &\text{in
  }(0,\xi_\e),\text{
    for all } h\in \mathbb{N}\setminus \{0\},\\
  b_{h,\e}''-(h^2+\e^2e^{2\xi}) b_{h,\e}=0, & \text{in }(0,\xi_\e),
  \text{ for all } h\in \mathbb{N}\setminus \{0\}.\\
\end{cases}
\end{equation}
Hence, in view of conditions \eqref{proof_prop_Kj_3}--\eqref{eq:coebor2},
we necessarily have $a_{h,\e}\equiv b_{h,\e}\equiv 0$ in
$(0,\xi_\e)$ for all  $h \in \mathbb{N}\setminus \{0\}$. 

In order to determine $a_{0,\e}$, let us consider
\begin{equation*}
  \tilde a_\e(t):=a_{0,\e}\big(\log (t/\e)\big).
\end{equation*}
Since 
\begin{align*}
&\tilde{a}'_{\e}(t)=t^{-1}a'_{0,\e}(\log(t/\e)),\\
&\tilde{a}''_{\e}(t)=-t^{-2}a'_{0,\e}(\log(t/\e)))+t^{-2}a''_{0,j\e}(\log(t/\e)),
\end{align*}
by \eqref{eq:eq-coef} $\tilde{a}_{\e}$ solves the modified Bessel equation 
\begin{equation*}
t^2 \tilde{a}''_{\e}+t \tilde{a}'_{\e}-t^2\tilde{a}_{\e}=0 \quad
\text{in } (\e, \e e^{\xi_\e}).
\end{equation*}
Hence, there exist two constants $c_1,c_2$, depending only on $\e$ and
$L$, such that
\begin{equation*}
\tilde{a}_{\e}(t)=c_1 I_0(t)+c_2 K_0(t),
\end{equation*}
where $I_0$ and $K_0$ are the  $0$-order modified Bessel functions 
of first and second
kind, respectively, see e.g. \cite[Chapter 3, Section
3.7]{W_Bessel_book}.  It follows that
\begin{equation*}
a_{0,\e}(\xi)=c_1 I_0(\e e^\xi)+c_2 K_0(\e e^\xi),\quad \xi\in(0,\xi_\e).
\end{equation*} 
Moreover, since $I_0'=I_1$ and $K_0'=-K_1$, see e.g.  \cite[Chapter 3,
Section 3.71]{W_Bessel_book},
\begin{equation*}
a'_{0,\e}(\xi)=\e e^\xi\big(c_1 I_1(\e e^\xi)-c_2 K_1(\e e^\xi)\big).
\end{equation*} 
The boundary conditions in \eqref{proof_prop_Kj_3} and
\eqref{eq:coebor3} imply 
\begin{align*}
&c_1=\frac{2\alpha K_1(\e e^{\xi_\e}) }{I_0(\e) K_1(\e
                 e^{\xi_\e})+K_0(\e) I_1(\e e^{\xi_\e})},\\
  &c_2=\frac{2\alpha I_1(\e e^{\xi_\e}) }{I_0(\e) K_1(\e
                 e^{\xi_\e})+K_0(\e) I_1(\e e^{\xi_\e})}.
\end{align*}
In conclusion, we have
\begin{equation*}
  \widehat{W}_\e(\xi,\eta)= \frac12 a_{0,\e}(\xi)=
  \alpha\, \frac{K_1(\e e^{\xi_\e})  I_0(\e e^\xi)+I_1(\e e^{\xi_\e})
    K_0(\e e^\xi)}
{I_0(\e) K_1(\e e^{\xi_\e})+K_0(\e) I_1(\e e^{\xi_\e})}.
\end{equation*}
We observe that the last boundary condition in \eqref{prob_hatWj} is obviously satisfied.

Let us study the asymptotic behavior as $\e\to0^+$ of $c_1$ and
$c_2$. By \cite[Chapter 3, Sections 3.7 and 3.71]{W_Bessel_book} we have
\begin{equation*}
I_0(\xi)=1+o(1) \quad \text{and}  \quad K_0(\xi)=|\log(\xi)|+o(\log(\xi)), \quad \text{as } \xi \to 0^+.
\end{equation*}
By \eqref{def_xi_j}
\begin{equation*}
\e e^{\xi_\e}=2L+o(1)\quad \text{as } \e \to 0^+, 
\end{equation*}
so that
\begin{align*}
  &c_1=\frac{2\alpha K_1(2L) }{I_1(2L)}\frac{1}{|\log\e|}+o\left(\frac{1}{|\log\e|}\right)  \text{ as }  \e \to 0^+,\\
  &c_2= \frac{2\alpha}{|\log\e|}+o\left(\frac{1}{|\log\e|}\right)  \text{ as }  \e\to0^+.
\end{align*}
By conformality of $F_\e$ and the change of variables $t=\e e^\xi$,
\begin{align*}
  \int_{E_\e(L)\setminus s_\e} |\nabla W_\e|^2 \, dx&= \int_0^{\xi_\e}\int_{-\pi/2}^{\pi/2} |\nabla \widehat W_\e|^2 \, d\eta\, d \xi\\
                                                    & =\frac\pi4 \int_0^{\xi_\e} |a'_{0,\e}(\xi)|^2 \, d\xi=
                                                      \frac\pi4
                                                      \e^2\int_0^{\xi_\e}
                                                      e^{2\xi}|c_1
                                                      I_1(\e
                                                      e^\xi)-c_2
                                                      K_1(\e e^\xi)|^2 \, d\xi\\
 & =\frac\pi4 \int_\e^{L+\sqrt{L^2+\e^2}} t\,
  |c_1 I_1(t)-c_2 K_1(t)|^2 \, dt\\
&  = \frac\pi4 c_1^2\int_\e^{L+\sqrt{L^2+\e^2}} t |I_1(t)|^2 \, dt + \frac\pi4 c_2^2\int_\e^{L+\sqrt{L^2+\e^2}}
  t |K_1(t)|^2 \, dt\\
  & \quad -\frac\pi2 c_1c_2\int_\e^{L+\sqrt{L^2+\e^2}} t I_1(t)K_1(t)
  \, dt.
\end{align*}
By \cite[Chapter 3, Sections 3.7 and 3.71]{W_Bessel_book},
\begin{equation*}
I_1(t)=O(t) \quad \text{ and }  \quad K_1(t)=\frac1t+ \frac{t}{2}\log
t+o(t\log t) \quad \text{as } t \to 0^+.
\end{equation*}
Hence 
\begin{equation*}
  \int_{E_\e(L)\setminus s_\e} |\nabla W_\e|^2 \, dx=
  \frac{\pi\alpha^2}{|\log\e|}+o\left(\frac{1}{|\log\e|}\right) \quad \text{as } \e \to 0^+,
\end{equation*}
thus proving \eqref{eq_E_Wj}.
Similarly, taking into account \eqref{def_q},
\begin{align*}
  \int_{E_\e(L)} &\rho_\e W_\e^2 \, dx
                   =\pi\e^2  \int_0^{\xi_\e} e^{2\xi}\widehat W_\e^2 \, d \xi\\
                                &=\frac{\pi}{4}\e^2 \int_0^{\xi_\e}
                                  e^{2\xi} |a_{0,\e}(\xi)|^2 \, d\xi
                                  =\frac\pi 4\e^2\int_0^{\xi_\e}
                                  e^{2\xi}|c_1 I_0(\e e^\xi)+c_2
                                  K_0(\e e^\xi)|^2 \, d\xi\\
&=\frac\pi4 \int_\e^{L+\sqrt{L^2+\e^2}} t |c_1 I_0(t)+c_2 K_0(t)|^2 \, dt\\
&= \frac\pi4 c_1^2\int_\e^{L+\sqrt{L^2+\e^2}} t |I_0(t)|^2 \, dt
                                + \frac\pi4 c_2^2\int_\e^{L+\sqrt{L^2+\e^2}} t |K_0(t)|^2 \, dt\\
  &\quad+ \frac\pi2 c_1c_2\int_\e^{L+\sqrt{L^2+\e^2}} t I_0(t)K_0(t) \, dt\\
                                &=
                                  \pi\alpha^2\frac{1+o(1)}{|\log\e|^2}\left(
                                  \int_0^{2L}t\left( \frac{|K_1(2L)|^2}{|I_1(2L)|^2}|I_0(t)|^2+2\frac{K_1(2L)}{I_1(2L)}I_0(t)K_0(t)+|K_0(t)|^2\right) \, dt+o(1)\right)\\
&=O\left(\frac1{|\log\e|^2}\right)= o\left(\int_{E_\e(L)\setminus
                                                                                                                                                                            s_\e} |\nabla W_\e|^2 \, dx\right)\quad  \text{as }\e\to0^+.
\end{align*}
Hence,  \eqref{eq_qjWj_o} is proved.
\end{proof}

For some fixed $L>0$, $\alpha\in\R\setminus\{0\}$, and $\beta>0$, we
consider, for every $\e\in(0,\frac L2)$, the quantity $\mathcal
E_{E_\e(L),\beta}^{(\e,0),\alpha}$, which, according to \eqref{eq:E_a},
is given by
\begin{equation*}
 \mathcal
E_{E_\e(L),\beta}^{(\e,0),\alpha}= \min\left\{
    \frac12\int_{E_\e(L) \setminus s_\e}\!\!\big(|\nabla  w|^2 +\beta w^2\big)\,
    dx:\hskip-10pt
    \begin{array}{ll}&\phantom{a}\\&w \in H^1(E_\e(L) \setminus s_\e),\\[5pt]  
    &\gamma^+_\e(w-\alpha)+\gamma^-_\e(w-\alpha)=0 \text{ on } s_\e
       \end{array}
\right\},
\end{equation*}
and the unique minimizer attaining it, i.e., according to the notation
introduced in \eqref{eq:eq_Va}, 
\begin{equation}\label{eq:defZe}
  Z_\e=Z_{\e,\alpha,\beta,L}:=V_{E_\e(L),\beta}^{(\e,0),\alpha}.
\end{equation}
Since $L^{(\e,0),\alpha}_{E_\e(L)}\equiv0$, by
\eqref{eq:eq_Va} we have that $Z_\e \in H^1(E_\e(L) \setminus s_\e)$
satisfies $Z_\e-\alpha\in \mathcal H_{(\e,0)}(E_\e(L))$ and 
\begin{equation}\label{eq_Zj}
  \int_{E_\e(L) \setminus s   _\e} \big(
  \nabla Z_\e\cdot \nabla w +\beta Z_\e w \big)
  \, dx =0 \quad\text{for all }w \in \mathcal H_{(\e,0)}(E_\e(L)).
\end{equation}
Equivalently, $Z_\e$ is the unique weak solution to the problem 
\begin{equation*}
\begin{cases}
-\Delta Z_\e+\beta Z_\e =0, &\text{ in } E_\e(L)\setminus s_\e,\\[2pt]
\gamma^+_\e(Z_\e-\alpha)+\gamma^{-}_\e(Z_\e-\alpha)=0, & \text{ on }  s_\e,\\[2pt]
\gamma^+_\e\big(\frac{\partial Z_\e}{\partial x_2}\big)+
\gamma^-_\e\big(\frac{\partial Z_\e}{\partial x_2}\big)=0, & \text{ on } s_\e,\\[2pt]
\frac{\partial Z_\e}{\partial \nu}=0, & \text{on } \partial E_\e(L).
\end{cases}
\end{equation*}
To compare the asymptotic behavior of $W_\e$ and $Z_\e$, we need some
preliminary lemmas.

\begin{lemma}\label{lemma_Zjo}
For every $b \in [1,+\infty)$
\begin{equation}\label{eq_Zjo}
  \norm{Z_\e}_{L^b(E_\e(L))}= o(\norm{\nabla
    Z_\e}_{L^2(E_\e(L)\setminus s_\e)})
  \quad \text{as } \e\to0^+. 
\end{equation}
\end{lemma}
\begin{proof}
  It is not restrictive to assume $b \ge 2$. We argue by
  contradiction, assuming that there exists a constant $c>0$ and a
  sequence $\e_j\to 0^+$ such that
\begin{equation*}
  \|\nabla Z_{\e_j}\|_{L^2(E_{\e_j}(L) \setminus s_{\e_j})} < c\,
  \|Z_{\e_j}\|_{L^b(E_{\e_j}(L))}
\end{equation*}
for every $j \in \mb{N}$.

One can easily prove that there exists a sequence of uniformly bounded
extension operators
\begin{equation*}
\Psi_j:H^1(E_{\e_j}(L) \setminus s_{\e_j}) \to H^1(E_1(L)\setminus
s_1),
\end{equation*}
satisfying, for every $w\in H^1(E_{\e_j}(L) \setminus s_{\e_j})$, 
\begin{equation*}
  \Psi_j(w)\Big|_{E_{\e_j}(L)} =w \quad \text{and}
  \quad \|\Psi_j (w)\|_{ H^1(E_1(L)\setminus
    s_1)}\le \tilde c\,\|w\|_{H^1(E_{\e_j}(L) \setminus s_{\e_j}) }
\end{equation*}
for some positive constant $\tilde c$ that does not depend on $j$.
Let us define 
\begin{equation*}
\tilde{Z}_{j}:=\frac{\Psi_j(Z_{\e_j})}{\|Z_{\e_j}\|_{L^b(E_{\e_j}(L))}}.
\end{equation*}
Then $\{\tilde{Z}_{j}\}_{j \in \mb{N}}$ is bounded in
$H^1(E_1(L)\setminus s_1)$ and
$\|\tilde{Z}_{j}\|_{L^b(E_{\e_j}(L))}=1$ for all $j$.  In view of
Proposition \ref{prop_mosco}, up to passing to a subsequence, there
exists $\tilde Z \in H^1(E_1(L))$ such that
$\tilde{Z}_{\e_j} \rightharpoonup \tilde Z$ weakly in
$H^1(E_1(L)\setminus s_1)$.

In view of \eqref{eq_Zj}, for every
$\varphi \in C^\infty(\overline{E_1(L)})$ such that $\varphi\equiv 0$
in a neighborhood of $0$, we have
\begin{equation}\label{eq_Zjtp}
  \int_{E_{\e_j}(L)} \big(
  \nabla \tilde Z_j\cdot \nabla \varphi +\beta \tilde Z_j \varphi \big)
  \, dx =0 \quad\text{for all $j$ sufficiently large}.
\end{equation}
Letting $D_L^+=\{(x_1,x_2)\in\R^2:x_1^2+x_2^2< L^2\text{ and }x_1>0\}$, we have
\begin{equation*}
  \int_{E_{\e_j}(L)} 
  \nabla \tilde Z_j\cdot \nabla \varphi \, dx=
    \int_{E_{\e_j}(L)\setminus D_L^+} 
  \nabla \tilde Z_j\cdot \nabla \varphi \, dx+   \int_{D_L^+} 
  \nabla \tilde Z_j\cdot \nabla \varphi \, dx
\end{equation*}
and 
\begin{equation*}
  \left|
    \int_{E_{\e_j}(L)\setminus D_L^+} 
    \nabla \tilde Z_j\cdot \nabla \varphi \, dx\right|
  \le \|\nabla \varphi\|_{L^\infty(E_1(L))}
  \|\nabla \tilde Z_j\|_{L^2(E_1(L) \setminus s_1)} |E_{\e_j}(L)\setminus D_L^+|^{\frac{1}{2}},
\end{equation*}
where $|\cdot|$ denotes the Lebesgue measure in $\R^2$. Since,
$\lim_{j\to\infty}|E_{\e_j}(L)\setminus D_L^+|=0$, we conclude that
\begin{equation*}
\lim_{j\to\infty}  \int_{E_{\e_j}(L)} 
  \nabla \tilde Z_j\cdot \nabla \varphi \, dx=  \int_{D_L^+} 
  \nabla \tilde Z\cdot \nabla \varphi \, dx.
\end{equation*}
Similarly, 
\begin{equation*}
\lim_{j\to\infty}  \int_{E_{\e_j}(L)} 
  \tilde Z_j \varphi \, dx=  \int_{D_L^+} 
  \tilde Z \varphi \, dx.
\end{equation*}
Therefore, we can pass to the limit in \eqref{eq_Zjtp}, thus obtaining
\begin{equation*}
  \int_{D_L^+} \big(
  \nabla \tilde Z\cdot \nabla \varphi +\beta \tilde Z \varphi \big)
  \, dx =0 
\end{equation*}
for every $\varphi \in C^\infty(\overline{E_1(L)})$ such that
$\varphi\equiv 0$ in a neighborhood of $0$. Since a singleton has null
capacity in $\R^2$, it can be easily proved that the set of smooth
functions vanishing in the neighborhood of a point are dense in $H^1$, so the
above identity holds true for every $\varphi\in H^1(E_1(L))$. It
follows that
\begin{equation}\label{eq:restr-null}
  \tilde Z\Big|_{D_L^+}\equiv0.
\end{equation}
On the other hand
\begin{equation*}
1=\|\tilde{Z}_{j}\|_{L^b(E_{\e_j}(L))}^b=\int_{D_L^+}
|\tilde{Z}_{j}|^b \, dx +
\int_{E_{\e_j}(L)\setminus D_L^+}  |\tilde{Z}_{j}|^b \, dx,
\end{equation*}
and, by compactness of the Sobolev embedding 
$H^1(E_1(L)\setminus s_1)\hookrightarrow L^b(E_1(L))$ and \eqref{eq:restr-null},
\begin{equation*}
  \lim_{j \to \infty}\int_{E_{\e_j}(L)\setminus D_L^+}  |\tilde{Z}_{j}|^b \, dx= 0 \quad 
\text{ and } \quad 
\lim_{j \to \infty}\int_{D_L^+}  |\tilde{Z}_{j}|^b \, dx= \int_{D_L^+} |\tilde{Z}|^b \, dx =0,
\end{equation*}
thus giving rise to a contradiction.
\end{proof}

\begin{lemma}\label{lemma_q}
Let $\rho_\e$ be as in \eqref{def_q}. Then, for every $\gamma \in (1,2 )$,
\begin{equation}\label{ineq_q}
  \|\rho_\e\|_{L^\gamma(E_\e(L))}= O(1) \quad    \text{and}
  \quad \|\rho^{-1}_\e\|_{L^\gamma(E_\e(L))} = O(1)
\end{equation}
as $\e\to0^+$.
\end{lemma}
\begin{proof}
Let $\gamma \in (1,2)$. Passing to elliptic coordinates 
\begin{align}
\notag\|\rho_\e&\|_{L^\gamma(E_\e(L))}^{\gamma}= \int_{E_\e(L)} \rho_\e^{\gamma} \, dx
=\e^2 \int_0^{\xi_\e}\int_{-\pi/2}^{\pi/2} \frac{e^{2 \gamma
                                            \xi}}{(\cosh^2\xi-\cos^2\eta )^{\gamma-1}} \,d \eta \,d \xi\\
\label{eq:nqO} & =\e^2 \int_0^{1}\int_{-\pi/2}^{\pi/2} \frac{e^{2 \gamma
                                            \xi}}{(\cosh^2\xi-\cos^2\eta )^{\gamma-1}} \,d \eta \,d\xi +\e^2 \int_1^{\xi_\e}\int_{-\pi/2}^{\pi/2} \frac{e^{2 \gamma
                                            \xi}}{(\cosh^2\xi-\cos^2\eta )^{\gamma-1}} \,d \eta \,d \xi.
\end{align}
A  Taylor expansion yields
\begin{equation*}
  \cosh^2\xi = 1+ \xi^2+o(\xi^2) \text{ as }\xi\to0,\quad
  \cos^2 \eta= 1-\eta^2+o(\eta^2) \text{ as }\eta\to0,
\end{equation*}
hence
\begin{equation*}
  \cosh^2\xi-\cos^2\eta\geq \frac12(\xi^2+\eta^2)\quad\text{in a
    sufficiently small neighborhood of $(0,0)$}.
\end{equation*}
Hence the function
$(\xi,\eta)\mapsto \frac{e^{2 \gamma \xi}}{(\cosh^2\xi-\cos^2\eta
  )^{\gamma-1}}$ is integrable in
$(0,1)\times (-\frac\pi2, \frac\pi2)$ provided $\gamma<2$, thus
implying that the first term at the right hand side of \eqref{eq:nqO}
is bounded uniformly with respect to $\e$.

Moreover, $\cosh^2\xi-\cos^2\eta\geq \cosh^2\xi-1\geq \frac18
e^{2\xi}$ for every $\xi\geq1$ and $\eta\in  (-\frac\pi2,
\frac\pi2)$, so that
\begin{equation*}
  \e^2 \int_1^{\xi_\e}\int_{-\pi/2}^{\pi/2} \frac{e^{2 \gamma   \xi}}
  {(\cosh^2\xi-\cos^2\eta)^{\gamma-1}} \,d \eta \,d\xi
  \leq \pi 8^{\gamma-1}\e^2\int_1^{\xi_\e}e^{2\xi}\,d\xi=\frac\pi2
  8^{\gamma-1}\e^2
  \Big(\Big(\tfrac{L}{\e}+\sqrt{1+\tfrac{L^2}{\e^2}}\Big)^2-e^2\Big)
\end{equation*}
and also the second term at the right hand side of \eqref{eq:nqO} is
bounded uniformly with respect to $\e$. The first estimate in
\eqref{ineq_q} is thereby proved.

On the other
hand,
$\cosh^2\xi-\cos^2\eta\leq \cosh^2\xi\leq e^{2\xi}$ for
every $\xi\geq0$ and $\eta\in (-\frac\pi2, \frac\pi2)$, so that,
in view of \eqref{def_xi_j},
\begin{align*}
\|\rho_\e^{-1}\|_{L^\gamma(E_\e(L))}^{\gamma}&= \int_{E_\e(L)} \rho_\e^{-\gamma} \, dx
=\e^2\int_0^{\xi_\e}\int_{-\pi/2}^{\pi/2}
\frac{(\cosh^2\xi-\cos^2\eta)^{1+\gamma}}{e^{2\gamma\xi}}\,
d \eta \,d \xi \\
&\le
\pi \e^2 \int_0^{\xi_\e}e^{2\xi}\,d\xi= O(1) \quad  \text{ as } j \to \infty.
\end{align*}
The second estimate in \eqref{ineq_q} is thereby proved.
\end{proof}

We are now able to make an asymptotic comparison between the energies
of $W_\e$  and $Z_\e$.
\begin{proposition}\label{prop_WZ}
  Let $L>0$, $\alpha\in\R\setminus\{0\}$, and $\beta>0$. For every
  $\e\in(0,L/2)$, let $W_\e=W_{\e,\alpha,L}$ be as in
  \eqref{eq:defWe}--\eqref{eq_Wj} and $Z_\e=Z_{\e,\alpha,\beta,L}$ be
  as in \eqref{eq:defZe}--\eqref{eq_Zj}. Then 
\begin{equation}\label{eq_WZ}
  \|\nabla Z_\e\|^2_{L^2(E_\e(L) \setminus s_\e)}=
  \|\nabla W_\e\|^2_{L^2(E_\e(L) \setminus s_\e)}(1+o(1))
\quad  \text{as } \e\to0^+.
\end{equation}
\end{proposition}
\begin{proof}
Testing \eqref{eq_Zj} with $Z_\e- W_\e$  we obtain 
\begin{equation*}
  \int_{E_\e(L) \setminus s_\e} |\nabla Z_\e|^2dx =
  \int_{E_\e(L) \setminus s_\e} \big(\nabla W_\e\cdot \nabla  Z_\e+\beta
  Z_\e(W_\e-Z_\e)\big)
  \, dx.
\end{equation*}
Similarly, testing \eqref{eq_Wj} with  $W_\e-Z_\e$ we find
\begin{equation*}
\int_{E_\e(L) \setminus s_\e} |\nabla W_\e|^2\, dx = \int_{E_\e(L)
  \setminus s_\e}\big(
\nabla W_\e\cdot \nabla Z_\e+ \rho_\e W_\e(Z_\e-W_\e) \big)\, dx.
\end{equation*}
Hence, taking the difference of the above identities,
\begin{equation}\label{eq:diff}
\left|\int_{E_\e(L) \setminus s_\e} |\nabla W_\e|^2dx-\int_{E_\e(L)
    \setminus s_\e}
  |\nabla Z_\e|^2\, dx\right|\\
\le \int_{E_\e(L)} \big(
\beta Z_\e^2+ \rho_\e W_\e^2 +(\beta + \rho_\e) |W_\e| |Z_\e|\big) \, dx.
\end{equation}
Furthermore, by the H\"older inequality,
\begin{align}\notag
  \int_{E_\e(L)} |W_\e| |Z_\e| \, dx 
  &\le  \left(\int_{E_\e(L)} \rho_\e|W_\e|^2\,dx\right)^{\!\!1/2}
    \left(\int_{E_\e(L)} \frac{|Z_\e|^2}{\rho_\e}\,dx\right)^{\!\!1/2} \\
 \label{eq:WeZe1} &\le \left(\int_{E_\e(L)} \rho_\e|W_\e|^2\,dx\right)^{\!\!1/2}
    \left(\int_{E_\e(L)} \rho_\e^{-\frac{3}{2}}\,dx\right)^{\!\!1/3}
    \left(\int_{E_\e(L)} |Z_\e|^6\,dx\right)^{\!\!1/6}
\end{align}
and 
\begin{align}\notag
\int_{E_\e(L)}  \rho_\e|W_\e| |Z_\e| \, dx 
&\le  \left(\int_{E_\e(L)} \rho_\e|W_\e|^2\,dx\right)^{\!\!1/2}
\left(\int_{E_\e(L)} \rho_\e|Z_\e|^2\,dx\right)^{\!\!1/2} \\
\label{eq:WeZe2}&\le \left(\int_{E_\e(L)} \rho_\e|W_\e|^2\,dx\right)^{\!\!1/2}
\left(\int_{E_\e(L)} \rho_\e^{\frac{3}{2}}\,dx\right)^{\!\!1/3}
\left(\int_{E_\e(L)} |Z_\e|^6\,dx\right)^{\!\!1/6}.
\end{align}
Combining \eqref{eq:diff}, \eqref{eq:WeZe1}, and \eqref{eq:WeZe2}, in view
of \eqref{eq_qjWj_o}, \eqref{eq_Zjo} with $b=2$ and $b=6$, and
\eqref{ineq_q}  with $\gamma=3/2$, we conclude that
\begin{equation*}
  \|\nabla W_\e\|_{L^2(E_\e(L) \setminus s_\e)}^2-
  \|\nabla Z_\e\|_{L^2(E_\e(L) \setminus s_\e)}^2=
  o\left( \|\nabla W_\e\|_{L^2(E_\e(L) \setminus s_\e)}^2\right)+
  o\left( \|\nabla Z_\e\|_{L^2(E_\e(L) \setminus s_\e)}^2\right)
\end{equation*}
as $\e\to0^+$, thus proving \eqref{eq_WZ}.
\end{proof}

Combining Propositions \ref{prop_WZ} and \ref{prop_Wj} we obtain the
following asymptotic expansion for $\mc E_{E_\e(L),\beta}^{(\e,0),\alpha}$.

\begin{corollary}\label{cor:E_ellipse}
  For every $L>0$, $\alpha\in\R\setminus\{0\}$, and $\beta>0$, we have
\begin{equation*}
 \mathcal
 E_{E_\e(L),\beta}^{(\e,0),\alpha}=
\frac{\pi\alpha^2}{2|\log \e|}+o\left(\frac{1}{|\log\e|}\right) \quad
\text{as } \e \to 0^+.
\end{equation*} 
\end{corollary}
\begin{proof}
  From \eqref{eq_WZ} and \eqref{eq_E_Wj} it follows that
  \begin{equation*}
    \|\nabla Z_\e\|_{L^2(E_\e(L) \setminus s_\e)}^2=\frac{\pi\alpha^2}{|\log \e|}+o\left(\frac{1}{|\log\e|}\right)
  \end{equation*}
as $\e\to0^+$,  so that, in view of \eqref{eq_Zjo} with $b=2$, we conclude that
  \begin{equation*}
    \mathcal
 E_{E_\e(L),\beta}^{(\e,0),\alpha}=\frac12\|\nabla Z_\e\|_{L^2(E_\e(L)
   \setminus s_\e)}^2+\frac{\beta}2 \|Z_\e\|_{L^2(E_\e(L)}^2=\frac{\pi\alpha^2}{2|\log \e|}+o\left(\frac{1}{|\log\e|}\right),
  \end{equation*}
as $\e\to0^+$.  
\end{proof}

\subsection{Proof of Proposition \ref{p:asyE}, Proposition
  \ref{p:asy-r}, and Theorem \ref{t:toward-a-fixed-point}}
We are now ready to prove Proposition \ref{p:asyE}, and, as a
consequence, Proposition \ref{p:asy-r}.
\begin{proof}[Proof of Proposition \ref{p:asyE}]
  Let $w\in E(\lambda_n(\Omega_s,\tilde p))$ with
  $\int_{\Omega_s}\tilde pw^2\,dx=1$ and $\kappa=w(0)$. Let
  $\e_\ell=\Phi_1(a_\ell)$ and $\tau_\ell=\Phi_2(a_\ell)$, so that, by
  \eqref{eq:Omega-s},
  $\mathop{\rm dist}(\Phi(a_\ell),\partial\Omega_s)=\e_\ell$ provided
  $\ell$ is sufficiently large. By Proposition
  \ref{p:confronto-costante} we have
  \begin{equation}\label{eq:pr-est}
\mathcal E^{\Phi(a_\ell),w}_{\Omega_s,\tilde p}=
\mathcal E^{\Phi(a_\ell),\kappa}_{\Omega_s,\tilde p}+o\left(\frac{1}{|\log\e_\ell|}\right)
    \quad\text{as }\ell\to\infty.
  \end{equation}
  If $w(0)=\kappa=0$, then
  $\mathcal E^{\Phi(a_\ell),\kappa}_{\Omega_s,\tilde p}=0$ and
  expansion \eqref{eq:asyE1} directly follows from
  \eqref{eq:pr-est}. In the case $\kappa\neq0$, let us fix
  $\beta_1,\beta_2>0$ such that
  \begin{equation*}
    \beta_1\leq \tilde p\leq \beta_2\quad\text{a.e. in }\Omega_s.
  \end{equation*}
  Hence, in view of Remark \ref{rem:comparison-weight},
  \begin{equation}\label{eq:comp1}
    \mathcal E^{\Phi(a_\ell),\kappa}_{\Omega_s,\beta_1}\leq
    \mathcal E^{\Phi(a_\ell),\kappa}_{\Omega_s,\tilde p}\leq
    \mathcal E^{\Phi(a_\ell),\kappa}_{\Omega_s,\beta_2}.
  \end{equation}
  It can be easily proved that there exist $L_1,L_2>0$  such that, for
  all $\ell$ sufficiently large,
  \begin{equation*}
    E_{\e_\ell}(L_1)+\tau_\ell\mathbf{e_2}\subset\Omega_s\subset  E_{\e_\ell}(L_2)+\tau_\ell\mathbf{e_2},
  \end{equation*}
where $\mathbf{e_2}=(0,1)$. Hence, from \eqref{eq:comp1} we deduce
that 
 \begin{align*}
 \mathcal E^{(\e_\ell,0),\kappa}_{E_{\e_\ell}(L_1),\beta_1}
   =     \mathcal
   E^{\Phi(a_\ell),\kappa}_{E_{\e_\ell}(L_1)+\tau_\ell\mathbf{e_2},\beta_1}\leq
   \mathcal E^{\Phi(a_\ell),\kappa}_{\Omega_s,\beta_1}&\leq
   \mathcal E^{\Phi(a_\ell),\kappa}_{\Omega_s,\tilde p}\\
 &  \leq
    \mathcal E^{\Phi(a_\ell),\kappa}_{\Omega_s,\beta_2}\leq \mathcal
    E^{\Phi(a_\ell),\kappa}_{E_{\e_\ell}(L_2)+\tau_\ell\mathbf{e_2},\beta_2}
    = \mathcal E^{(\e_\ell,0),\kappa}_{E_{\e_\ell}(L_2),\beta_2},
 \end{align*}
  provided $\ell$ is sufficiently large. Expansion \eqref{eq:asyE1} then follows from
  \eqref{eq:pr-est} and
  Corollary \ref{cor:E_ellipse}.

  Moreover, Proposition \ref{p:negli} and \eqref{eq:Vato0}
  yield
  \begin{equation*}
    \big\|V^{\Phi(a_\ell),w}_{\Omega_s,\tilde
      p}\big\|_{L^2(\Omega_s)}^2=
    o\big\|\nabla
V^{\Phi(a_\ell),w}_{\Omega_s,\tilde
      p}\big\|^2_{L^2(\Omega_s\setminus
      S_{\Phi(a_\ell)}^{\Omega_s})}=o\left(\frac{1}{|\log
        \Phi_1(a_\ell)|}\right)
    \quad\text{as }\ell\to\infty,
  \end{equation*}
  thus proving \eqref{eq:asyE2}.
\end{proof}
\begin{proof}[Proof of Proposition \ref{p:asy-r}]
  By the bilinearity of the form
  $ r_{\Omega_s,\tilde p,n}^{\Phi(a_\ell)}$ and Proposition
  \ref{p:relation_E_r} we have
  \begin{align*}
    &r_{\Omega_s,\tilde p,n}^{\Phi(a_\ell)} ( w_{n+i-1}, w_{n+j-1})\\
    &=
    \frac12\left(
r_{\Omega_s,\tilde p,n}^{\Phi(a_\ell)}
\Big(\tfrac{w_{n+i-1}+w_{n+j-1}}{\sqrt2},
\tfrac{w_{n+i-1}+w_{n+j-1}}{\sqrt2}\Big)-
r_{\Omega_s,\tilde p,n}^{\Phi(a_\ell)}
\Big(\tfrac{w_{n+i-1}-w_{n+j-1}}{\sqrt2},
\tfrac{w_{n+i-1}-w_{n+j-1}}{\sqrt2}\Big)
      \right)\\
    &=\mathcal E^{\Phi(a_\ell), (w_{n+i-1}+w_{n+j-1})/\sqrt2}_{\Omega_s,\tilde p}-
\mathcal E^{\Phi(a_\ell),
      (w_{n+i-1}-w_{n+j-1})/\sqrt2}_{\Omega_s,\tilde p}\\
    &\quad +
O\left( \left\|V^{\Phi(a_\ell), (w_{n+i-1}+w_{n+j-1})/\sqrt2}_{\Omega_s,\tilde
      p}\right\|_{L^2(\Omega_s)}^2\right)+
O\left( \left\|V^{\Phi(a_\ell), (w_{n+i-1}-w_{n+j-1})/\sqrt2}_{\Omega_s,\tilde
    p}\right\|_{L^2(\Omega_s)}^2\right)
  \end{align*}
  as $\ell\to\infty$.
  Hence, from \eqref{eq:asyE1} and \eqref{eq:asyE2} we conclude 
  \begin{align*}
    r_{\Omega_s,\tilde p,n}^{\Phi(a_\ell)} ( w_{n+i-1}, w_{n+j-1})&=
    \frac{\pi }{2|\log \Phi_1(a_\ell)|}\, \left(
    \left(\tfrac{w_{n+i-1}(0) +w_{n+j-1}(0)}{\sqrt2}\right)^2-
    \left(\tfrac{w_{n+i-1}(0) -w_{n+j-1}(0)}{\sqrt2}\right)^2\right)\\
                                                                  &\quad    +o\left(\frac{1}{|\log \Phi_1(a_\ell)|}\right)\\
    &=
    \frac{\pi }{|\log \Phi_1(a_\ell)|}\, w_{n+i-1}(0) w_{n+j-1}(0)
    +o\left(\frac{1}{|\log \Phi_1(a_\ell)|}\right)
\end{align*}
as $\ell\to\infty$, thus completing the proof.
\end{proof}

We are finally in a position to prove Theorem
\ref{t:toward-a-fixed-point}.
\begin{proof}[Proof of Theorem \ref{t:toward-a-fixed-point}]
  Let $\{a_\ell\}\subset\Omega$ be a sequence of points such that $a_\ell\to
  a_0\in\partial\Omega$.
  Let $\Omega_s$ be as in \eqref{eq:Omega-s} and 
  $\Phi=(\Phi_1,\Phi_2)$ be the biholomorphic map from $\Omega$
  to $\Omega_s$ provided by Theorem \ref{theor_conf}, chosen
  as in \eqref{eq:phi-of-a0}. By 
\eqref{eq:first-est-aj}, Proposition
  \ref{p:asy-r}, and \eqref{eq:asyE2}, for every $j=1,\dots,m$  we have 
\begin{equation*}
  \lambda_{n+j-1}^{a_\ell}(\Omega,p) -\lambda_n (\Omega,p)=
  \frac{\pi}{|\log \Phi_1(a_\ell)|}\zeta_j+o\left(\frac{1}{|\log
      \Phi_1(a_\ell)|}\right)\quad\text{as }\ell\to\infty,
\end{equation*}
where $\{\zeta_j\}_{j=1,\dots,m}$ are the eigenvalues (in ascending order) of
the quadratic form
\begin{equation*}
  \mathcal Q: E(\lambda_n(\Omega_s,\tilde p))\times
  E(\lambda_n(\Omega_s,\tilde p))\to\R,\quad \mathcal Q(w,v)=w(0)v(0).
\end{equation*}
In view of \eqref{eq:w-u} and \eqref{eq:phi-of-a0},
$\{\zeta_j\}_{j=1,\dots,m}$ coincide with the eigenvalues of the
$m\times m$ real symmetric matrix
\begin{equation}\label{eq:matrix}
  \Big(
    u_{n+i-1}(a_0)\cdot u_{n+j-1}(a_0)\Big)_{1\leq
    i,j\leq m}.
\end{equation}
If $u_{n+j-1}(a_0)=0$ for every $1\leq j\leq m$, we have $\zeta_j=0$
for every $1\leq j\leq m$. If
$u_{n+j-1}(a_0)\neq 0$ for some $1\leq j\leq m$, then the matrix
\eqref{eq:matrix} has eigenvalues $0$ (with multiplicity $m-1$) and
$\sum_{i=1}^mu_{n+i-1}^2(a_0)$ (with multiplicity $1$), so that
\begin{equation*}
  \zeta_j=0\quad\text{for every }1\leq j\leq m-1,\quad \zeta_m=
  \sum_{i=1}^mu_{n+i-1}^2(a_0).
\end{equation*}
Therefore, also taking Proposition \ref{p:Phi1asd} into account, we obtain 
\begin{equation*}
  \lambda_{n+m-1}^{a_\ell}(\Omega,p) =\lambda_n (\Omega,p)+
\frac{\pi\big(\sum_{i=1}^mu_{n+i-1}^2(a_0)\big)}{|\log \mathop{\rm
    dist}(a_\ell,\partial\Omega)|}+o\left(\frac{1}{|\log \mathop{\rm
    dist}(a_\ell,\partial\Omega)|}\right) \quad\text{as }\ell\to\infty,
\end{equation*}
and, if $m>1$,
\begin{equation*}
  \lambda_{n+i-1}^{a_\ell}(\Omega,p) =\lambda_n (\Omega,p)+o\left(\frac{1}{|\log \mathop{\rm
    dist}(a_\ell,\partial\Omega)|}\right) \quad\text{as }\ell\to\infty,
\end{equation*}
for every $1\leq i\leq m-1$. The conclusion then follows
from the Urysohn subsequence principle.
\end{proof}

\appendix
\section{} \label{sec_appendix}

In this appendix, we collect some known results used in the paper.

We first recall the well-known \emph{Lemma on small
  eigenvalues} by Y. Colin de Verdi\`ere \cite{ColindeV1986}, see also \cite{Courtois1995} and
\cite{ALM2022}.

\begin{lemma}[Lemma on small eigenvalues \cite{ColindeV1986}]\label{lemma:CdV}
  Let $(H,(\cdot,\cdot)_{H})$ and
  $(\mathcal D,(\cdot,\cdot)_{\mathcal D})$ be real Hilbert spaces
  such that $\dim H=+\infty$, $\mathcal{D}$ is a dense subspace of
  $H$, and the embedding
  $\mathcal D \hookrightarrow \hookrightarrow H$ is continuous and
  compact.  Let $q: \mathcal{D}\times\mathcal{D}\to\R$ be a symmetric
  and continuous bilinear form, such that $q$ is weakly coercive,
  i.e., there exists $A,B>0$ such that
  \begin{equation*}
    q(v,v)+A\|v\|_H^2\geq B\|v\|_{\mathcal D}^2\quad\text{for
      every }v\in \mathcal D,
  \end{equation*}
  where $\|\cdot\|_H=\sqrt{(\cdot,\cdot)_{H}}$ and
  $\|\cdot\|_{\mathcal D}=\sqrt{(\cdot,\cdot)_{\mathcal D}}$.

  Let $\{\nu_j\}_{j\in\N\setminus\{0\}}$ be the eigenvalues of $q$
  (repeated with their multiplicities, in ascending order) and let
  $\{{\mathbf e}_j\}_{j\in\N\setminus\{0\}}$ be an orthonormal basis of $H$ such
  that, for every $j\geq1$, ${\mathbf e}_j\in\mathcal D$ is an
  eigenvector of $q$ associated to the eigenvalue $\nu_j$, i.e.
  $q({\mathbf e}_j,v)=\nu_j({\mathbf e}_i,v)_{H}$ for every
  $v\in \mathcal D$ and $j\geq1$.

Let $m\in \N\setminus\{0\}$ and $F\subset \mathcal{D}$ be
a $m$-dimensional linear subspace of $\mathcal D$. Let
 $\{\xi_j^F\}_{j=1,\dots,m}$ be the eigenvalues (in
          ascending order, repeated with their multiplicities) of $q$
          restricted to $F$. 

Assume that there exist $n\in\N\setminus\{0\}$ and
  $\gamma>0$ such that
	\begin{itemize}
        \item[\rm (H1)] $\nu_j\leq -\gamma$ for all $j\leq n-1$,
          $|\nu_j|\leq \gamma$ for all $j=n,\dots,n+m-1$, and 
          $\nu_j\geq \gamma$ for all $j\geq n+m$;
        \item[\rm (H2)]
          $\delta:=\sup\{|q(v,\varphi)|\colon v\in \mathcal{D},~\varphi\in
          F,\|v\|_H=\|\varphi\|_H=1\}<\gamma/\sqrt{2}$.
	\end{itemize}
        Then
	\begin{equation*}
          \left|\nu_{n+j-1}-\xi_j^F\right|\leq\frac{4\delta^2}{\gamma}
          \quad\text{for all }j=1,\dots,m.
	\end{equation*}
	Moreover, if
        $\Pi:H\to
        \mathop{\rm span}\,\{{\mathbf e}_n,\dots, {\mathbf e}_{n+m-1}\}$ denotes the
        orthogonal projection onto the subspace of
          $\mathcal D$  spanned by $\{{\mathbf e}_n,\dots, {\mathbf e}_{n+m-1}\}$,
         we have
	\begin{equation}\label{eq:cdv-eigenfunctions}
		\frac{\|v-\Pi v\|_{H}}{\|v\|_H}\leq
                \frac{\sqrt{2}\delta}{\gamma}
                \quad\text{for every }v\in F.
	\end{equation}
\end{lemma}

The following lemma states the well known \emph{diamagnetic
  inequality}. We refer to \cite[Theorem 7.21]{LL} for a proof.
\begin{lemma}\label{l:diam}
  If $\Omega\subset\R^2$ is a bounded open set, $a\in\Omega$, and
  $u\in H^{1,a}(\Omega,\C)$, then $|u|\in H^1(\Omega)$ and
\begin{equation*}
|\nabla |u|(x)|\leq
\left|i\nabla u(x)+A_a(x)u(x)\right| \quad \text{for 
  a.e. }  x\in \Omega.
\end{equation*}
\end{lemma}

\section*{Acknowledgements}
\noindent
V.~Felli and G.~Siclari are partially supported by the 2026
INdAM--GNAMPA project ``Asymptotic analysis of variational problems''
(CUP E53C25002010001). P.~Roychowdhury is partially supported by the JSPS
standard postdoctoral research (Fellowship ID P25328).

\bibliographystyle{acm}
\bibliography{references}	
\end{document}